\documentclass[11pt,reqno]{amsart}
\usepackage{amsmath,amsfonts,amssymb,amscd,amsthm,amsbsy,bbm, epsf,calc,  comment, caption, enumerate}
\usepackage{color}
\usepackage[usenames,dvipsnames]{xcolor}
\usepackage{datetime}
\usepackage{soul}

\usepackage{thmtools, thm-restate}
\usepackage{thm-patch, thm-kv, thm-autoref}
\usepackage[colorlinks=true, urlcolor=blue, linkcolor=blue, citecolor=black]{hyperref}

\usepackage[pdftex]{graphicx}
\usepackage{wrapfig}
\usepackage{geometry}
\numberwithin{equation}{section}
\newcounter{hours}\newcounter{minutes}

\theoremstyle{plain}
\declaretheorem[title=Theorem, parent=section]{theorem}
\declaretheorem[title=Lemma,sibling=theorem]{lemma}
\declaretheorem[title=Corollary,sibling=theorem]{corollary}
\declaretheorem[title=Proposition,sibling=theorem]{proposition}
\declaretheorem[title=Definition,sibling=theorem]{definition}
\declaretheorem[title=Remark,sibling=theorem]{rem}

\newtheorem*{prop*}{Proposition}

\makeatletter
\newcommand\RedeclareMathOperator{%
  \@ifstar{\def\rmo@s{m}\rmo@redeclare}{\def\rmo@s{o}\rmo@redeclare}%
}
\newcommand\rmo@redeclare[2]{%
  \begingroup \escapechar\m@ne\xdef\@gtempa{{\string#1}}\endgroup
  \expandafter\@ifundefined\@gtempa
     {\@latex@error{\noexpand#1undefined}\@ehc}%
     \relax
  \expandafter\rmo@declmathop\rmo@s{#1}{#2}}
\newcommand\rmo@declmathop[3]{%
  \DeclareRobustCommand{#2}{\qopname\newmcodes@#1{#3}}%
}
\@onlypreamble\RedeclareMathOperator
\makeatother

\def\e{\varepsilon}
\def\ep{\varepsilon}
\def\al{\alpha}
\def\del{\delta}

\def\om{\omega}
\def\O{\Omega}
\def\Om{\Omega}
\def\gam{\gamma} 
\def\Gam{\Gamma}
\def\lam{\lambda}
\def\Lam{\Lambda}

\def\grad{\nabla}
\def\integer{\mathbb Z}
\def\Natural{\mathbb N}
\def\real{\mathbb R}
\def\A{\mathcal A}
\def\D{\mathcal D}
\def\K{\mathcal K}
\def\L{\mathcal L}
\def\M{\mathcal M}

\def\R{\mathcal R}

\def\diam{\textnormal{diam}}

\def\Id{\textnormal{Id}}

\def\intersect{\cap}

\def\Indicator{{\mathbbm{1}}}
\def\graph{\textnormal{graph}}

\def\supp{\textnormal{supp}}

\def\Tr{\textnormal{tr}}
\def\Id{\textnormal{Id}}
\def\dini{\textnormal{Dini}}

\def\O{\Omega}
\def\Rn{\mathbb{R}^{n+1}}
\def\RN{\mathbb{R}^n}
\def\la{\left\langle}
\def\ra{\right\rangle}

\RedeclareMathOperator{\div}{\textnormal{div}}

\def\polhk#1{\setbox0=\hbox{#1}{\ooalign{\hidewidth
	    \lower1.5ex\hbox{`}\hidewidth\crcr\unhbox0}}}

\newcommand{\abs}[1]{\left| #1 \right|}

\newcommand{\norm}[1]{\lVert#1\rVert}

\begin{document}

\title{Regularity for a special case of two-phase Hele-Shaw flow via parabolic integro-differential equations}
%Arxiv version 1

\author{Farhan Abedin}
\author{Russell W. Schwab}

\address{Department of Mathematics\\
Michigan State University\\
619 Red Cedar Road \\
East Lansing, MI 48824}
\email{rschwab@math.msu.edu, abedinf1@msu.edu}

\begin{abstract}
  We establish that the $C^{1,\gam}$ regularity theory for translation invariant fractional order parabolic integro-differential equations (via Krylov-Safonov estimates) gives an improvement of regularity mechanism for solutions to a special case of a two-phase free boundary flow related to Hele-Shaw.  The special case is due to both a graph assumption on the free boundary of the flow and an assumption that the free boundary is $C^{1,\dini}$ in space.  The free boundary then must immediately become $C^{1,\gam}$ for a universal $\gam$ depending upon the Dini modulus of the gradient of the graph. These results also apply to one-phase problems of the same type.
 
\end{abstract}

\date{\today,\  arXiv ver 2}

\thanks{R. Schwab acknowledges partial support from the NSF with DMS-1665285. F. Abedin acknowledges support from the AMS and the Simons Foundation with an AMS--Simons Travel Grant. }
\keywords{Global Comparison Property, Integro-differential Operators, Dirichlet-to-Neumann, Free Boundaries, Hele-Shaw, Fully Nonlinear Equations, Viscosity Solutiuons, Krylov-Safonov}
\subjclass[2010]{
%updated and confirmed 3/24/2016
%35J99,      %pde other
35B51, %Comparison principles
35R09,  	%Integro-partial differential equations
35R35, %Free boundary problems
45K05,  	%Integro-partial differential equations
%46T99,   	%None of the above, but in this section
47G20,      %integro-differential operators
49L25,  	%Optimal Control Viscosity solutions
%49N70,  	%Differential games
60J75,      %jump processes
76D27, %	Other free-boundary flows; Hele-Shaw flows
76S05 %flows in porous media
%93E20       %optimal stoch. control
}

\maketitle

\markboth{Hele-Shaw Parabolic Regularization}{Hele-Shaw Parabolic Regularization}
%%%%%%%%%%%%%%%%%%%%%%%%%%%%%%%%%%%%%%%%%%%%%%
%%%%%%%%%%%%%%%%%%%%%%%%%%%%%%%%%%%%%%%%%%%%%%

%%%%%%%%%%%%%%%%%%%%%%%%%%%%%%%%%%%%%%%%%%%%%%%%%
%%%%%%%%%%%%%%%%%%%%%%%%%%%%%%%%%%%%%%%%%%%%%%%%%
%%%%%%%%%%%%%%%%%%%%%%%%%%%%%%%%%%%%%%%%%%%%%%%%%
%%%%%%%%%%%%%%%%%%%%%%%%%%%%%%%%%%%%%%%%%%%%%%%%%
%%%%%%%%%%%%%%%%%%%%%%%%%%%%%%%%%%%%%%%%%%%%%%%%%
%%%%%%%%%%%%%%%%%%%%%%%%%%%%%%%%%%%%%%%%%%%%%%%%%
%%%%%%%%%%%%%%%%%%%%%%%%%%%%%%%%%%%%%%%%%%%%%%%%%
%%%%%%%%%%%%%%%%%%%%%%%%%%%%%%%%%%%%%%%%%%%%%%%%%
%%%%%%%%%%%%%%%%%%%%%%%%%%%%%%%%%%%%%%%%%%%%%%%%%

\section{Introduction}\label{sec:introduction}
\setcounter{equation}{0}

This paper has two goals.  The first is to give a precise characterization of the integro-differential operators that can be used to represent the solution of some free boundary flows with both one and two phases, of what we call Hele-Shaw type.  We give a characterization that is precise enough to determine whether or not existing integro-differential results apply to this setting.  The second goal is to show that, indeed, a new regularization mechanism resulting from parabolic integro-differential theory is applicable.  This will show that solutions that are $C^{1,\dini}$ must immediately become $C^{1,\gam}$ regular.  We note that there is an earlier and stronger regularization mechanism for the one-phase Hele-Shaw flow by Choi-Jerison-Kim \cite{ChoiJerisonKim-2007RegHSLipInitialAJM} which shows that Lipschitz solutions with a dimensionally small Lipschitz norm must be $C^{1}$ regular and hence classical.   We want to emphasize that in our context, both one and two phase problems are treated under the exact same methods.  For simplicity and technical reasons, we focus on the case in which the free boundary is the graph of a time dependent function on $\real^n$, $n\geq 2$.

These free boundary problems are the time dependent evolution of the zero level set of a function $U: \real^{n+1} \times [0,T] \to \real$ that satisfies the following equation, with $V$ representing the normal velocity on $\partial\{U(\cdot,t)>0\}$, and $G$ a prescribed balance law.  Here $A_1$ and $A_2$ are two (possibly different) elliptic constant coefficient diffusion matrices that dictate the equations:
\begin{align}\label{eqIN:HSMain}
	\begin{cases}
	\Tr(A_1 D^2U)=0\ &\text{in}\ \{U(\cdot,t)>0\}\\
	\Tr(A_2 D^2U)=0\ &\text{in}\ \{U(\cdot,t)<0\}\\
	U(\cdot,t)=1\ &\text{on}\ \{x_{n+1}=0\}\\
	U(\cdot,t)=-1\ &\text{on}\ \{x_{n+1}=L\}\\
	V=G(\partial^+_\nu U,\partial^-_\nu U)\ &\text{on}\ \partial\{U(\cdot,t)>0\}.
	\end{cases}
\end{align}
Without loss of generality, we take $A_1=\Id$ (which can be obtained by an orthogonal change of coordinates).  The prescribed values for $U$ at $x_{n+1}=0$ and $x_{n+1}=L$ can be thought of as an ambient background pressure for $U$, and the free boundary, $\{U=0\}$, will be located in between.

As mentioned above, this work treats the special case of the free boundary problem in which the boundary of the positivity set can be given as the graph of a function over $\real^n$.  To this end, 
we will use the notation, $D_f$, as
\begin{align*}
	D_f=\{(x,x_{n+1})\in\real^{n+1}\ :\ 0<x_{n+1}<f(x)\},
\end{align*}
and in our context, we will assume that for some $f:\real^n\times[0,T]\to\real$,
\begin{align*}
	\{ U(\cdot,t)>0 \} = D_{f(\cdot,t)}
\end{align*}
and
\begin{align*}
	\partial\{ U(\cdot,t)>0\} = \graph(f(\cdot, t)).
\end{align*}

The main technical part of our work is centered on the properties of the (fully nonlinear) operator we call $I$, which is defined for the one-phase problem as
\begin{align}\label{eqIN:BulkEqForHSOperator}
	\begin{cases}
		\Delta U_f=0\ &\text{in}\ D_f\\
		U_f=1\ &\text{on}\ \real^n\times\{0\}\\
		U_f=0\ &\text{on}\ \Gam_f=\graph(f),
	\end{cases}
\end{align}
and $I$ is the map,
\begin{align}\label{eqIN:defHSOperator}
	I(f,x) = \partial_\nu U_f(x,f(x)).
\end{align}
We note, the map $I$ does not depend on $t$ and it is a fully nonlinear function of $f$ (in the sense that it does not have a divergence structure, and it fails linearity in the highest order terms acting on $f$ -- in fact it is fails linearity for all terms).  Here, $I$, can be thought of as a nonlinear Dirichlet-to-Neumann operator, but one that tracks how a particular solution depends on the boundary.  This type of operator is not at all new, and we will briefly comment on its rather long history later on, in Section \ref{sec:BackgroundLiterature}.

It turns out (probably not surprisingly) that the key features of (\ref{eqIN:HSMain}) are entirely determined by the properties of the mapping, $I$.  To this end, we will define a two phase version of this operator via the positive and negative sets,
\begin{align}\label{eqIN:DefOfSetsDfPlusMinus}
	&D_f^+ = \{ (x,x_{n+1}) \ :\ 0<x_{n+1}<f(x) \},\\
	& D_f^- = \{  (x,x_{n+1})\ :\ f(x)<x_{n+1}<L  \},
\end{align}
with the equation, (recall we take $A_1=\Id$)
\begin{align}\label{eqIN:TwoPhaseBulk}
	\begin{cases}
		\Delta U_f = 0\ &\text{in}\ D_f^+\\
		\Tr(A_2 D^2 U_f) = 0\ &\text{in}\ D_f^-\\
		U_f = 0\ &\text{on}\ \Gam_f\\
		U_f=1\ &\text{on}\ \{x_{n+1}=0\}\\
		U_f=-1\ &\text{on}\ \{x_{n+1}=L\}.
	\end{cases}
\end{align}
We define the respective normal derivatives to the positive and negative sets:
\begin{align}
	&\text{for}\ X_0\in\Gamma_f,\ \text{and}\ \nu(X_0)\ \text{the unit normal derivative to $\Gam_f$, pointing into the set}\ D^+f,\nonumber\\ 
	&\partial^+_\nu U(X_0):=\lim_{t\to0}\frac{U(X_0+t\nu(X_0))-U(X_0)}{t}\ \ \text{and}\ \  \partial^-_\nu U(X_0)=-\lim_{t\to0}\frac{U(X_0-t\nu(X_0))-U(X_0)}{t}.\label{eqIN:DefOfPosNegNormalDeriv}
\end{align}
With these, we can define the operator, $H$, as 
\begin{align}\label{eqIN:DefOfH}
	H(f,x):= G(I^+(f,x),I^-(f,x))\cdot \sqrt{1+\abs{\grad f}^2},
\end{align}
where
\begin{align}\label{eqIN:DefIPLusAndMinus}
	I^+(f,x):= \partial_\nu^+ U_f(x,f(x)),\ \ \text{and}\ \ 
	I^-(f,x):= \partial_\nu^- U_f(x,f(x)).
\end{align}
The standard ellipticity assumption on $G$ is the following:
\begin{align}\label{eqIN:GEllipticity}
	G\ \text{is Lipschitz and}\ \ \lam\leq \frac{\partial}{\partial a} G(a,b)\leq \Lam,\ \ 
	\lam\leq -\frac{\partial}{\partial b} G(a,b)\leq \Lam.
\end{align}
A canonical example of $G$ for the two-phase problem is $G(a,b)=a-b$, whereas a one-phase problem will simply be given by $G(a,b)=\tilde G(a)$, and the problem often referred to as one-phase Hele-Shaw flow is $G(a,b)=a$ (we note that the name ``Hele-Shaw'' has multiple meanings, depending upon the literature involved; both instances can be seen in Saffman-Taylor \cite{SaffamnTaylor-1958PorousMedAndHeleShaw}).

In a previous work, \cite{ChangLaraGuillenSchwab-2019SomeFBAsNonlocalParaboic-NonlinAnal}, it was shown that under the graph assumption, the flow (\ref{eqIN:HSMain}) is equivalent in the sense of viscosity solutions for free boundary problems to viscosity solutions of the nonlinear, nonlocal, parabolic equation for $f$
\begin{align}\label{eqIN:HeleShawIntDiffParabolic}
	\begin{cases}
		\partial_t f = G(I^+(f), I^-(f))\cdot\sqrt{1+\abs{\grad f}^2}\ &\text{in}\ \real^n\times [0,T],\\
		f(\cdot,0) = f_0\ &\text{on}\ \real^n\times\{0\}.
	\end{cases}
\end{align}
We remark that a viscosity solution for the respective equations (\ref{eqIN:HSMain}) and (\ref{eqIN:HeleShawIntDiffParabolic}) (they are different objects) will exist whenever the free boundary (or in this case, $f$) is uniformly continuous, i.e. in very low regularity conditions.

In this paper, we explore a higher regularity regime, already assuming the existence of a classical solution of (\ref{eqIN:HSMain}). Whenever $f$ remains in a particular convex set of $C^{1,\dini}$ (the set of $C^1$ functions whose gradients enjoy a Dini modulus), we will show that the operator $H$ takes a precise form as an integro-differential operator.  This convex set is denoted as, $\K(\del,L,m,\rho)$, and is made up as

\begin{align*}
	C^{1,\dini}_\rho(\real^n) = \{ f:\real^n\to\real\ |\ \grad f\in L^\infty\  \text{and is Dini continuous with modulus}\ \rho  \},
\end{align*}
\begin{align}\label{eqIN:DefOfSetK}
	\K(\del,L,m,\rho) = \{ f\in C^{1,\dini}_\rho\ :\ \del<f<L-\del,\ \abs{\grad f}\leq m \}.
\end{align}
We note that the extra requirement $\del<f<L-\del$ is simply that the free boundary remains away from the fixed boundary where the pressure conditions are imposed.

The first theorem gives the integro-differential structure of $H$, and the details of which ellipticity class it falls into.

\begin{theorem}\label{thm:StructureOfHMain}
	Assume that $G$ satisfies (\ref{eqIN:GEllipticity}) and $H$ is the operator defined by (\ref{eqIN:DefOfH}), using the equation, (\ref{eqIN:TwoPhaseBulk}).
	
	\begin{enumerate}[(i)]
		\item
	For each fixed $\del$, $L$, $m$, $\rho$, that define the set $\K$ in (\ref{eqIN:DefOfSetK}) there exists a collection $\{a^{ij}, c^{ij}, b^{ij}, K^{ij}\}\subset{\real\times\real\times\real^n\times \textnormal{Borel}(\real^n\setminus\{0\})}$ (depending upon $\del$, $L$, $m$, $\rho$), so that
	\begin{align*}
		\forall\ f\in \K(\del,L,m,\rho),\ \ \  
		H(f,x) = \min_i\max_j\left( 
		a^{ij}+c^{ij}f(x) + b^{ij}\cdot\grad f(x) + \int_{\real^n}\del_y f(x)K^{ij}(y)dy
		\right),
	\end{align*}
	where for an $r_0$ depending upon $\del$, $L$, $m$, we use the notation,
	\begin{align}\label{eqIN:DelhFNotation}
		\del_yf(x)= f(x+y)-f(x)-\Indicator_{B_{r_0}}(y)\grad f(x)\cdot y.
	\end{align}
	
	\item
	Furthermore, there exists $R_0$ and $C$, depending on $\del$, $L$, $m$, $\rho$, so that for all $i,j$,
	\begin{align*}
		\forall\ y\in\real^n,\ \ \ 
		C^{-1}\abs{y}^{-n-1}\Indicator_{B_{R_0}}(y)\leq K^{ij}(y)\leq C\abs{y}^{-n-1}, 
	\end{align*}
	and
	\begin{align*}
		\sup_{0<r<r_0}\abs{b^{ij} - \int_{B_{r_0}\setminus B_r}yK^{ij}(y)dy}\leq C.
	\end{align*}
	The value of $r_0$ in (\ref{eqIN:DelhFNotation}) depends on $R_0$.

	\end{enumerate}
	
\end{theorem}

The second result of this paper is to use the above result, plus recent results for parabolic integro-differential equations that include (\ref{eqIN:HeleShawIntDiffParabolic}), thanks to part (ii) of Theorem \ref{thm:StructureOfHMain}, to deduce regularity for the resulting free boundary (in this case, the set $\Gam_f=\graph(f(\cdot,t))$).  This is the content of our second main result.

\begin{theorem}\label{thm:FBRegularity}
	There exist universal constants, $C>0$ and $\gam\in(0,1)$, depending upon $\del$, $L$, $m$, and $\rho$, which define $\K$ in (\ref{eqIN:DefOfSetK}) so that 
	if $f$ solves (\ref{eqIN:HeleShawIntDiffParabolic}) and for all $t\in[0,T]$, $f(\cdot,t)\in\K(\del,L, m, \rho)$, then $f\in C^{1,\gam}(\real^n\times[\frac{T}{2},T])$, and
	\begin{align*}
		\norm{f}_{C^{1,\gam}(\real^n\times[\frac{T}{2},T])}\leq \frac{C(\del,L,m,\rho)(1 + T)}{T^\gam}\norm{f(\cdot,0)}_{C^{0,1}}.
	\end{align*}
	
	In particular, under the \emph{assumption} that for all $t\in[0,T]$, $\partial\{U(\cdot,t)>0\}=\graph(f(\cdot,t))$, \emph{and} for all $t\in[0,T]$, $f\in\K(\del,L,m,\rho)$, we conclude that $\partial\{U>0\}$ is a $C^{1,\gam}$ hypersurface in space and time. 
		
\end{theorem}

\begin{rem}
	It is important to note the strange presentation of the $C^{1,\gam}$ estimate in Theorem \ref{thm:FBRegularity} with only $\norm{f(\cdot,0)}_{C^{0,1}}$ on the right hand side.  We emphasize that we have \emph{not} proved that Lipschitz free boundaries become $C^{1,\gam}$, due to the constant, $C(\del,L,m,\rho)$.  As the reader will see in Section \ref{sec:KrylovSafonovForHS}, the constant $C$ depends in a complicated way on the parameters, $\del$, $L$, $m$, $\rho$, as all of these impact the boundary behavior of the Green's function for the elliptic equations in (\ref{eqIN:HSMain}), which in turn changes the estimates in Theorem \ref{thm:StructureOfHMain}, and hence the resulting parabolic estimates in Section \ref{sec:KrylovSafonovForHS}.  Nevertheless, once one knows that $f\in\K(\del,L,m,\rho)$ for some fixed choice of $\del$, $L$, $m$, $\rho$, subsequently  decreasing the Lipschitz norm of $f(\cdot,0)$ would decrease the $C^{1,\gam}$ norm of the solution at later times.  However, since the parameters $\del$, $L$, $m$, $\rho$ give an upper bound on the quantity, for each $t$, $\norm{f(\cdot,t)}_{C^{1,\dini}_\rho}$, a reasonable interpretation of the result is rather given as
	\begin{align*}
		\norm{f}_{C^{1,\gam}(\real^n\times[\frac{T}{2},T])}\leq \frac{C(\del,L,m,\rho)(1+T)}{T^\gam}\sup_{t\in[0,T]}\norm{f(\cdot,t)}_{C^{1,\dini}_\rho}.
	\end{align*}
\end{rem}

	We note that the work \cite{ChangLaraGuillenSchwab-2019SomeFBAsNonlocalParaboic-NonlinAnal} established the equivalence between free boundary viscosity solutions of some Hele-Shaw type evolutions, like (\ref{eqIN:HSMain}), and viscosity solutions of fractional integro-differential parabolic equations in (\ref{eqIN:HeleShawIntDiffParabolic}).  However, the results in \cite{ChangLaraGuillenSchwab-2019SomeFBAsNonlocalParaboic-NonlinAnal} focused on this equivalence at the level of viscosity solutions and low regularity properties, and they stopped short of addressing the question of a \emph{regularization} phenomenon that may occur in a slightly higher regularity regime.  As shown in the current paper, one needs to obtain much more precise information about the integro-differential operators appearing in, for instance, Theorem \ref{thm:StructureOfHMain} in order to utilize recent tools from the realm of integro-differential equations to investigate how this equation regularizes.  Furthermore, obtaining the estimate as in Theorem \ref{thm:StructureOfHMain} required a slightly different approach than the one pursued in \cite{ChangLaraGuillenSchwab-2019SomeFBAsNonlocalParaboic-NonlinAnal}, instead invoking a finite dimensional approximation technique from \cite{GuSc-2019MinMaxEuclideanNATMA}.  This can be seen in Sections \ref{sec:LevyMeasureEstimate} and \ref{sec:DriftEstimates}.

%%%%%%%%%%%%%%%%%%%%%%%%%%%%%%%%%%%%%%%%%%%%%%%%%
%%%%%%%%%%%%%%%%%%%%%%%%%%%%%%%%%%%%%%%%%%%%%%%%%
%%%%%%%%%%%%%%%%%%%%%%%%%%%%%%%%%%%%%%%%%%%%%%%%%
%%%%%%%%%%%%%%%%%%%%%%%%%%%%%%%%%%%%%%%%%%%%%%%%%
%%%%%%%%%%%%%%%%%%%%%%%%%%%%%%%%%%%%%%%%%%%%%%%%%
%%%%%%%%%%%%%%%%%%%%%%%%%%%%%%%%%%%%%%%%%%%%%%%%%
%%%%%%%%%%%%%%%%%%%%%%%%%%%%%%%%%%%%%%%%%%%%%%%%%
%%%%%%%%%%%%%%%%%%%%%%%%%%%%%%%%%%%%%%%%%%%%%%%%%
%%%%%%%%%%%%%%%%%%%%%%%%%%%%%%%%%%%%%%%%%%%%%%%%%

\section{Some Historical Background and Related Results}\label{sec:BackgroundLiterature}

Basically, (\ref{eqIN:HSMain}) is a two-phase Hele-Shaw type problem without surface tension and neglecting the effects of gravity.  For our purposes, we are interested in (\ref{eqIN:HSMain}) for mathematical reasons to uncover some of its structural properties and to explore the possibility of regularizing effects.  Thus, we do not comment much on the model's physical origins.  The fact that (\ref{eqIN:HSMain}) governs a two-phase situation is important for us to demonstrate that these techniques work for both one and two phase problems of a certain type.

In the following discussion, we attempt to focus on results most closely related to (\ref{eqIN:HSMain}), and we note that a more extended discussion can be found in the works \cite{ChangLaraGuillen-2016FreeBondaryHeleShawNonlocalEqsArXiv} and \cite{ChangLaraGuillenSchwab-2019SomeFBAsNonlocalParaboic-NonlinAnal}.

%%%%%%%%%%%%%%%%%%%%%%%%%%%%%%%%%%%%%%%%%%%%%%%%%
%%%%%%%%%%%%%%%%%%%%%%%%%%%%%%%%%%%%%%%%%%%%%%%%%
%%%%%%%%%%%%%%%%%%%%%%%%%%%%%%%%%%%%%%%%%%%%%%%%%
%%%%%%%%%%%%%%%%%%%%%%%%%%%%%%%%%%%%%%%%%%%%%%%%%
\subsection{Hele-Shaw type free boundary problems without gravity.}

In most of the existing literature, (\ref{eqIN:HSMain}) is studied in its one-phase form, where the set $\{U<0\}$ is ignored by simply dictating that the velocity condition is $V=G(\partial^+_\nu U^+)$.

Some of the earliest works for short time existence and uniqueness are \cite{ElliottJanovsky-1981VariationalApproachHeleShaw} and \cite{EscherSimonett-1997ClassicalSolutionsHeleShaw-SIAM}, where a type of variational problem is studied in \cite{ElliottJanovsky-1981VariationalApproachHeleShaw} and a classical solution (for short time) is produced in \cite{EscherSimonett-1997ClassicalSolutionsHeleShaw-SIAM}.  For the one-phase problem, under a smoothness and convexity assumption, \cite{DaskaLee-2005AllTimeSmoothSolHeleShawStefan-CPDE} gives global in time smooth solutions.  Viscosity solutions for the one-phase version of (\ref{eqIN:HSMain}) are defined and shown to exist and be unique in \cite{Kim-2003UniquenessAndExistenceHeleShawStefanARMA}, which follows the approach first developed in \cite{Caffarelli-1988HarnackApproachFBPart3PISA} for the stationary two-phase problem and subsequently used in \cite{AthanaCaffarelliSalsa-1996RegFBParabolicPhaseTransitionACTA} for the two-phase Stefan problem.  A follow-up modification of the definition of viscosity solutions for (\ref{eqIN:HSMain}) was given in \cite[Section 9]{ChangLaraGuillenSchwab-2019SomeFBAsNonlocalParaboic-NonlinAnal}.  Of course, for our results, we are assuming already the existence of a classical solution, and so none of the definitions of viscosity solutions for (\ref{eqIN:HSMain}) are invoked here.  (However, we do invoke viscosity solutions for the function $f$, as they are useful even when studying smooth solutions, such as in investigating the equation for discrete spatial derivatives of solutions.  But the notion of solution for $f$ is entirely different from that of $U_f$.)

Moving on to issues of regularity, beyond the smooth initial data case in \cite{EscherSimonett-1997ClassicalSolutionsHeleShaw-SIAM}, and the convex case in \cite{DaskaLee-2005AllTimeSmoothSolHeleShawStefan-CPDE}, there are a number of works.  All of the following works apply to the one-phase problem.  With some assumptions on the quantity $\abs{U_t}/\abs{DU}$, \cite{Kim-2006RegularityFBOnePhaseHeleShaw-JDE} showed a Lipschitz free boundary becomes $C^1$ in space-time with a modulus, and long time regularity, involving propagation of a Lipschitz modulus, was obtained in \cite{Kim-2006LongTimeRegularitySolutionsHeleShaw-NonlinAnalysis}.  Subsequently, the extra condition on the space-time non-degeneracy in \cite{Kim-2006RegularityFBOnePhaseHeleShaw-JDE} was removed in the work of \cite{ChoiJerisonKim-2007RegHSLipInitialAJM}, where under a dimensional small Lipschitz condition on the initial free boundary, Lipschitz free boundaries must be $C^1$ in space-time and hence classical.  This was then followed up by the work \cite{ChoiJerisonKim-2009LocalRegularizationOnePhaseINDIANA} where more precise results can be proved when the solution starts from a global Lipschitz graph.   In this context, it is fair to say that our results are the extension of \cite{ChoiJerisonKim-2009LocalRegularizationOnePhaseINDIANA} to the two-phase case, but with paying the extra price of requiring $C^{1,\dini}$ regularity of the initial graph instead of being only Lipschitz.  There is another regularity result for the one-phase Hele-Shaw problem in \cite{ChangLaraGuillen-2016FreeBondaryHeleShawNonlocalEqsArXiv} that follows more the strategy of \cite{DeSilva-2011FBProblemWithRHS} and \cite{Savin-2007SmallPerturbationCPDE},  instead of \cite{Caffarelli-1987HarnackInqualityApproachFBPart1RevMatIbero}, \cite{Caffarelli-1989HarnackForFBFlatAreLipCPAM},  \cite{ChoiJerisonKim-2007RegHSLipInitialAJM}, \cite{Kim-2006RegularityFBOnePhaseHeleShaw-JDE}.  In \cite{ChangLaraGuillen-2016FreeBondaryHeleShawNonlocalEqsArXiv} the approach to regularity for the one-phase Hele-Shaw invoked parabolic regularity theory for fractional equations, but in that context the regularity theory applied to a blow-up limit of the solutions under a flatness condition in space-time, which resulted in a local $C^{1,\gam}$ space-time regularity for the solution.   Thus already \cite{ChangLaraGuillen-2016FreeBondaryHeleShawNonlocalEqsArXiv} foreshadowed the type of strategy that we have pursued in Theorem \ref{thm:FBRegularity}.

%%%%%%%%%%%%%%%%%%%%%%%%%%%%%%%%%%%%%%%%%%%%%%%%%
%%%%%%%%%%%%%%%%%%%%%%%%%%%%%%%%%%%%%%%%%%%%%%%%%
%%%%%%%%%%%%%%%%%%%%%%%%%%%%%%%%%%%%%%%%%%%%%%%%%
%%%%%%%%%%%%%%%%%%%%%%%%%%%%%%%%%%%%%%%%%%%%%%%%%
\subsection{The nonlinear Dirichlet-to-Neumann mapping}

In this paper, it is reasonable to call the operator, $I$, defined in (\ref{eqIN:BulkEqForHSOperator}) and (\ref{eqIN:defHSOperator}) a nonlinear version of the classical Dirichlet-to-Neumann mapping.  In this case it records the dependence on the shape of the domain of a particular harmonic function.  This operator $I$, and the resulting mapping defined as $H$ in (\ref{eqIN:DefOfH}) are  key components in our analysis, as well as were one of the main ingredients in the previous work \cite{ChangLaraGuillenSchwab-2019SomeFBAsNonlocalParaboic-NonlinAnal}.  Such operators are not new, and they have a relatively long study, particularly in some water wave equations (in fact, the authors in \cite{ChangLaraGuillenSchwab-2019SomeFBAsNonlocalParaboic-NonlinAnal} were unaware of this long history).  Although the map, $I$, appearing in (\ref{eqIN:defHSOperator}), is not exactly the operator appearing in earlier works, it is very similar.  Most of the earlier versions are a slight variant on the following: given two functions, $h:\real^n\to\real$ and $\psi:\real^n\to\real$, $U_{h,\psi}$ is the unique, bounded, harmonic function,
\begin{align*}
	\begin{cases}
		\Delta U_{h,\psi}=0\ &\text{in}\ \{(x,x_{n+1})\in\real^{n+1}\ :\ x_{n+1}<h(x)\}\\
		U_{h,\psi}(x,x_{n+1}) = \psi(x)\ &\text{on}\ \graph(h)
	\end{cases}
\end{align*}
and the Dirichlet-to-Nuemann operator is
\begin{align*}
	[\tilde G(h)\psi](x):= \partial_\nu U_{h,\psi}(x,h(x))\sqrt{1+\abs{\grad h(x)}^2}.
\end{align*}
We note that this operator is in fact in the literature usually denoted as $G(h)\psi$, but we use $\tilde G(h)\psi$ due to the conflicting notation with our use of ``$G$'' in (\ref{eqIN:HSMain}), which is entirely different.  The reader should note that in this context, it is very frequent that $\psi$ actually does not depend on $x_{n+1}$, which can be justified in that $\tilde G(h)$ often is used when acting on such vertically constant boundary data. Sometimes instead of taking $U$ to be defined in the subgraph of $h$, there may be other boundary conditions, such as, for example when $h>1$, a no flux condition $\partial_\nu U_{h,\psi}|_{\{x_{n+1}=0\}}=0$, or even there could be a fixed bottom boundary with a nontrivial shape.  For the purposes of discussion, the equation in the subgraph of $h$ will suffice.  The use of the map, $\tilde G(h)$, appears to go back to \cite{Zakharov-1968StabilityPeriodicWaves} and then \cite{CraigSulem-1993NumericalSimulationGravityWaves-JCompPhys}.  The operator, $\tilde G(h)$ was revisited in \cite{NichollsReitich-2001NewApproachDtoNAnalyticity} for the sake of improving computational tractability for various problems like (\ref{eqIN:HSMain}) that may involve interfaces moving via a normal derivative.  The work \cite{Lannes-2005WellPoseWaterWave-JAMS} investigates the mapping and boundedness properties of $\tilde G(h)$ on various Sobolev spaces for proving well-posedness of water wave equations, and also gives a very detailed description of the usage of $\tilde G$ in earlier works on water waves; we refer to \cite{Lannes-2005WellPoseWaterWave-JAMS} for more discussion on the history of $\tilde G$ in water wave results.  The subsequent article \cite{AlazardBurqZuily-2014CauchyProblemGravityWaterWaves-Inventiones} showed a more careful analysis of $\tilde G$ could give improved conditions on well-posedness in gravity water waves.  $\tilde G$ recently played a central role in \cite{NguyenPausader-2020ParadifferentialWellPoseMuskatARMA} for well-posedness of the Muskat problem and in  \cite{Alazard-2020ConvexityAndHeleShaw-ArXiv}, \cite{AlazardMeunierSmets-2019LyapounovCauchyProbHeleShaw-ArXiv} for well-posedness of the one-phase Hele-Shaw equation with gravity as well as to deduce results related to Lyapunov functionals for the solution.

%%%%%%%%%%%%%%%%%%%%%%%%%%%%%%%%%%%%%%%%%%%%%%%%%
%%%%%%%%%%%%%%%%%%%%%%%%%%%%%%%%%%%%%%%%%%%%%%%%%
%%%%%%%%%%%%%%%%%%%%%%%%%%%%%%%%%%%%%%%%%%%%%%%%%
%%%%%%%%%%%%%%%%%%%%%%%%%%%%%%%%%%%%%%%%%%%%%%%%%
\subsection{Hele-Shaw type free boundary problems with gravity-- Muskat type problems}

A pair of free boundary problems that is closely related to (\ref{eqIN:HSMain}), but pose their own set of additional challenges are those that are also called Hele-Shaw and Muskat problems.  They can be cast as both one and two phase problems, and they govern the free surface between two fluids of different density and possibly different viscosity.  We note that in both, gravity is taken into account, and this changes the nature of the equation a bit away from (\ref{eqIN:HSMain}); also the pressure is not required to be constant along the free boundary.  There is a large amount of literature on this class of problems, and we focus on the ones most closely related to (\ref{eqIN:HSMain}).  A feature that links the Muskat problem to that considered in this paper is to rewrite the solution for the original problem in $n+1$ space dimensions as a problem in $n$ dimensions that governs the free surface itself, directly, via a nonlinear equation that is inherently integro-differential in nature and which linearizes to the fractional heat equation of order $1/2$.  Already the reformulation of the problem in terms of integro-differential equations goes back to \cite{Ambrose-2004WellPoseHeleShawWithoutSurfaceTensionEurJAppMath},  \cite{CaflischOrellanaSiegel-1990LocalizedApproxMethodVorticalFLows-SIMA}, \cite{CaflischHowisonSiegel-2004GlobalExistenceSingularSolMuskatProblem-CPAM}, with gloabl existence of solutions with small data in \cite{CaflischHowisonSiegel-2004GlobalExistenceSingularSolMuskatProblem-CPAM} and short time existence of solutions with large data in an appropriate Sobolev space in \cite{Ambrose-2004WellPoseHeleShawWithoutSurfaceTensionEurJAppMath}.  This method of writing the Muskat problem as an equation for the free surface directly continues in \cite{CordobaGancedo-2007ContourDynamics-CMP}, and this is an integro-differential type equation \emph{for the gradient} of the free surface function, where for a 2-dimensional interface is
\begin{align}\label{eqLIT:MuskatFPrime}
	\partial_t f = \frac{\rho_2-\rho_1}{4\pi}\int_{\real^2}\frac{(\grad f(x,t)-\grad f(x-y,t))\cdot y}{(y^2+[f(x,t)-f(x-y,t)]^2)^{3/2}}dy.
\end{align}
This formulation was then used to show that near a stable solution that is sufficiently regular, the equation linearizes to the 1/2-heat equation, and \cite{CordobaGancedo-2007ContourDynamics-CMP} further showed existence of solutions in this region (see a few more comments about linearization in Section \ref{sec:Commentary}). It was subsequently used to produce many well-posedness and regularity results, both short time and global time, a few of which are:  \cite{ConstantinCordobaGancedoStrain-2013GlobalExistenceMuskat-JEMS}, \cite{ConstantinGancedoShvydkoyVicol-2017Global2DMuskatRegularity-AIHP}, \cite{CordobaCordobaGancedo-2011InterfaceHeleShawMuskat-AnnalsMath}.

There are (at least) two other variants on studying the Muskat problem as an equation for the free surface alone, and the ones that are very close in spirit to our work are, on the one hand,  \cite{Cameron-2019WellPosedTwoDimMuskat-APDE}, \cite{Cameron-2020WellPose3DMuskatMediumSlope-ArXiv}, \cite{CordobaGancedo-2009MaxPrincipleForMuskat-CMP}, and on the other hand, \cite{AlazardMeunierSmets-2019LyapounovCauchyProbHeleShaw-ArXiv}, \cite{NguyenPausader-2020ParadifferentialWellPoseMuskatARMA}.   In \cite{CordobaGancedo-2009MaxPrincipleForMuskat-CMP},  equation (\ref{eqLIT:MuskatFPrime}) was rewritten as a fully nonlinear integro-differential equation on $f$ itself, instead of $\partial_x f$, which is given in 1-d as
\begin{align}\label{eqLITMuskatForF}
	\partial_t f = \int_\real \frac{f(y,t)-f(x,t)-(y-x)\partial_x f(x,t)}{(y-x)^2 + (f(y,t)-f(x,t))^2}dy,
\end{align}
which is an equation of the form,
\begin{align*}
	\partial_t f = \int_\real \del_y f(x,t)K_f(y,t)dy,
\end{align*}
where $K_f\geq0$ is a kernel that depends on $f$ and has the same structure as what we provide in Theorem \ref{thm:StructureOfHMain} above.  The integro-differential equation for $f$ (as opposed to $\partial_x f$) played a role in \cite{CordobaGancedo-2009MaxPrincipleForMuskat-CMP} to show non-expansion of the Lipschitz norm of solutions with nice enough data.   The integro-differential nature of the Muskat problem was subsequently utilized in \cite{Cameron-2019WellPosedTwoDimMuskat-APDE}, \cite{Cameron-2020WellPose3DMuskatMediumSlope-ArXiv} to study well-posedness for Lipschitz data as well as establish regularizing effects from (\ref{eqLITMuskatForF}).  Thus, in spirit, our work combined with \cite{ChangLaraGuillenSchwab-2019SomeFBAsNonlocalParaboic-NonlinAnal} is very close to \cite{Cameron-2019WellPosedTwoDimMuskat-APDE}, \cite{Cameron-2020WellPose3DMuskatMediumSlope-ArXiv}.  The other variation closely related to our work is to utilize the equation for $f$ given by the operator, $\tilde G(f)$ shown above, and this is used in \cite{AlazardMeunierSmets-2019LyapounovCauchyProbHeleShaw-ArXiv} for one-phase Hele-Shaw with gravity and \cite{NguyenPausader-2020ParadifferentialWellPoseMuskatARMA} for both the one and two phase Muskat problem.  The analogy is easiest to see for the one phase problem, and in both \cite{AlazardMeunierSmets-2019LyapounovCauchyProbHeleShaw-ArXiv} and \cite{NguyenPausader-2020ParadifferentialWellPoseMuskatARMA} it is established that if the graph of $f$ gives the free surface, then $f$ can be completely characterized by the flow
\begin{align}\label{eqLIT:DtoNMuskat}
	\partial_t f = \tilde G(f)f  \ \ \ \text{on}\ \real^n\times[0,T].
\end{align}
At least for the Hele-Shaw type flow we study in (\ref{eqIN:HSMain}), it appears as though the first result to show that weak solutions (viscosity solutions) of (\ref{eqIN:HSMain}) are equivalent to the flow governed by the Dirichlet-to-Neumann operator acting on $f$, as above in (\ref{eqLIT:DtoNMuskat}) (in our context, this is $H$ in (\ref{eqIN:DefOfH}) and (\ref{eqIN:HeleShawIntDiffParabolic})), was proved in \cite{ChangLaraGuillenSchwab-2019SomeFBAsNonlocalParaboic-NonlinAnal}.  The reduction to the equation for the free surface is not surprising, as a similar (and more complicated) reduction to a system for the free surface in water waves was known since \cite{CraigSulem-1993NumericalSimulationGravityWaves-JCompPhys} (also appearing in  \cite{AlazardBurqZuily-2014CauchyProblemGravityWaterWaves-Inventiones}, \cite{Lannes-2005WellPoseWaterWave-JAMS}, among others)-- the novelty in \cite{ChangLaraGuillenSchwab-2019SomeFBAsNonlocalParaboic-NonlinAnal} was that the reduction holds for viscosity solutions, which may not be classical.  In \cite{ChangLaraGuillenSchwab-2019SomeFBAsNonlocalParaboic-NonlinAnal} it was shown that under the graph assumption, the notion of the viscosity free boundary solution for $U$ is equivalent to the viscosity solution of the equation for $f$, which is (\ref{eqIN:HeleShawIntDiffParabolic}).  Furthermore, global in time existence and uniqueness for (\ref{eqIN:HeleShawIntDiffParabolic})-- or well-posedness-- holds, and it can be used to construct solutions to (\ref{eqIN:HSMain}), as well as show that a modulus of continuity for the initial interface will be preserved for all time.  Subsequently, both \cite{AlazardMeunierSmets-2019LyapounovCauchyProbHeleShaw-ArXiv} and \cite{NguyenPausader-2020ParadifferentialWellPoseMuskatARMA} showed that for respectively the one-phase Hele-Shaw with gravity and the Muskat problem, the equation (\ref{eqLIT:DtoNMuskat}) is equivalent to solving the original free boundary problem, and that this equation is globally in time well posed in $H^s(\real^n)$ for $s>1+\frac{n}{2}$, regardless of the size of the initial data in $H^s$.  Thus, the work in \cite{ChangLaraGuillenSchwab-2019SomeFBAsNonlocalParaboic-NonlinAnal} and our work here is again, very closely related to \cite{AlazardMeunierSmets-2019LyapounovCauchyProbHeleShaw-ArXiv}, \cite{NguyenPausader-2020ParadifferentialWellPoseMuskatARMA}, by utilizing (\ref{eqLIT:DtoNMuskat}) directly.  There is an important difference to note, however, where the results in \cite{ChangLaraGuillenSchwab-2019SomeFBAsNonlocalParaboic-NonlinAnal} and our results here exploit the fact that $H$ enjoys the global comparison property (see Definition \ref{defFD:GCP}) and the structure provided by Theorem \ref{thm:StructureOfHMain}, contrasted with \cite{AlazardMeunierSmets-2019LyapounovCauchyProbHeleShaw-ArXiv}, \cite{NguyenPausader-2020ParadifferentialWellPoseMuskatARMA} for which the analysis is derived from the properties of $\tilde G$ as a mapping on $H^s$.

%%%%%%%%%%%%%%%%%%%%%%%%%%%%%%%%%%%%%%%%%%%%%%%%%
%%%%%%%%%%%%%%%%%%%%%%%%%%%%%%%%%%%%%%%%%%%%%%%%%
%%%%%%%%%%%%%%%%%%%%%%%%%%%%%%%%%%%%%%%%%%%%%%%%%
%%%%%%%%%%%%%%%%%%%%%%%%%%%%%%%%%%%%%%%%%%%%%%%%%
\subsection{Parabolic integro-differential equations}

For the sake of presentation, in the context of this paper, the parabolic integro-differential equations that we utilize are of the form
\begin{align}\label{eqLIT:ParabolicIntDiff}
	\partial_t f = b(x)\cdot \grad f + \int_{\real^n} \del_hf(x,t) K(x,h)dh,
\end{align} 
with $\del_h f(x) = f(x+h)-f(x)-\Indicator_{B_{r_0}}(h)\grad f(x)\cdot h$,
and their nonlinear counterparts given as those in Theorem \ref{thm:StructureOfHMain}. Here, $b$ is a bounded vector field, and $K\geq0$.  The main issue for our work is the possibility that solutions of (\ref{eqLIT:ParabolicIntDiff}) enjoy some sort of extra regularity when $K$ has better behavior than simply being non-negative.  Are solutions to (\ref{eqLIT:ParabolicIntDiff}) H\"older continuous in some way that still allows for rough coefficients?  Are they $C^{1,\al}$?  We note that as written, (\ref{eqLIT:ParabolicIntDiff}), is an equation in non-divergence form, and in the literature, the theory that addresses these questions commonly carries the name Krylov-Safonov results, which comes from the result for local, second order parabolic equations \cite{KrSa-1980PropertyParabolicEqMeasurable} (in the divergence case, they usually carry the name De Giorgi - Nash - Moser).  These questions pertaining to (\ref{eqLIT:ParabolicIntDiff}) have gathered considerable attention in the past 20 or so years, and most of the works relevant to our study find their origins in either \cite{BaLe-2002Harnack} or \cite{BaLe-2002TransitionProb}, followed by a combination of \cite{CaSi-09RegularityIntegroDiff} and \cite{Silv-2006Holder}.   Examples of the works on parabolic equations that are close to our needs include \cite{ChDa-2012RegNonlocalParabolicCalcVar}, \cite{ChangLaraDavila-2014RegNonlocalParabolicII-JDE}, \cite{ChangLaraDavila-2016HolderNonlocalParabolicDriftJDE}, \cite{SchwabSilvestre-2014RegularityIntDiffVeryIrregKernelsAPDE}, \cite{Serra-2015RegularityNonlocalParabolicRoughKernels-CalcVar}, \cite{Silvestre-2011RegularityHJE}, \cite{Silvestre-2014RegularityParabolicICM}.   We note there are many references for elliptic problems and problems involving existence and uniqueness of viscosity solutions which are not mentioned above.

A common feature of most of the parabolic works listed above is that they arose from the interest of studying the probabilistic implications and analytical properties of equations like (\ref{eqLIT:ParabolicIntDiff}) for their sake as fundamental mathematical objects in their own right.  A typical and frequently mentioned application among the nonlinear works is their relationship to optimal control and differential games.  There has also been interest in utilizing equations like (\ref{eqLIT:ParabolicIntDiff}) in situations which are not necessarily originally posed as an integro-differential equation, such as we do in this work as it pertains to (\ref{eqIN:HeleShawIntDiffParabolic}).  In our case, we find that (\ref{eqIN:HeleShawIntDiffParabolic}) coincidentally landed within the scope of existing results, as the reader may see in Sections \ref{sec:BackgroundParabolic} and \ref{sec:KrylovSafonovForHS}.   This is not always the case, and sometimes the intended application of the integro-differential theory has led to new advances in the integro-differential field.  One recent occurrence of this is the application of integro-differential techniques to the Boltzmann equation.  For the homogeneous Boltzmann equation, new integro-differential results were first produced in \cite{SchwabSilvestre-2014RegularityIntDiffVeryIrregKernelsAPDE} to be subsequently applied in \cite{Silvestre2016NewRegularizationBoltzCMP} (which is mentioned in \cite[Section 1B]{SchwabSilvestre-2014RegularityIntDiffVeryIrregKernelsAPDE}).  Even more advanced techniques were required for the inhomogenous Boltzmann equation, and one can see the evolution of the integro-differential theory in \cite{ImbertSilvestre-2020WeakHarnackBoltzmann-JEMS}, which was followed by \cite{ImbertSilvestre-2018SchauderForKineticIntegralEq-ArXiV}, \cite{ImbertSilvestre-2019GlobalRegularityBoltzmannWithoutCutOff-ArXiv}, \cite{ImbertSilvestre-2020RegularityBoltzmannConditionalMacro-ArXiv}.

%%%%%%%%%%%%%%%%%%%%%%%%%%%%%%%%%%%%%%%%%%%%%%%%%
%%%%%%%%%%%%%%%%%%%%%%%%%%%%%%%%%%%%%%%%%%%%%%%%%
%%%%%%%%%%%%%%%%%%%%%%%%%%%%%%%%%%%%%%%%%%%%%%%%%
%%%%%%%%%%%%%%%%%%%%%%%%%%%%%%%%%%%%%%%%%%%%%%%%%
%%%%%%%%%%%%%%%%%%%%%%%%%%%%%%%%%%%%%%%%%%%%%%%%%
%%%%%%%%%%%%%%%%%%%%%%%%%%%%%%%%%%%%%%%%%%%%%%%%%
%%%%%%%%%%%%%%%%%%%%%%%%%%%%%%%%%%%%%%%%%%%%%%%%%
%%%%%%%%%%%%%%%%%%%%%%%%%%%%%%%%%%%%%%%%%%%%%%%%%
%%%%%%%%%%%%%%%%%%%%%%%%%%%%%%%%%%%%%%%%%%%%%%%%%

\section{Notation and Assumptions}

We will collect some notation here.

\begin{itemize}
	\item $n$ is the dimension of the free boundary hypersurface, with $n\geq 2$.
	\item $X=(x,x_{n+1})\in\real^{n+1}$.
	\item $B_r(x)\subset\real^n$ and $B^{n+1}_r(X)\subset\real^{n+1}$.  When the context is clear, the superscript may be dropped.
	\item $d(x,y)$ is the distant between $x$ and $y$, $d(x,E)$ is the distance between $x$ and a set $E$, and may be abbreviated $d(x)$ when $d(x,E)$ is understood for a particular $E$.
	\item $\nu_f(X)$ is the unit normal vector to the boundary at $X\in\Gam_f$, often abbreviated without the subscript. 
	\item $I^+$ and $I^-$ are the respective normal derivatives from the positive and negative phases of $U_f$, defined using (\ref{eqIN:TwoPhaseBulk}), (\ref{eqIN:DefOfPosNegNormalDeriv}), and (\ref{eqIN:DefIPLusAndMinus}).
	\item $C^{1,\dini}_\rho$ is a Banach space, as in Stein \cite[Chapter VI, Cor 2.2.3 and Exercise 4.6]{Stei-71} (also see \eqref{eqIN:DefOfSetK} as well as $X_\rho$ in Definition \ref{defFD:XrhoSpace}).
	\item $X_\rho$, see Definition \ref{defFD:XrhoSpace} and Remark \ref{remFD:ModulusIsAlsoInXRho}, cf. \cite[Chapter VI, Cor 2.2.3 and Exercise 4.6]{Stei-71}.
	\item $C^0(\real^N)$ is the space of continuous functions on $\real^N$.
	\item $C^0_b(\real^N)$ is the Banach space of continuous bounded functions with the norm $\norm{\cdot}_{L^\infty}$.
	\item $C^{1,\al}_b(\real^N)$ is the space of functions that are bounded with bounded derivatives, with the derivatives $\al$-H\"older continuous.
	\item $D_f=\{(x,x_{n+1})\ :\ 0< x_{n+1}<f(x) \}=D^+_f$, $D^-_f=\{(x,x_{n+1})\ :\ f(x)<x_{n+1}<L \}$ 
	\item $\Gam_f=\graph(f)=\{ (x,x_{n+1})\ :\ x_{n+1}=f(x) \}$.
	\item $d\sigma_f$ the surface measure on $\Gam_f$, often abbreviated without the subscript as $d\sigma$.
	\item $G_f$ the Green's function in $D_f$ for the operator in \eqref{eqIN:TwoPhaseBulk}.
	\item $P_f$ the Poisson kernel for $D_f$ on $\Gam_f$.
	\item $\K(\del, L, m, \rho)$, see (\ref{eqIN:DefOfSetK}).
	\item $\del_h f(x) = f(x+h)-f(x)-\Indicator_{B_{r_0}}(h) \grad f(x)\cdot h$
\end{itemize}

%%%%%%%%%%%%%%%%%%%%%%%%%%%%%%%%%%%%%%%%%%%%%%%%%
%%%%%%%%%%%%%%%%%%%%%%%%%%%%%%%%%%%%%%%%%%%%%%%%%
%%%%%%%%%%%%%%%%%%%%%%%%%%%%%%%%%%%%%%%%%%%%%%%%%
%%%%%%%%%%%%%%%%%%%%%%%%%%%%%%%%%%%%%%%%%%%%%%%%%
%%%%%%%%%%%%%%%%%%%%%%%%%%%%%%%%%%%%%%%%%%%%%%%%%
%%%%%%%%%%%%%%%%%%%%%%%%%%%%%%%%%%%%%%%%%%%%%%%%%
%%%%%%%%%%%%%%%%%%%%%%%%%%%%%%%%%%%%%%%%%%%%%%%%%
%%%%%%%%%%%%%%%%%%%%%%%%%%%%%%%%%%%%%%%%%%%%%%%%%
%%%%%%%%%%%%%%%%%%%%%%%%%%%%%%%%%%%%%%%%%%%%%%%%%

\section{Background results on Green's Functions and Parabolic Equations}\label{sec:BackgroundTools}

This section has two subsections, collecting respectively background results related to Green's functions for equations in Dini domains and background results for fractional parabolic equations.

%%%%%%%%%%%%%%%%%%%%%%%%%%%%%%%%%%%%%%%%%%%%%%%%%
%%%%%%%%%%%%%%%%%%%%%%%%%%%%%%%%%%%%%%%%%%%%%%%%%
%%%%%%%%%%%%%%%%%%%%%%%%%%%%%%%%%%%%%%%%%%%%%%%%%
%%%%%%%%%%%%%%%%%%%%%%%%%%%%%%%%%%%%%%%%%%%%%%%%%

\subsection{Boundary behavior of Green's functions}\label{sec:GreenFunction}

We utilize results about the boundary behavior of Green's functions for equations with Dini coefficients in domains that have $C^{1,\dini}$ boundaries, for a Dini modulus, $\omega$. We will use the shorthand $d(x) := \text{dist}(x,\partial \O)$ for $x \in \O$.  The main way in which we use the boundary behavior of the Green's function is to deduce the boundary behavior of the Poisson kernel as well as that of solutions that may vanish on a portion of the boundary.  The study of the boundary behavior of Green's functions is a well developed topic, and none of what we present here is new.  The results in either Theorem \ref{thm:GreenBoundaryBehavior} or Proposition \ref{propGF:PoissonKernel} reside in the literature in various combinations of \cite{Bogdan-2000SharpEstGreenLipDomJMAA}, \cite{Cho-2006GreensFunctionPotentialAnal}, \cite{Zhao-1984UniformBoundednessConditionalGaugeCMP}, among other references.

\begin{theorem}\label{thm:GreenBoundaryBehavior} 
	If $G_f$ is the Green's function for the domain, $D_f$, then there exist positive constants $C_1$, $C_2$, and $R_0$, that depend upon the Dini modulus of $\grad f$ and other universal parameters so that for all $x,y\in D_f$ with $\abs{x-y}\leq R_0$
	\begin{equation}\label{GlobalGreenEstimate}
		C_1\min\left\{ \frac{d(x)d(y)}{\abs{x-y}^{n+1}}, \frac{1}{4\abs{x-y}^{n-1}}   \right\}
		\leq G_f(x,y)
		\leq C_2\min\left\{ \frac{d(x)d(y)}{\abs{x-y}^{n+1}}, \frac{1}{4\abs{x-y}^{n-1}}   \right\}.
	\end{equation}
\end{theorem}

The essential ingredient in the proof of Theorem \ref{thm:GreenBoundaryBehavior} is the following Lemma \ref{lemGF:lineargrowth} on the growth of solutions away from their zero set. Before stating this result, we need a few definitions. Denote by $[x_0,z_0]$ the closed line segment with endpoints $x_0,z_0 \in \O$, and denote by $\A_{2r}(x_0)$ the annulus $B_{2r}(x_0) \backslash B_r(x_0)$. 

\begin{definition} 
A domain $\O \subset \real^{n+1}$ satisfies the uniform interior ball condition with radius $\rho_0$ if for every $\xi \in \partial \O$, there exists an open ball $B$ of radius $\rho_0$ such that $B \subset \O$ and $\overline{B} \cap \partial \O = \{\xi\}$. 
\end{definition} Observe that since $\delta \leq f\leq L-\delta$ and $\nabla f$ has a Dini modulus of continuity $\rho$, there exists a $C^{1,\text{Dini}}$ map $T_f : \overline{D_f} \rightarrow \RN \times [0,L]$ satisfying
\begin{equation}\label{flatteningtransformation}
\begin{cases}
T_f(D_f) = \RN \times [0,L], \\
T_f(\Gamma_f) = \left\{x_{n+1} = L \right\}, \\
T_f(\left\{x_{n+1} = 0 \right\}) = \left\{x_{n+1} = 0 \right\}.
\end{cases}
\end{equation}
Consequently, the function $V_f := U_f\circ T_f^{-1}$ satisfies an equation of the form $L_A V_f = -\text{div}(A(y) \nabla V_f(y)) = 0$ on $\RN \times [0,L]$, where the coefficients $A(\cdot)$ satisfy $0 < \lambda \mathbb{I}_{n+1} \leq A \leq \Lambda \mathbb{I}_{n+1}$ with $\lambda, \Lambda$ depending on $\delta, L, m$, and are Dini continuous on $\RN \times [0,L]$ up to the boundary with a modulus of continuity $\omega$. Thus, for the purposes of the next lemma, we will only consider a domain $\O \subset \Rn$ which satisfies the uniform interior ball condition with radius $\rho_0$ and a solution to a uniformly elliptic equation in divergence form on $\O$ with coefficients having a Dini modulus of continuity $\omega$.

\begin{lemma}\label{lemGF:lineargrowth} Suppose $\O \subset \Rn$ satisfies the uniform interior ball condition with radius $\rho_0$. Let $u\in C^2(\O) \cap C(\overline{\Omega})$ be non-negative and satisfy
$$
\begin{cases}
L_A u = 0 \quad \text{in } \O,\\
u = 0 \quad \text{on } \Gamma \subset \partial \O,
\end{cases}
$$
Then there exist positive constants $C = C(n,\lambda,\Lambda)$ and $r_0 = r_0(n,\omega,\lambda,\Lambda) \leq \frac{\rho_0}{2}$ such that for all balls $B_{2r}(x_0) \subset \O$ with $\overline{B_{2r}(x_0)} \cap \Gamma \neq \emptyset$ and $r \leq r_0$, we have the estimate	
\begin{equation}\label{lineargrowthatboundary}
u(x) \geq \frac{C}{r} u(x_0) d(x) + o(d(x)) \qquad \text{ for all } x \in [x_0,z_0] \cap \A_{2r}(x_0), \ z_0 \in \overline{B_{2r}(x_0)} \cap \Gamma.
\end{equation}
\end{lemma}

Let us state some useful consequences of Lemma \ref{lemGF:lineargrowth} and Theorem \ref{thm:GreenBoundaryBehavior}. First, notice that Lemma \ref{lemGF:lineargrowth} implies the following uniform linear growth of $U_f$ away from $\Gam_f$.

\begin{lemma}\label{lemGF:LinearGrowthFromGammaF}
	There exist a constant $C>0$ that depends on $\del$, $L$, $m$, $\rho$, so that for all $f\in\K(\del,L,m,\rho)$, for $U_f$ defined in (\ref{eqIN:BulkEqForHSOperator}), and for all $Y\in\Gam_f$,
\begin{align*}
	\frac{s}{C}\leq U_f(Y-sy_{n+1})\leq Cs,\ \ \text{and}\ \ 
	\frac{s}{C}\leq U_f(Y+s\nu_f(Y))\leq Cs.
\end{align*}
(Recall $\nu_f$ is the inward normal to $D_f$.)
\end{lemma}

Theorem \ref{thm:GreenBoundaryBehavior} also induces the following behavior on the Poisson kernel.

\begin{proposition}\label{propGF:PoissonKernel}
	If $f\in \K(\del, L, m, \rho)$ and $P_f$ is the Poisson kernel for the domain, $D_f$, then there exists constants $C_1$, $C_2$, $C_3$ and $R_0$, that depend upon  $\delta, L, m, \rho$ and other universal parameters so that for all $X\in D_f$, $Y\in\Gam_f$, with $\abs{X-Y}\leq R_0$,
	\begin{align*}
		C_1 \frac{d(X)}{\abs{X-Y}^{n+1}}
		\leq P_f(X,Y)
		\leq C_2 \frac{d(X)}{\abs{X-Y}^{n+1}}.
	\end{align*}
	Furthermore, there exists an exponent, $\al\in(0,1]$, depending on $\delta, L, m, \rho$ and universal parameters, so that for $X\in\Gam_f$ and with $R>R_0$,  
	\begin{equation}\label{eqGF:DecayOnMassOfPoissonKernel}
		\int_{\Gam_f\setminus B_{R}} P_f(X+s\nu(X),Y)d\sigma_f(Y)\leq \frac{Cs}{R^\al}.
	\end{equation}
\end{proposition}

For technical reasons, we also need a slight variation on Proposition \ref{propGF:PoissonKernel}, which is related to conditions necessary to invoke results from the earlier work \cite{GuSc-2019MinMaxEuclideanNATMA} that we state here in  Theorem \ref{thmFD:StructureOfJFromGSEucSpace}.  

\begin{lemma}\label{lemGF:DecayAtInfinityForTechnicalReasons}
	There exists constants $c_0$, $C>0$ and $\al\in(0,1]$, depending on $\del$, $L$, $m$, $\rho$, so that if $f\in C^{1,\dini}_\rho(B_{2R}(0))$, $\del\leq f\leq L-\del$, and $\abs{\grad f}\leq m$,  then for $X\in B_R\intersect \Gam_f$ and $0<s<c_0$,
	\begin{align*}
		\int_{\Gam_f\setminus B_{2R}}P_f(X+s\nu(X),Y)d\sigma_f(Y) \leq \frac{Cs}{R^\al}.
	\end{align*}
\end{lemma}

\begin{rem}
	It is worth noting that based on purely the Lipschitz constant of $f$, one would obtain this same estimate in Lemma \ref{lemGF:DecayAtInfinityForTechnicalReasons}, but with the upper bound of $C\frac{s^\al}{r^\al}$.  The Dini condition in $B_{2R}$ is what allows to obtain $s$, instead of $s^\al$ in the estimate.
\end{rem}

For the convenience of the reader, we have provided proofs of the above results in the Appendix. See \cite{Cho-2006GreensFunctionPotentialAnal} for a parabolic version of related results.

%%%%%%%%%%%%%%%%%%%%%%%%%%%%%%%%%%%%%%%%%%%%%%%%%
%%%%%%%%%%%%%%%%%%%%%%%%%%%%%%%%%%%%%%%%%%%%%%%%%
%%%%%%%%%%%%%%%%%%%%%%%%%%%%%%%%%%%%%%%%%%%%%%%%%
%%%%%%%%%%%%%%%%%%%%%%%%%%%%%%%%%%%%%%%%%%%%%%%%%

\subsection{Background results on regularity for integro-differential equations}\label{sec:BackgroundParabolic}
For our purposes, we will invoke results for parabolic integro-differential equations that originate mainly in Chang Lara - Davila \cite{ChangLaraDavila-2016HolderNonlocalParabolicDriftJDE} and Silvestre \cite{Silvestre-2011RegularityHJE}.

Following \cite{Chan-2012NonlocalDriftArxiv} and \cite{ChangLaraDavila-2016HolderNonlocalParabolicDriftJDE}, we consider fully nonlinear parabolic equations whose linear versions are $(\partial_t - L_{K,b} )u$, where for $u : \RN  \times \mathbb{R} \rightarrow \mathbb{R}$,
\begin{equation}\label{eqPara:linearoperators}
L_{K,b} u(x,t) := b(x,t)\cdot\grad u(x,t) + \int_{\real^n}\del_h u(x,t)K(x,t,h) \ dh, 
\end{equation}
 $b(x,t) \in \mathbb{R}^n$ is a bounded vector field and $\del_hu(x,t) := u(x+h,t)-u(x,t)- \Indicator_{B_{r_0}}(h)\grad u(x,t)\cdot h$. For any $r \in (0,r_0)$, consider the rescaled function
$$u_r(x,t) := \frac{1}{r} u(rx, rt).$$
A direct calculation shows that if $u$ satisfies the equation $(\partial_t - L_{K,b})u(x,t) = \varphi(x,t)$, then $u_r$ satisfies the equation $(\partial _t - L_{K_r,b_r})u_r(x,t) = \varphi_r(x,t)$, where
$$K_r(x,t,h) := r^{n+1} K(rx,rt, rh), \quad b_r(x,t) := b(rx,rt) - \int_{B_{r_0} \backslash B_r} h K(rx,rt,h) \ dh, \quad \varphi_r(x,t) = \varphi(rx,rt).$$

Based on this scaling behavior, we are led to consider the following class of linear operators.

\begin{definition}[cf. Section 2 of \cite{ChangLaraDavila-2016HolderNonlocalParabolicDriftJDE}]\label{defPara:scaleinvariantclass} Given a positive number $\Lam$, the class $\L_{\Lam}$ is the collection of linear operators of the form $L_{K,b}$ as in \eqref{eqPara:linearoperators} with $K$ and $b$ satisfying the properties
\begin{align*}
	&\mathrm{(i)} \ 
	\Lam^{-1} \abs{h}^{-n-1}\leq K(x,t,h)\leq \Lam\abs{h}^{-n-1} \qquad \text{for all } (x,t,h) \in \RN\times[0,T]\times\RN, \\
	&\mathrm{(ii)} \ \sup_{0 < \rho < 1, \ (x,t) \in \Rn} \abs{b(x,t)-\int_{B_{r_0} \setminus B_{\rho}} hK(x,t,h)dh}\leq \Lam.
\end{align*}
\end{definition}
Let us show that if $b,K \in \L_{\Lam}$ then $b_r,K_r\in\L_{\Lam}$ for all $r \in (0,1)$. We suppress the dependence on $t$. The bounds (i) on the kernels are immediate: for the upper bound, we have
$$K_r(x,h) = r^{n+1} K(rx, rh) \leq r^{n+1} \Lambda |rh|^{-n-1} = \Lambda |h|^{-n-1},$$
while for the lower bound, we have 
$$K_r(x,h) = r^{n+1} K(rx, rh) \geq \Lam^{-1} r^{n+1} |rh|^{-n-1} \geq  \Lam^{-1} |h|^{-n-1}.$$
To show that $b_r, K_r$ satisfy (ii), let $\rho \in (0,1)$ and $x \in \RN$ be arbitrary. Then
\begin{align*}
\abs{b_r(x)-\int_{B_1 \setminus B_{\rho}} hK_r(x,h)dh} & = \abs{b(rx) - \int_{B_1 \setminus B_r} hK(rx,h)dh - \int_{B_1 \setminus B_{\rho}} hr^{n+1}K(rx,rh)dh} \\
& = \abs{b(rx) - \int_{B_1 \setminus B_r} hK(rx,h)dh - \int_{B_r \setminus B_{\rho r}} h K(rx,h)dh} \\
& = \abs{ b(rx) - \int_{B_1 \setminus B_{\rho r}} hK(rx,h)dh} \leq \Lambda.
\end{align*}
Consequently, $b_r, K_r \in \L_{\Lam}$.

The class $\L_{\Lam}$ gives rise to the extremal operators 
\begin{align}\label{eqPara:ExtremalOperators}
	\M^+_{\L_{\Lam}}(u) = \sup_{L\in\L_{\Lam}}L(u), \quad \M^-_{\L_{\Lam}}(u) = \inf_{L\in \L_{\Lam}}L(u).
\end{align}
These operators are typically used to characterize differences of a given nonlocal operator, say $J$, in which one would require
\begin{align}\label{eqPara:ExtremalInequalities}
	\M^-_{\L_{\Lam}}(u-v)\leq J(u)-J(v)\leq M^+_{\L_\Lam}(u-v),
\end{align}
where one can change the operators by changing the set of functionals included in $\L_\Lam$.  This is what is known as determining an ``ellipticity'' class for $J$.
By the scale invariance of $\L_{\Lam}$ we know that $\M^{\pm}_{\L_{\Lam}}(u_r)(x) = \M^{\pm}_{\L_{\Lam}}(u)(rx)$.

The cylinders corresponding to the maximal operators $\M^{\pm}_{\L_{\Lam}}$ are
$$Q_r = (-r,0] \times B_r(0), \qquad Q_r(t_0,x_0) = (t_0 - r, t_0] \times B_r(x_0).$$

\begin{definition} The function $u$ is a viscosity supersolution of the equation
$$\partial_t u - \M^{-}_{\L_{\Lam}}u = \varphi$$
if for all $\e > 0$ and $\psi : (t,x) \in \real \times \RN \rightarrow \real$ left-differentiable in $t$, twice pointwise differentiable in $x$, and satisfying $\psi(t,x) \leq u(t,x)$ with equality at $(t_0,x_0)$, the function $v_{\e}$ defined as
$$
v_{\e}(t,x) =
\begin{cases}
\psi(t,x) \quad \text{if } (t,x) \in Q_{\e}(t_0,x_0), \\
u(t,x) \quad \text{otherwise}
\end{cases}
$$
satisfies the inequality
$$\partial_t v_{\e}(t_0,x_0) - \M^{-}_{\L_{\Lam}}v_{\e}(t_0,x_0) \geq \varphi(t_0,x_0).$$

The corresponding definition of a viscosity subsolution is obtained by considering a function $\psi$ satisfying $\psi(t,x) \geq u(t,x)$ with equality at $(t_0,x_0)$, and requiring
\begin{align*}
	\partial_t v_{\e}(t_0,x_0) - \M^{-}_{\L_{\Lam}}v_{\e}(t_0,x_0) \leq \varphi(t_0,x_0).
\end{align*}

The same definitions hold for $\partial_t u -\M^+_{\L_{\Lam}}=\varphi$. 
\end{definition}

The main regularity result that we need is stated below, and can be found in \cite{ChangLaraDavila-2016HolderNonlocalParabolicDriftJDE}; see also \cite{SchwabSilvestre-2014RegularityIntDiffVeryIrregKernelsAPDE,   Silvestre-2011RegularityHJE, Silvestre-2014RegularityParabolicICM}.

\begin{proposition}[H\"older Estimate, Section 7 of \cite{ChangLaraDavila-2016HolderNonlocalParabolicDriftJDE}]\label{propPara:holderestimate} Suppose $u$ is bounded in $\mathbb{R}^n \times [0,t_0]$ and satisfies in the viscosity sense
\begin{equation}\label{eqPara:subandsuper}
\begin{cases}
\partial_t u - \M_{\L_{\Lam}}^+u \leq A \\
\partial_t u - \M_{\L_{\Lam}}^-u \geq -A
\end{cases}
\end{equation}
in $Q_{t_0}(t_0,x_0)$ for some constant $A > 0$. Then there exist constants $C>0$ and $\gamma\in(0,1)$, depending only on $n$ and $\Lambda$, such that
$$
||u||_{C^{\gamma}(Q_{\frac{t_0}{2}}(t_0,x_0))} \leq \frac{C}{t_0^{\gamma}} \left(||u||_{L^{\infty}(\mathbb{R}^n \times [0,t_0])} + t_0 A \right).
$$
\end{proposition}

\begin{rem}
	The equations in (\ref{eqPara:subandsuper}) simply say that $u$ is a subsolution of $\partial_t u - \M^+_{\L_\Lam} u = A$ and a supersolution of $\partial_t u - \M^-_{\L_\Lam} u = -A$.
\end{rem}

Since Proposition \ref{propPara:holderestimate} differs slightly from \cite{ChangLaraDavila-2016HolderNonlocalParabolicDriftJDE} in that it accommodates the cylinder $Q_{t_0}(t_0,x_0)$ other than the standard cylinder $Q_1$ and also from \cite{Silvestre-2011RegularityHJE} in that it includes a non-zero right hand side, $A$, we make a small comment here as to the appearance of the term $t_0A$ in the conclusion of the estimate.  Indeed, this is simply a result of rescaling the equation. As in \cite{ChangLaraDavila-2016HolderNonlocalParabolicDriftJDE}, we already know that Proposition \ref{propPara:holderestimate} holds for $u$ that are bounded in $\real^n \times [-1,0]$ and satisfy \eqref{eqPara:subandsuper} in $Q_1$; in this case, the $C^{\gamma}$ estimate holds on $Q_{\frac{1}{2}}$.  Let us now show what happens for arbitrary $t_0 > 0$ and $x_0 \in \real^n$.

Let $u$ be as in the statement of Propositon \ref{propPara:holderestimate} and define $\tilde{u}(t,x) := u((t_0, x_0) + t_0(t,x))$. Notice that if $(t,x) \in Q_r$, then $(t_0, x_0) + t_0(t,x) \in Q_{t_0 r}(t_0,x_0)$ for all $r \in [0,1]$. By the translation and scaling invariance properties of the operators $\partial_t - \M^{\pm}_{\L_\Lam}$, we thus have
$$\partial_t \tilde{u} - \M^+_{\L_\Lam} \tilde{u} = t_0 (\partial_t u - \M^+_{\L_\Lam}u) \leq t_0 A \quad \text{ and } \quad \partial_t \tilde{u} - \M^-_{\L_\Lam}\tilde{u} =  t_0(\partial_t u - \M^-_{\L_\Lam}u) \geq -t_0A \quad \text{ in } Q_1.$$
On the other hand, we also have $||\tilde{u}||_{L^{\infty}(\real^n \times [-1,0])} = ||u||_{L^{\infty}(\real^n \times [0,t_0])}$ and for all $(t,x) \in Q_{\frac{1}{2}}$,
$$\frac{|\tilde{u}(t,x) - \tilde{u}(0,0)|}{|(t,x)|^{\gamma}} = \frac{|u((t_0,x_0) + t_0(t,x)) - u(t_0,x_0)|}{|(t,x)|^{\gamma}} = \frac{t_0^{\gamma} |u((t_0,x_0) + t_0(t,x)) - u(t_0,x_0)|}{|(t_0,x_0) + t_0(t,x) - (t_0,x_0)|^{\gamma}}.$$
Consequently, $||\tilde{u}||_{C^{\gamma}(Q_{\frac{1}{2}})} = t_0^{\gamma} ||u||_{C^{\gamma}(Q_{\frac{t_0}{2}}(t_0,x_0))}$. The conclusion follows by applying to $\tilde{u}$ the version of Proposition \ref{propPara:holderestimate} for functions that are bounded in $\real^n \times [-1,0]$ and satisfy \eqref{eqPara:subandsuper} in $Q_1$, and then rewriting the resulting $C^{\gamma}$ estimate in terms of $u$.

%%%%%%%%%%%%%%%%%%%%%%%%%%%%%%%%%%%%%%%%%%%%%%%%%
%%%%%%%%%%%%%%%%%%%%%%%%%%%%%%%%%%%%%%%%%%%%%%%%%
%%%%%%%%%%%%%%%%%%%%%%%%%%%%%%%%%%%%%%%%%%%%%%%%%
%%%%%%%%%%%%%%%%%%%%%%%%%%%%%%%%%%%%%%%%%%%%%%%%%
%%%%%%%%%%%%%%%%%%%%%%%%%%%%%%%%%%%%%%%%%%%%%%%%%
%%%%%%%%%%%%%%%%%%%%%%%%%%%%%%%%%%%%%%%%%%%%%%%%%
%%%%%%%%%%%%%%%%%%%%%%%%%%%%%%%%%%%%%%%%%%%%%%%%%
%%%%%%%%%%%%%%%%%%%%%%%%%%%%%%%%%%%%%%%%%%%%%%%%%
%%%%%%%%%%%%%%%%%%%%%%%%%%%%%%%%%%%%%%%%%%%%%%%%%

\section{A Finite Dimensional Approximation For $I$}\label{sec:FiniteDim}

An important note for this section is we will take $N$ to be an arbitrary dimension, and we are looking generically at operators on $C^{1,\dini}(\real^N)$.  The application to equation (\ref{eqIN:HSMain}) will be for $N=n$ (as $f:\real^n\to\real$).

Here we will record some tools that were developed in  \cite{GuSc-2019MinMaxEuclideanNATMA} and \cite{GuSc-2019MinMaxNonlocalTOAPPEAR} to investigate the structure of operators that enjoy what we call a global comparison property (see Definition \ref{defFD:GCP}, below).  The point of these tools is to build linear mappings that can be used to ``linearize'' the nonlinear operator, $I$, through the min-max procedure apparent in Theorem \ref{thm:StructureOfHMain}, or more precisely, to reconstruct $I$ from a min-max of a special family of linear operators.

The linear mappings we build to achieve a min-max for $I$ are limits of linear mappings that are differentials of maps with similar properties for a family of simpler operators that can be used to approximate $I$.  The advantage of the approximations constructed in  \cite{GuSc-2019MinMaxEuclideanNATMA} and \cite{GuSc-2019MinMaxNonlocalTOAPPEAR} is that they are operators with the same domain as $I$ but enjoy the property of having finite rank (with the rank going to infinity as the approximates converge to the original).  In this regard, even though the original operator and approximating operators are nonlinear, the approximates behave as Lipschitz operators on a high, but finite dimensional space, and are hence differentiable almost everywhere.  This differentiability makes the min-max procedure straightforward, and it is then passed through the limit back to the original operator, $I$. The basis for our finite dimensional approximation to $I$ is the Whitney extension for a family of discrete and finite subsets of $\real^n$, whose union is dense in $\real^n$.  The reason for doing this is that we can restrict the functions to be identically zero outside of a finite set, and naturally, the collection of these functions is a finite dimensional vector space.  Thus, Lipschitz operators on those functions will be differentiable almost everywhere, and as mentioned this is one of the main points of \cite{GuSc-2019MinMaxEuclideanNATMA} to represent $I$ as a min-max over linear operators.

\subsection{The Whitney Extension}

Here we just list some of the main properties of the Whitney extension constructed in \cite{GuSc-2019MinMaxEuclideanNATMA}.  It is a variant of the construction in Stein \cite{Stei-71}, where in \cite{GuSc-2019MinMaxEuclideanNATMA} it is designed to preserve the grid structure of $2^{-m}\integer^N$.  We refer the reader to \cite[Section 4]{GuSc-2019MinMaxEuclideanNATMA} for complete details.

\begin{definition}\label{defFD:GridSetGm}
For each $m\in\Natural$, the finite set, $G_m$, is defined as
\begin{align*}
	G_m=2^{-m}\integer^N.
\end{align*}
We will call $h_m$ the grid size, defined as $h_m=2^{-m}$.
\end{definition}

We note that in \cite[Section 4]{GuSc-2019MinMaxEuclideanNATMA}, the sets for the Whitney extension were constructed as a particular disjoint cube decomposition that covers $\real^n\setminus h_m\integer^N$ and was shown to be invariant under translations of $G_m$ by any vector in $G_m$.  For each $m$, we will index these sets by $k\in\Natural$, and we will call them $Q_{m,k}$.   See \cite[Section 4]{GuSc-2019MinMaxEuclideanNATMA} for the precise details of $Q_{m,k}$ and $\phi_{m,k}$.  Here we record these results.

\begin{lemma}[Lemma 4.3 in \cite{GuSc-2019MinMaxEuclideanNATMA}]\label{LemFD:Cubesmk}
  For every $m\in\Natural$, there exists a collection of cubes $\{Q_{m,k}\}_k$ such that
  \begin{enumerate}
    \item The cubes $\{Q_{m,k}\}_k$ have pairwise disjoint interiors.
	
    \item The cubes $\{Q_{m,k}\}_k$ cover $\mathbb{R}^d \setminus G_m$.	
	
    \item There exist a universal pair of constants, $c_1$, $c_2$, so that 
	\begin{align*}
		c_1 \textnormal{diam}(Q_{m,k})\leq \textnormal{dist}(Q_{m,k},G_m) \leq c_2 \textnormal{diam}(Q_{m,k}).
	\end{align*}
	
    \item For every $h \in G_m$, there is a bijection $\sigma_h:\mathbb{N}\to\mathbb{N}$ such that $Q_{m,k}+h = Q_{m,\sigma_h k}$ for every $k\in\mathbb{N}$. 
   
  \end{enumerate}
\end{lemma}

\begin{rem}\label{remFD:CubeParametersmk}
	Just for clarity, we make explicit for the reader: the parameter, $m\in\Natural$, is used for the grid size, $2^{-m}\integer^N$, and the parameter, $k\in\Natural$, in $Q_{m,k}$, etc. is the index resulting from a cube decomposition of $\real^N\setminus G_m$.
\end{rem}

\begin{rem}\label{remark:maximum number of overlapping cubes}
  In what follows, given a cube $Q$, we shall denote by $Q^*$ the cube with the same center as $Q$ but whose sides are increased by a factor of $9/8$. Observe that for every $m$ and $k$, we have $Q_{m,k}^* \subset \mathbb{R}^n\setminus 2^{2-m}\mathbb{Z}^N$, and that any given $x$ lies in at most some number $C(N)$ of the cubes $Q_{m,k}^*$.
\end{rem}

\begin{proposition}[Proposition 4.6 in \cite{GuSc-2019MinMaxEuclideanNATMA}]\label{propFD:ParitionOfUnitymk}
  For every $m$, there is a family of functions $\phi_{m,k}(x)$, with $k\in\Natural$, such that
  \begin{enumerate}
    \item $0\leq \phi_{m,k}(x)\leq 1$ for every $k$ and $\phi_{m,k} \equiv 0$ outside $Q_{m,k}^*$ (using the notation in Remark \ref{remark:maximum number of overlapping cubes})
    \item $\sum_k \phi_{m,k}(x) =1 $ for every $x \in \mathbb{R}^n\setminus G_m$.
    \item There is a constant $C$, independent of $m$ and $k$, such that
    \begin{align*}
      |\nabla\phi_{m,k}(x)| \leq \frac{C}{\diam(Q_{m,k})}.
    \end{align*}		
    \item For every $z \in G_m$, we have
    \begin{align*}
      \phi_{m,k}(x-z) = \phi_{m,\sigma_zk}(x),\;\;\forall\;k,\;x,
    \end{align*}		
    where $\sigma_z$ are the bijections introduced in Lemma \ref{LemFD:Cubesmk}. 
  \end{enumerate}

\end{proposition}

We will call $\{\phi_{m,k}\}$ the corresponding partition of unity for $\{Q_{m,k}\}$ that is appropriate for the Whitney extension.  As in \cite[Section 4]{GuSc-2019MinMaxEuclideanNATMA}, we use the following finite difference operator to construct approximate Taylor polynomials for the Whitney extension. Denote by $\grad_m^1u(x)$ the unique vector that satisfies for $x\in G_m$ and $j=1,\dots,n$
\begin{align}\label{eqFD:DiscreteGradient}
	\grad_m^1u(x)\cdot e_j = \frac{1}{2h_m}(u(x+h_me_j)-u(x-h_me_j)).
\end{align}
Note that this exploits the fact that $x\pm h_me_j\in G_m$ if $x \in G_m$.

In order to define the polynomials that will be used to build the Whitney extension, we need some notation for the centers of cubes and closest points in $G_m$.

\begin{definition}\label{defFD:CubeCentersAndClosestPoints}
	For each $m$ and $k$, we will call $y_{m,k}$ the center of the cube $Q_{m,k}$, and $\hat y_{m,k}$ will denote the unique element of $G_m$ so that
	\begin{align*}
		d(y_{m,k},G_m)=\abs{y_{m,k}-\hat y_{m,k}}.
	\end{align*}
\end{definition}

For $f:\real^N\to\real$, we can now define a polynomial used to approximate it:

\begin{definition}\label{defFD:PolymonialForWhitney}
	Using the discrete gradient, $\grad_m^1 f$ in (\ref{eqFD:DiscreteGradient}), we define a first order polynomial depending on $f$, $m$, $k$, as
\begin{align*}
	\text{for}\ x\in Q_{m,k},\ \   P^1_{f,k}(x)=f(\hat y_{m,k}) + \grad_m^1 f(\hat y_{mk})\cdot(x-\hat y_{m,k}).
\end{align*}
Given any $f$, we denote the $m$-level truncation, $\tilde f_m$ as
\begin{align*}
	\tilde f_m = f\Indicator_{B_{2^m}}.
\end{align*}
\end{definition}

With all of these ingredients in hand, we can define the Whitney extensions that we will use.

\begin{definition}\label{defFD:WhitneyExtension}
	Using the notation of Definition \ref{defFD:PolymonialForWhitney}, and partition of unity, $(\phi_{m,k})$, in Proposition \ref{propFD:ParitionOfUnitymk},
the zero order Whitney extension is
\begin{align*}
	E^0_m(f,x)= 
	\begin{cases}
		\tilde f_m(x)\ &\text{if}\ x\in G_m,\\
		\sum_{k\in\Natural}\tilde f_m(\hat y_{m,k})\phi_{m,k}(x)\ &\text{if}\ x\not\in G_m,
	\end{cases}
\end{align*}
and the first order Whitney extension is
\begin{align*}
	E^1_m(f,x)=
	\begin{cases}
		\tilde f_m(x)\ &\text{if}\ x\in G_m,\\
		\sum_{k\in\Natural} P^1_{\tilde f_{m,k}}(x)\phi_{m,k}(x)\ &\text{if}\ x\not\in G_m.
	\end{cases}
\end{align*}

\end{definition}

\subsection{The finite dimensional approximation}

As mentioned above, we give an approximation procedure and min-max formula for generic operators acting on convex subsets of $C^{1,\dini}(\real^N)$.  We will call these operators, $J:X_\rho\to C^0(\real^N)$, where the Banach space, $X_\rho$ appears below, in Definition \ref{defFD:XrhoSpace}.  Our particular interest is the eventual application of this material to the operator $I$ defined in (\ref{eqIN:BulkEqForHSOperator}) and (\ref{eqIN:defHSOperator}).

The spaces that are used for the domain of the operators, $J$, are given here.
\begin{definition}\label{defFD:XrhoSpace}
	\begin{align*}
		X_{\rho}= \left\{  f\in C^{1,\dini}(\real^N) \ :\ \exists\  C_f,\ \textnormal{s.t.}\ \abs{\grad f(x)-\grad f(y)}\leq C_f\rho(\abs{x-y})   \ \textnormal{ for all } x,y \in \real^N\right\}.
	\end{align*}
	\begin{align*}
		X_{\rho,x}=\left\{ f\in X_\rho \ :\ \exists\ C_f,\ \textnormal{s.t.}\  \abs{f(y)-f(x)}\leq C_f\abs{y-x}\rho(\abs{y-x}) \ \textnormal{ for all } y \in \real^N \right\}.
	\end{align*}
\end{definition}

We note that $X_\rho$ is a Banach space with the usual norm on $C^1$ combined with the additional $\dini$ semi-norm
\begin{align*}
	[\grad f]_\rho = \inf_{C}\{\sup_{x,y}\abs{\grad f(x)-\grad f(y)}\leq C\rho(\abs{x-y}\},
\end{align*}  
see \cite[Chapter VI, Cor 2.2.3 and Exercise 4.6]{Stei-71}.  Furthermore, $X_{\rho,x}$ is a subspace of $X_\rho$ consisting of those functions vanishing with a rate at $x$.

\begin{rem}\label{remFD:ModulusIsAlsoInXRho}
	We note that $f\in X_\rho$ if and only if
\begin{align*}
	\forall\ x,y\in\real^N,\ \ \abs{f(x+y)-f(x)-\grad f(x)\cdot y}\leq C_f\abs{y}\rho(\abs{y}).
\end{align*}
Without loss of generality, $\rho$ can be chosen so that  $\tilde \rho(y)=\norm{f}_{L^\infty}\abs{y}\rho(\abs{y})$ satisfies $\tilde \rho\in X_\rho$.  This means that whenever $f\in \K(\del,L,m,\rho)$, we have that $\psi(y)=\del+ \norm{f}_{L^\infty}\abs{y}\rho(\abs{y})$ satisfies $\psi\in\K(\del,L,m,\rho)$.
\end{rem}

The first step in making operators with finite rank is to first restrict input functions to the finite set, $G_m$.  So, we define the restriction operator,

\begin{align}\label{eqFD:DefOfTmRestrictionOperator}
	T_m: C^0(\real^N)\to \real^{G_m},\ \ \ T_m f := f|_{G_m}.
\end{align}

Thus, we can use the restriction operator to create a projection of $X_\rho$ onto a finite dimensional subspace of functions depending only on their values on $G_m$:
\begin{align}\label{eqFD:DefOfProjectionOperatorOntoGm}
	\pi_m = E^1_m\circ T_m : X_\rho \to X_\rho.
\end{align}
One of the reasons for using the Whitney extension to define $E^1$ is that operators such as $\pi_m$ will be Lipschitz, and with a norm that is independent from $G_m$.

\begin{theorem}[Stein Chapter VI result 4.6 \cite{Stei-71}]\label{thmFD:ProjectionIsLinearAndBounded}
	$E^0_m$ is linear, and if $g$ is Lipschitz on $G_m$, then $E^0_mg$ is Lipschitz on $\real^N$ with the same Lipschitz constant. Furthermore,
	$\pi_m$ is linear and, for a constant, $C>0$ that depends only on dimension, for all $f\in X_\rho$, 
	\begin{align*}
		\norm{\pi_m f}_{X_\rho}\leq C\norm{f}_{X_\rho}
	\end{align*}
\end{theorem}

On top of the boundedness of $\pi_m$, we have intentionally constructed the sets $G_m$, the cubes $\{Q_{m,k}\}_k$, and the partition functions $\phi_{m,k}$, to respect translations over $G_m$. 

\begin{definition}\label{defFD:TranslationOperator}
	For $f:\real^N\to \real$, and $z\in\real^N$, we define the translation operator $\tau_z$ as
	\begin{align*}
		\tau_zf(x) = f(x+z).
	\end{align*}
\end{definition}

 In particular, property (4) of Proposition \ref{propFD:ParitionOfUnitymk} gives the following translation invariance of $\pi_m$. 

\begin{lemma}[Proposition 4.14 of \cite{GuSc-2019MinMaxEuclideanNATMA}]\label{lemFD:Pi-mTranslationInvariance}
	If $f:\real^N\to\real$,  $z\in G_m$, fixed, and $\tau_z$ in Definition \ref{defFD:TranslationOperator}, then
	\begin{align*}
		\pi_m( \tau_z f) = \tau_z\left(\pi_m f\right),\ \ \text{and}\ \ E^0_m\circ T_m(\tau_z f) = \tau_z \left(E^0_m\circ T_m f\right).
	\end{align*}
\end{lemma}

With these nice facts about the projection operator, $\pi_m$, we can thus define our approximating operators to $J$, in which the approximates have finite rank.

\begin{definition}\label{defFD:DefOfFDApprox}
	Given $J$ that is a Lipschitz mapping of $X_\rho\to C^0_b(\real^N)$ the finite dimensional approximation, $J_m$, is defined as
\begin{align}\label{eqFD:DefOfFdApprox}
	J^m:= E^0_m\circ T_m\circ J \circ E^1_m\circ T_m= E^0_m\circ T_m\circ J\circ \pi_m,
\end{align}
where $E^0$ and $E^1$ appear in Definition \ref{defFD:WhitneyExtension}, $T_m$ is defined in (\ref{eqFD:DefOfTmRestrictionOperator}), and $\pi_m$ is defined in (\ref{eqFD:DefOfProjectionOperatorOntoGm}).
\end{definition}

Below, we will see $J^m$ are Lipschitz maps.  It will also matter in which way $J^m\to J$; for our purposes, it is enough that these approximate operators converge pointwise to $J$ over $X_\rho$, in the following sense.

\begin{proposition}[Corollary 5.20 of \cite{GuSc-2019MinMaxEuclideanNATMA}]\label{propFD:ConvergenceOFApprox}
	For all $f\in X_\rho$, for each $R>0$,
	\begin{align*}
		\lim_{m\to\infty}\norm{J^m(f)-J(f)}_{L^\infty(B_R)}=0.
	\end{align*}
\end{proposition}

A property that was observed in \cite{GuSc-2019MinMaxNonlocalTOAPPEAR} and also used in \cite{GuSc-2019MinMaxEuclideanNATMA} is the ``almost'' preservation of ordering by the projections, $\pi_m$.  Although ordering is, in general, not preserved, on functions that are regular enough, there is a quantifiable error term.  We record this here because it plays a fundamental role later on, in Section \ref{sec:ProofOfStructureThm}, to preserve certain estimates.  In particular, we eventually focus on the fact that our operators have an extra structure called the global comparison property (see Definition \ref{defFD:GCP}), and so whenever $J$ enjoys the global comparison property, then $J_m$ almost enjoys the global comparison property, up to a quantifiable error term over a large enough subspace of $X_\rho$.  The main ingredient to this end is the following lemma.

\begin{lemma}[Lemma 4.17 of \cite{GuSc-2019MinMaxEuclideanNATMA}]\label{lemFD:RemainderForNonNegPi}
	If $w\in C^{1,\al}(\real^N)$, $x_0\in G_m$, $w\geq0$, $w(x_0)=0$, then there exists a function, $R_{\al,m,w,x_0}\in C^{1,\al/2}(\real^N)$ with $R_{\al,m,w,x_0}(x_0)=0$,
	\begin{align*}
		\forall\ x\in\real^N,\ \pi_mw(x)+R_{\al,m,w,x_0}(x)\geq 0,
		\ \ \ \textnormal{and}\ \ \ 
		\norm{R_{\al,m,w,x_0}}_{C^{1,\al/2}}\leq C h_n^\beta\norm{w}_{C^{1,\al}},
	\end{align*}
	where $\beta\in(0,1)$ and depends upon $\al$.
\end{lemma}

\begin{rem}\label{remFD:JAlsoLipOnC1Al}
	If $J:X_\rho\to C^0_b$ is Lipschitz, then for any modulus, $\om$ so that $\om\leq\rho$, $J$ is also a Lipschitz mapping on $X_\om$.  In particular, for all $\al\in(0,1)$, such a $J$ is a Lipschitz mapping on $C^{1,\al}_b(\real^N)$.
\end{rem}

\subsection{A subset of ``supporting'' linear operators, $\D_J$}

The main reason for using the approximating operators, $J_m$, is that as maps that have finite rank, they are effectively maps on a finite dimensional space and hence are differentiable at almost every $f\in X_\rho$.  Furthermore, this a.e. $f$ differentiability endows them with a natural min-max structure.  It turns out that taking limits of ``linearizations'' of $J_m$ produces a rich enough family to construct a min-max representation for the original $J$.  That is the purpose of this subsection.

First, we have some notation for the set of ``supporting'' differentials of maps on $X_\rho$.  The first is simply the collection of limits of derivatives of a map that is differentiable almost everywhere.

\begin{definition}[Differential Set Almost Everywhere]
	If $J$ is differentiable a.e. $X_\rho$, we call the differential set, 
	\begin{align*}
		\D J = \textnormal{c.h.} \{ L=\lim_k DJ[f_k;\cdot]\ :\ f_k\to f\ \text{and}\ J\ \text{is differentiable at}\ f \},
	\end{align*}
	where we used the abbreviation ``\textnormal{c.h.}'' to denote the convex hull.
	Here $DJ[f;\cdot]$ is the derivative of $J$ at $f$.
\end{definition}

This is used to build a weaker notion of ``differential'' set that we will use later, which is the limits of all derivatives of approximating operators.

\begin{definition}[Weak Differential Set]\label{defFD:DJLimitingDifferential}
	For $J:X_\rho\to C^0_b(\real^N)$, we can define a weak differential set as the following:
\begin{align}\label{eqFD:DJLimitingDifferential}
	\D_J = \textnormal{c.h.} \{ L\ :\  \exists m_k,\ L_{m_k}\in \D J^{m_k}\ \text{s.t.}\ \forall\ f\in X_\rho,\ 
	\lim_{k\to\infty} L_{m_k}(f,\cdot)= L(f, \cdot)  \},
\end{align}
where we used the abbreviation ``\textnormal{c.h.}'' to denote the convex hull.  Here, $J^{m_k}$ are the approximating operators for $J$ that are given in Definition \ref{defFD:DefOfFDApprox}.
\end{definition}

\begin{lemma}\label{lemFD:TransInvariant}
	If $J: X_\rho\to C^0_b(\real^n)$ is Lipschitz and translation invariant, then so are $J^m$, and all $L\in\D J^m$ enjoy a bound which is the Lipschitz norm of $J^m$.
\end{lemma}

\begin{proof}[Main idea of proof of Lemma \ref{lemFD:TransInvariant}]
	We do not give all the details here, but simply comments on a few points.  First of all, the Lipschitz nature of $J^m$ is evident from that of $J$ and Theorem \ref{thmFD:ProjectionIsLinearAndBounded} (translation invariance is not used here).  Furthermore, as $J^m$ is a Lipschitz function on a finite dimensional space, we see that all $L\in\D J^m$ must be realized as limits of derivatives of $J^m$.  However, it is easily checked that the operator norm of any differential is bounded by the Lipschitz norm of the original operator, hence the claim about $L\in \D J^m$.  Finally, we need to address the translation invariance of $J^m$ and $L$.  This follows immediately from the translation invariance properties of the projection and extension operators listed in Lemma \ref{lemFD:Pi-mTranslationInvariance}.  Furthermore, again, this translation invariance will also be inherited by any derivative of $J^m$ and hence $L\in \D J^m$.
\end{proof}

The reason that the set, $\D_J$, is useful for our purposes is that it gives a sort of ``maximal'' mean value inequality, which is just a variant on the usual mean value theorem (cf. Lebourg's Theorem in \cite{Clarke-1990OptimizationNonsmoothAnalysisSIAMreprint}).

\begin{lemma}[Lemma 5.2 and Remark 5.4 of \cite{GuSc-2019MinMaxEuclideanNATMA}]\label{lemFD:MVmaxProperty}
	If $\K$ is a convex subset of $X_\rho$ and $J: \K\to C^0_b(\real^N)$ is Lipschitz, then 
	\begin{align*}
	 	\forall\ f,g\in\K,\ \    J(f) - J(g)\leq \max_{L\in \D_J} L(f-g),
	\end{align*}
	where $\D_J$ is from Definition \ref{defFD:DJLimitingDifferential}.
\end{lemma}

\begin{proof}[Sketch of Lemma \ref{lemFD:MVmaxProperty}]
	We note more careful details are given in \cite[Section 5]{GuSc-2019MinMaxEuclideanNATMA}, and so we just give the main idea. Given $f,g\in\K$, the usual Mean Value theorem of Lebourg \cite{Clarke-1990OptimizationNonsmoothAnalysisSIAMreprint} shows that there exists $t\in[0,1]$ and $z=tf+(1-t)g$ with the property that there is at least one $L\in \D_J(z)$ (the differential only at $z$) with the property that
	\begin{align*}
		J(f)-J(g)= L(f-g).
	\end{align*}
	Hence taking that maximum gives the result.  The actual result requires a small amount more detail in the invocation of Lebourg's mean value theorem, which is presented in \cite[Section 5]{GuSc-2019MinMaxEuclideanNATMA}.

\end{proof}

From this mean value inequality, a generic min-max formula for $J$ becomes immediate.

\begin{corollary}\label{corFD:GenericMinMaxForJ}
	Given  a convex subset $\K\subset X_\rho$, and $J:\K\to C^0_b(\real^N)$ that is Lipschitz, $J$ can be realized in the following way:
	\begin{align*}
		\forall\ f\in \K,\ \ 
		J(f,x) = \min_{g\in \K}\max_{L\in \D_J} J(g,x) + L(f-g,x),
	\end{align*}
	where $\D_J$ is from Definition \ref{defFD:DJLimitingDifferential}.
\end{corollary}

\begin{proof}[Proof of Corollary \ref{corFD:GenericMinMaxForJ}]
	For generic $f,g\in X_\rho$, we can utilize Lemma \ref{lemFD:MVmaxProperty}, and then taking the minimum over all $g\in X_\rho$ yields the claim.

\end{proof}

The next result needs a feature we call the global comparison property.

\begin{definition}\label{defFD:GCP}
	We say that $J:X_\rho\to C^0(\real^N)$ obeys the global comparison property (GCP) provided that for all $f,g\in X_\rho$ and $x_0$ such that $f\leq g$ and $f(x_0)= g(x_0)$, $J$ satisfies $J(f,x_0)\leq J(g,x_0)$.
\end{definition}

In the case that $J$ enjoys the GCP, more can be said.  This is one of the main results from \cite{GuSc-2019MinMaxEuclideanNATMA} and \cite{GuSc-2019MinMaxNonlocalTOAPPEAR}.

\begin{theorem}[Theorem 1.11 in \cite{GuSc-2019MinMaxEuclideanNATMA}, Theorem 1.6 \cite{GuSc-2019MinMaxNonlocalTOAPPEAR}]\label{thmFD:StructureOfJFromGSEucSpace}
	If $\K$ is a convex subset of $X_\rho$ and $J:\K\to C^0_b(\real^N)$ is such that
	\begin{enumerate}[(i)]
		\item $J$ is Lipschitz
		\item $J$ is translation invariant
		\item $J$ enjoys the GCP
		\item there exists a modulus, $\om$, with $\lim_{R\to\infty}\om(R)=0$ and
	\begin{align}\label{eqFD:ExtraModulusConditionOutsideBR}
		\forall\ f,g\in \K,\ \text{with}\ f\equiv g\ \text{in}\ B_{2R},\ 
		\norm{J(f)-J(g)}_{L^\infty(B_R)}\leq \om(R)\norm{f-g}_{L^\infty(\real^N)},
	\end{align}
	\end{enumerate}
	
	then for each $L\in \D_J$, there exists the following parameters that are independent of $x$:
	\begin{align*}
		c_L\in\real,\ b_L\in \real^N,\ \mu_L\in\textnormal{measures}(\real^n\setminus \{0\}),
	\end{align*}
	such that for all $f$,
	\begin{align*}
		L(f,x) = c_Lf(x) + b_L\cdot \grad f(x) + \int_{\real^n}\del_h f(x)\mu_L(dh),
	\end{align*}
	and $J$ can be represented as
	\begin{align*}
		\forall\ f\in \K,\ \ J(f,x) = \min_{g\in\K}\max_{L\in \D_J} J(g,x) + L(f-g,x).
	\end{align*}
	Here, for some appropriate, fixed, $r_0$, depending upon $J$, we use the notation
	\begin{align*}
		\del_h f(x) = f(x+h)-f(x)-\Indicator_{B_{r_0}}(h) \grad f(x)\cdot h.
	\end{align*}
	Furthermore, for a universal $C>0$, we have 
	\begin{align*}
		\sup_{L\in\D_J}  \left\{ \abs{c_L} + \abs{b_L} + \int_{\real^N}\min\{\abs{h}\rho(\abs{h}),1\}\mu_L(dh)\right\}
		\leq C\norm{J}_{Lip,\ X_\rho\to C^0_b}.
	\end{align*}
\end{theorem}

\begin{rem}
	Generically, $r_0$ can be taken as $r_0=1$, allowing for a change to each of the corresponding $b_L$, but in our context, it is more natural to choose $r_0$ depending on $J$.
\end{rem}

\begin{proof}[Comments on the proof of Theorem \ref{thmFD:StructureOfJFromGSEucSpace}]
	As the way Theorem \ref{thmFD:StructureOfJFromGSEucSpace} is stated does not match exactly the statements of those in  \cite[Theorem 1.11]{GuSc-2019MinMaxEuclideanNATMA} or \cite[Theorem 1.6]{GuSc-2019MinMaxNonlocalTOAPPEAR},  some comments are in order.  The point is that we explicitly show that the min-max representation for $J$ uses the set of linear mappings, $\D_J$, which is not made explicit in the theorems in \cite{GuSc-2019MinMaxEuclideanNATMA}, \cite{GuSc-2019MinMaxNonlocalTOAPPEAR}.  This is purely a matter of presentation.
	
	By Lemma \ref{lemFD:TransInvariant}, we know that since $J$ is translation invariant, then also all $L\in \D J^m$ are translation invariant.  Taking this fact in hand, and combining it with the analysis that appears in \cite[Section 3]{GuSc-2019MinMaxEuclideanNATMA}, in particular, \cite[Lemma 3.9]{GuSc-2019MinMaxEuclideanNATMA}, we see that all $L\in\D J^m$ have the form claimed here in Theorem \ref{thmFD:StructureOfJFromGSEucSpace}.  The passage from operators in $\D J^m$ to $\D_J$ and the preservation of their structure follows in the same way as in \cite[Section 5]{GuSc-2019MinMaxEuclideanNATMA}.  We note that the structured imparted on $L\in \D J^m$ by the fact that $L$ is an operator that is translation invariant and enjoys the GCP allows us to remove any requirement of \cite[Assumption 1.4]{GuSc-2019MinMaxEuclideanNATMA} as it pertains to the arguments in \cite[Section 5]{GuSc-2019MinMaxEuclideanNATMA}.

\end{proof}

\begin{rem}
	A curious reader may notice that in \cite{GuSc-2019MinMaxEuclideanNATMA}, all of Theorems 1.9, 1.10, and 1.11 apply to the $J$ that we study herein.  The most relevant two are Theorems 1.10 and 1.11 in \cite{GuSc-2019MinMaxEuclideanNATMA}, and in particular as here $J$ is translation invariant, Theorem 1.10 in \cite{GuSc-2019MinMaxEuclideanNATMA} is much simpler in that there is no requirement for (\ref{eqFD:ExtraModulusConditionOutsideBR}) as we do above.  The reason Theorem 1.10 in \cite{GuSc-2019MinMaxEuclideanNATMA} does not suit us here is subtle, and is based on the fact that we will subsequently require a non-degeneracy property of all of the $L$ used to reconstruct $J$ as a min-max.  In our case this will result from using the approximations $J^m$ as above, and to describe the limits of $L_m\in\D J^m$, we need an extra condition to get some compactness on the nonlocal terms, which is the use of (\ref{eqFD:ExtraModulusConditionOutsideBR}).  The type of non-degeneracy we will need for $L$ will be apparent in Section \ref{sec:ProofOfStructureThm}, and we will add some further discussion later.
\end{rem}

%%%%%%%%%%%%%%%%%%%%%%%%%%%%%%%%%%%%%%%%%%%%%%%%%
%%%%%%%%%%%%%%%%%%%%%%%%%%%%%%%%%%%%%%%%%%%%%%%%%
%%%%%%%%%%%%%%%%%%%%%%%%%%%%%%%%%%%%%%%%%%%%%%%%%
%%%%%%%%%%%%%%%%%%%%%%%%%%%%%%%%%%%%%%%%%%%%%%%%%
%%%%%%%%%%%%%%%%%%%%%%%%%%%%%%%%%%%%%%%%%%%%%%%%%
%%%%%%%%%%%%%%%%%%%%%%%%%%%%%%%%%%%%%%%%%%%%%%%%%
%%%%%%%%%%%%%%%%%%%%%%%%%%%%%%%%%%%%%%%%%%%%%%%%%
%%%%%%%%%%%%%%%%%%%%%%%%%%%%%%%%%%%%%%%%%%%%%%%%%
%%%%%%%%%%%%%%%%%%%%%%%%%%%%%%%%%%%%%%%%%%%%%%%%%

\section{Lipschitz Property of $I$ and $H$}\label{sec:ILipschitz}

First, we will show  that for each fixed choice of parameters, $\del$, $L$, $m$, $\rho$, $I$ is a Lipschitz mapping, from $\K(\del,L,m,\rho)$ to $C^{0}_b(\real^n)$.  The main properties of $H$ are deduced from the more basic operator, $I$, which we study first.   Then, later in the section we will show how the same results follow for $H$.

%%%%%%%%%%%%%%%%%%%%%%%%%%%%%%%%%%%%%%%%%%%%%%%%%
%%%%%%%%%%%%%%%%%%%%%%%%%%%%%%%%%%%%%%%%%%%%%%%%%
%%%%%%%%%%%%%%%%%%%%%%%%%%%%%%%%%%%%%%%%%%%%%%%%%
%%%%%%%%%%%%%%%%%%%%%%%%%%%%%%%%%%%%%%%%%%%%%%%%%

\subsection{The analysis for the operator, $I$}

Because $H$ is defined as a function of two operators that take the form, (\ref{eqIN:defHSOperator}), the key result in proving $H$ is Lipschitz is to prove that $I$ as in (\ref{eqIN:defHSOperator}) is Lipschitz.

\begin{proposition}\label{propLIP:BigILipOnKRho}
	If $I$ is the operator defined via (\ref{eqIN:BulkEqForHSOperator}) and (\ref{eqIN:defHSOperator}), then for each $\del$, $L$, $m$, $\rho$ fixed, $I$ is a Lipschitz mapping,
	\begin{align*}
		I: \K(\del,L,m,\rho)\to C^0_b(\real^n),
	\end{align*}
	and the Lipschitz norm of $I$ depends upon all of $\del$, $L$, $m$, $\rho$.
\end{proposition}

Because of the definition of $I^+$ and $I^-$ using (\ref{eqIN:TwoPhaseBulk}) and (\ref{eqIN:DefIPLusAndMinus}), we see that all of the argument in the domain $D^+_f$ for the operator, $I$ (which is, by definition $I^+$), have direct analogs to the operator $I^-$ and the domain $D_f^-$.  Thus, we state the following as a corollary of the techniques that prove Proposition \ref{propLIP:BigILipOnKRho}, but we do not provide a proof.

\begin{corollary}\label{corLIP:IMinusIsLipschitz}
	The operator, $I^-$, defined in (\ref{eqIN:TwoPhaseBulk}) and (\ref{eqIN:DefIPLusAndMinus}) has the same Lipschitz property as $I$ in Proposition \ref{propLIP:BigILipOnKRho}.
\end{corollary}

Before we can establish Proposition \ref{propLIP:BigILipOnKRho}, we give some more basic results.

\begin{lemma}\label{lemLIP:LittleILipEstPart1}
	For $R_0$ as in Theorem \ref{thm:GreenBoundaryBehavior}, there exists a universal $C>0$ and $\al\in(0,1]$, so that if $\psi\geq 0$, $\psi(0)=0$, $\psi(y)\leq c\abs{y}\rho(y)$, $f\in\K(\del,L,m,\rho)$, and $f+\psi\in\K(\del,L,m,\rho)$, then for $\nu=\nu_f=\nu_{f+\psi}$ and $X_0=(0,f(0))=(0,(f+\psi)(0))$, with $U_f$, $U_{f+\psi}$ as in (\ref{eqIN:BulkEqForHSOperator}),
	\begin{align*}
		&\frac{1}{C}\left(
		\int_{\Gam_f\intersect B^{n+1}_{R_0}(X_0)} \psi(y) \abs{Y-X_0}^{-n-1} dY
		\right)\\
		&\ \ \ \ \leq \partial_\nu U_{f+\psi}(X_0)- \partial_\nu U_f(X_0)\\
		&\ \ \ \ \ \ \ \ \leq C\left(
		R_0^{-\al}\norm{\psi}_{L^\infty(\real^n\setminus B_{R_0})}
		+ \int_{\Gam_f\intersect B^{n+1}_{R_0}(X_0)} \psi(y) \abs{Y-X_0}^{-n-1} dY
		\right).
	\end{align*}
	Recall, $B_R\subset\real^n$ and $B^{n+1}_R(X_0)\subset\real^{n+1}$.
\end{lemma}

\begin{corollary}\label{corLIP:CorOfLipEstPart1}
	With $f$ and $\psi$ as in Lemma \ref{lemLIP:LittleILipEstPart1}, 
	\begin{align*}
		&\frac{1}{C}\left(
		\int_{B_{R_0}} \psi(y) \abs{y}^{-n-1} dy
		\right)\\
		&\ \ \ \ \leq \partial_\nu U_{f+\psi}(X_0)- \partial_\nu U_f(X_0)\\
		&\ \ \ \ \ \ \ \ \leq C\left(
		R_0^{-\al}\norm{\psi}_{L^\infty(\real^n\setminus B_{R_0})}
		+ \int_{B_{R_0}} \psi(y) \abs{y}^{-n-1} dy
		\right),
	\end{align*}
	where the integration occurs over $\real^n$ instead of $\Gam_f$.
\end{corollary}

\begin{rem}
	The exponent, $\al$, in Lemma \ref{lemLIP:HEnjoysLittleIEstPart1} and Corollary \ref{corLIP:CorOfLipEstPart1} is the same exponent that appears in the second part of Proposition \ref{propGF:PoissonKernel}, from (\ref{eqGF:DecayOnMassOfPoissonKernel}).
\end{rem}

First we note how the corollary follows from Lemma \ref{lemLIP:LittleILipEstPart1}.

\begin{proof}[Proof of Corollary \ref{corLIP:CorOfLipEstPart1}]
	Because $\Gam_f$ is a $C^{1,\dini}$ graph, we know that up to a constant (depending on only the Lipschitz norm of $f$),
		\begin{align*}
			&\frac{1}{C}\int_{\Gam_f\intersect \left( B_{R_0}(X_0)   \right)} \psi(y)\abs{X_0-Y}^{n+1}dY\\
			&\leq \int_{B_{R_0}(0)} \psi(h)\abs{h}^{-n-1}dh\\
			&\leq C\int_{\Gam_f\intersect \left( B_{R_0}(X_0)   \right)} \psi(y)\abs{X_0-Y}^{n+1}dY,
		\end{align*} 
		and we emphasize that the first and third integrals occur on the set $\Gam_f$, whereas the second integral is over a subset of $\real^n$.
\end{proof}

\begin{proof}[Proof of Lemma \ref{lemLIP:LittleILipEstPart1}]
	This lemma uses, via the fact that $\psi\geq 0$, a sort of ``semigroup'' property of $U_f$ (recall $U_f$, $U_{f+\psi}$ are as in (\ref{eqIN:BulkEqForHSOperator})).  In particular, since $D_f\subset D_{f+\psi}$, we can decompose $U_{f+\psi}$ as the following
	\begin{align*}
		U_{f+\psi} = U_f+W,
	\end{align*}
	where $W$ is the unique solution of
	\begin{align*}
		\begin{cases}
			\Delta W = 0\ &\text{in}\ D_f\\
			W=0\ &\text{on}\ \{x_{n+1}=0\}\\
			W=U_{f+\psi}|_{\Gam_f}\ &\text{on}\ \Gam_f.
		\end{cases}
	\end{align*}	
	We can invoke the linear growth of $U_{f+\psi}$ away from $\Gam_{f+\psi}$ given in Lemma \ref{lemGF:LinearGrowthFromGammaF} to see that
	\begin{align}
		\forall\  Y=(y,y_{n+1})=(y,f(y))\in\Gam_f,\ \ \frac{\psi(y)}{C}
		\leq U_{f+\psi}(Y)\leq C\psi(y). \label{eqLIP:LittleIEstLinearGrowthUf}
	\end{align}

	Now, we can fix $0<s<<1$ and use the Poisson kernel, $P_f$, to evaluate $U_{f+\psi}(X_0+s\nu(X_0))$ (and we recall that $X_0=(0,f(0))$). 
	We first show the details of the next argument as they pertain to the lower bound.  The argument for the upper bound follows by analogous arguments, invoking the upper bound on $U_{f+\psi}(Y)$, given previously.  We will also use the boundary behavior of $P_f$ given in Proposition \ref{propGF:PoissonKernel} (the lower bound in $B_{R_0}(X_0)$ here, and the upper bound for the analogous upper bound argument on $U_{f+\psi}$).
	Thus, we can estimate:
	\begin{align}
		&U_{f+\psi}\left(X_0+s\nu_f(X_0)\right) \nonumber \\
		&= U_f(X_0+s\nu_f(X_0)) + W(X_0+s\nu_f(X_0)) \nonumber \\
		&= U_f(X_0+s\nu_f(X_0))
		+ \int_{\Gam_f} U_{f+\psi}|_{\Gam_f}(Y) P_f\left(X_0+s\nu_f(X_0), Y \right)dY \nonumber \\
		&\geq U_f(X_0+s\nu_f(X_0)) 
		+ \int_{\Gam_f} \frac{\psi(y)}{C}P_f\left(X_0+s\nu_f(X_0), Y \right)dY \label{eqLIP:RefLine1InLittleIEstPart1}  \\
		&\geq U_f(X_0+s\nu_f(X_0)) 
		+ \int_{\Gam_f\intersect B^{n+1}_{R_0} (X_0)}  \tilde C s\psi(y)\abs{X_0-Y}^{-n-1}  dY, \nonumber
	\end{align}
	where in the second to last line, we invoke the estimate of Lemma \ref{lemGF:LinearGrowthFromGammaF} as in (\ref{eqLIP:LittleIEstLinearGrowthUf}).
	(We have used $\nu_f$ as the inward normal derivative to $D_f$ and we recall the notation $Y=(y,y_{n+1})$, as well as $R_0$ originating in Proposition \ref{propGF:PoissonKernel}.)  We note the use of the assumption that $\psi(y)\leq c\abs{y}\rho(\abs{y})$ in order that the following integral is well defined:
	\begin{align*}
		\int_{\Gam_f\intersect B^{n+1}_{R_0}(X_0)}\psi(y)\abs{X_0-Y}^{-n-1}dY.
	\end{align*}

	Thus, since $U_{f+\psi}(0,f(0))=0=U_f(0,f(0))$, as well $\nu_{f+\psi}(X_0)=\nu_f(X_0)$ (as $\grad (f+\psi)(0)=\grad f(0)$), we see that by rearranging terms, dividing by $s$, and taking $s\to0$ (with an abuse of the use of the constant, $C$)
	\begin{align*}
		\partial_\nu U_{f+t\psi}(X_0) - \partial_\nu U_f(X_0)\geq 
		C\left( \int_{\Gam_f\intersect B^{n+1}_{R_0}(X_0)} \psi(y)\abs{X_0-Y}^{-n-1}  dY
		\right).
	\end{align*}
	
	Now, we mention the minor modification to obtain the upper bound.   Working just as above, we can start at the upper bound analog of line (\ref{eqLIP:RefLine1InLittleIEstPart1}), and then we invoke Proposition \ref{propGF:PoissonKernel}, both the pointwise estimates in $B_{R_0}$ and the integral estimate in $B_{R_0}^C$ in (\ref{eqGF:DecayOnMassOfPoissonKernel}).  This yields:
	\begin{align*}
		&U_{f+\psi}\left(X_0+s\nu_f(X_0)\right)\\
		&\leq U_f(X_0+s\nu_f(X_0)) 
		+ \int_{\Gam_f} \frac{\psi(y)}{C}P_f\left(X_0+s\nu_f(X_0), Y \right)dY\\
		&\leq U_f(X_0+s\nu_f(X_0)) 
				+ \int_{\Gam_f\intersect B^{n+1}_{R_0}(X_0)}  \tilde C s\psi(y)\abs{X_0-Y}^{-n-1}  dY\\
		&\ \ \ \ \ \ \ \ \ \ + \int_{\Gam_f\setminus B^{n+1}_{R_0}(X_0)} \norm{\psi}_{L^\infty(\real^n\setminus B_{R_0})} P_f(X_0+s\nu,Y)dY\\
		&\leq U_f(X_0+s\nu_f(X_0)) 
				+ \int_{\Gam_f\intersect B^{n+1}_{R_0}(X_0)}  \tilde C s\psi(y)\abs{X_0-Y}^{-n-1}  dY
				+ \frac{Cs\norm{\psi}_{L^\infty(\real^n\setminus B_{R_0})}}{R_0^\al}.
	\end{align*}
	
	The upper bound concludes as the lower bound, and this finishes the proof of the lemma.

\end{proof}

\begin{lemma}\label{lemLIP:LittleILipEstPart2}
	There exists a universal $C>0$ and $\ep_2>0$ so that if $\psi(0)=0$, $\abs{\grad\psi}\leq \ep_2$, $f\in\K(\del,L,m,\rho)$, and $f+\psi\in\K(\del,L,m,\rho)$, then for $X_0=(0,f(0))$,
	\begin{align*}
		\abs{\partial_{\nu_{f+\psi}}U_{f+\psi}(X_0) - \partial_{\nu_f}U_f(X_0)}
		\leq C\abs{\grad\psi(0)} + C\ep_2\norm{\psi}_{L^{\infty}}.
	\end{align*}
\end{lemma}

\begin{proof}[Proof of Lemma \ref{lemLIP:LittleILipEstPart2}]
	The main part of this proof is to use a rotation to reduce to the case of Lemma \ref{lemLIP:LittleILipEstPart1}.  Let $\R$ be the unique rotation that satisfies 
	\begin{align*}
	\R(\nu_{f+\psi}(X_0))=\nu_f(X_0)
	\end{align*}
	 and leaves 
	 \begin{align*}
	 \left(\textnormal{span}\{\nu_{f+\psi}(X_0), \nu_{f}(X_0\}\right)^{\perp} 
	 \end{align*}
	 unchanged.  Then we can define for a yet to be chosen cutoff function, $\eta$, the transformation $T$
	\begin{align*}
		T:\real^{n+1}\to\real^{n+1},\ \     T(X) = X_0 + \eta(\abs{X-X_0})\R (X-X_0) + (1-\eta(\abs{X-X_0}))(X-X_0)
	\end{align*}
	We compose this mapping with $U_{f+\psi}$ to define an auxiliary function,
	\begin{align*}
		V(X) = (U_{f+\psi}\circ T^{-1})(X).
	\end{align*}
	If the parameter, $\ep_2$, in the assumption of the lemma is not too large (depending upon the Lipschitz bound on $f$, which is $m$), the transformation induces a new domain, whose top boundary will still be a graph.   Let $g$ be the unique function which defines the transformed domain, i.e.
	\begin{align*}
		TD_{f+\psi} = D_g.
	\end{align*}
	By construction, we have $\nu_f(X_0)=\nu_g(X_0)$.
	On top of the previous restriction on $\ep_2$, we can choose it smaller so that $\norm{\grad g}_{L^\infty}\leq 2m$.  This means that we can also make a choice of $\eta$ so that 
	\begin{align*}
		T\Gam_{f+\psi}=\Gam_g,\ \ \text{and}\ \ g\in\K(\del/2, L+\del/2, 2m, \tilde \rho),
	\end{align*} 
	where the new modulus, $\tilde\rho$ is simply $\tilde\rho(s)=\rho(Cs)$, for a universal $C$.  Finally, we will enforce that $\eta$ satisfies
	\begin{align}\label{eqLIP:LittleIEstPart2SizeOfEtaAndRotation}
		\eta\equiv1\ \text{in}\ [0,r_0],\ \ \text{and}\ \ r_0=c\norm{\R}\leq c\abs{\grad \psi(0)},
	\end{align}
	which is possible if $\ep_2$ is small enough, depending upon $\del$, $L$, $m$, $\rho$.

	We remark that these restrictions on $\ep_2$ and the choice of $\eta$ will be such that the function $g$, satisfies
	\begin{align}\label{eqLIP:LittleIEstPart2GisCloseToF}
		\abs{f(x)-g(x)}\leq C\ep_2\norm{\psi}_{L^\infty}\abs{x}\tilde \rho(\abs{x}),
	\end{align}
	as by assumption, $\abs{\grad\psi(0)}\leq \ep_2$.

	We will use three steps to estimate 
	\begin{align*}
		\abs{ \partial_{\nu_{f+\psi} } U_{f+\psi}(X_0)-\partial_{\nu_f} U_f(X_0) },
	\end{align*}
	 using the two additional auxiliary functions, $V$ and $U_g$.  We emphasize that $V$ is not harmonic in all of $D_g$.
	
	\underline{Step 1:} 
	\begin{align}\label{eqLIP:LittleIEstPart2Step1Goal}
	\partial_{\nu_{f+\psi}} U_{f+\psi}(X_0) =\partial_{\nu_g} V(X_0) .
	\end{align}
	
	\underline{Step 2:}
	\begin{align}\label{eqLIP:LittleIEstPart2Step2Goal}
		\abs{  \partial_{\nu_g}V(X_0) -  \partial_{\nu_g} U_g(X_0)   } \leq C\abs{\grad \psi(0)}.
	\end{align}

	\underline{Step 3:}
	\begin{align}\label{eqLIP:LittleIEstPart2Step3Goal}
		\abs{ \partial_{\nu_g} U_g(X_0) - \partial_{\nu_f}U_f(X_0)  } \leq C\ep_2\norm{\psi}_{L^\infty} .
	\end{align}

	Step 1 follows by a direct calculation, by the definition of $R$ and that $\nu_g=R\nu_{f+\psi}$.

 	Next, to establish step 2, we will use that fact that once $\eta$ is chosen, depending only on $\ep_2$ and the collection $\del$, $L$, $m$, $\rho$, if $\eta\equiv 1$ on the interval $[0,r_0]$, then $V$ is harmonic in $B_{r_0}(X_0)\intersect D_g$.  We can then compare the respective normal derivatives of $V$ and $U_g$ using a global Lipschitz estimate combined with the comparison principle.  Indeed, both $V$ and $U_g$ enjoy global Lipschitz estimates, for some $C$ that depends only on $\del$, $L$, $m$, $\rho$, 
	\begin{align*}
		\norm{\grad V}_{L^\infty(D_g)},\ \norm{\grad U_g}_{L^\infty(D_g)}\leq C.
	\end{align*}
	Since on the upper part of $\partial (B_{r_0}(X_0)\intersect D_g)$, we have
	\begin{align*}
		V=U_g\equiv 0\ \ \text{on}\ \ B_{r_0}(X_0)\intersect\Gam_g,
	\end{align*}
	it follows from the Lipschitz estimates and (\ref{eqLIP:LittleIEstPart2SizeOfEtaAndRotation}) that 
	\begin{align*}
		\norm{V-U_g}_{L^\infty(\partial(B_{r_0}(X_0)\intersect D_g))}\leq Cr_0\leq Cc\abs{\grad \psi(0)}.
	\end{align*}
	Since the function $V-U_g$ is harmonic in $B_{r_0}(X_0)\intersect D_g$ we can use linearly growing barriers for $C^{1,\dini}$ domains to deduce that for $s>0$ and small enough,
	\begin{align*}
		&\abs{V(X_0+s\nu)-U_g(X_0+s\nu)}\leq Cs\norm{V-U_g}_{L^\infty(B_{r_0}(X_0)\intersect D_g)}\\
		&\leq Cs\norm{V-U_g}_{L^\infty(\partial(B_{r_0}(X_0)\intersect D_g))}\leq s\tilde C r_0\leq s\tilde C \abs{\grad\psi(0)}.
	\end{align*}
	This establishes Step 2 after dividing by $s$ and taking $s\to 0$.  (Note, these are the same type of barriers from Lemma \ref{lemGF:3.2InGW},  and they can be combined with a transformation that flattens $D_g$.)

	Now we finish with Step 3.
	
	We can break up the estimate into two separate parts, for which we define the functions $g_1$ and $g_2$ as
	\begin{align*}
		g_1=\min\{g,f\},\ \ \text{and}\ \ g_2=\max\{g,f\}.
	\end{align*}
	Notice that $\grad f(0)=\grad g(0)=\grad g_1(0)=\grad g_2(0)$, and so we will denote $\nu=\nu_f(X_0)=\nu_g(X_0)$.  
	By construction, it follows that 
	\begin{align*}
		U_{g_1}(X_0+s\nu) - U_f(X_0+s\nu)
		\leq U_g(X_0+s\nu) - U_f(X_0+s\nu)
		\leq U_{g_2}(X_0 + s\nu) - U_f(X_0+s\nu).
	\end{align*}
 The key improvement from this construction is that by the $C^{1,\dini}$ property of $f$, and $g$, owing to (\ref{eqLIP:LittleIEstPart2GisCloseToF}),
	\begin{align*}
		0\leq  g_2(y)-f(y)\leq C\ep_2\norm{\psi}_{L^\infty}\abs{y}\tilde\rho(\abs{y}),
	\end{align*}
	and as noted in the assumptions, we know that the function $\abs{y}\tilde\rho(y)$ is actually in $X_\rho$.  This is useful because $g_2-f$ will be Lipschitz, but may only enjoy a one sided modulus.

	First, we will demonstrate the upper bound that comes from $U_{g_2}$.
	Defining the function, $\tilde \psi$ as
	\begin{align*}
		\tilde\psi=C\ep_2\norm{\psi}_{L^\infty}\abs{y}\tilde\rho(\abs{y}),
	\end{align*}
we see that $\tilde\psi$ satisfies the assumptions of Lemma \ref{lemLIP:LittleILipEstPart1} (recall that we have defined the modulus so that $\abs{y}\rho(y)$ is an element of $X_\rho$).  Thus we have that
\begin{align*}
	0&\leq \partial_\nu U_{g}(X_0)-\partial_\nu U_f(X_0)\\
	&\leq \partial_\nu U_{f+(g_2-f)}(X_0)-\partial_\nu U_f(X_0)\\
	&\leq \partial_\nu U_{f+\tilde\psi}(X_0) - \partial_\nu U_f(X_0) \\
	&\leq C\left(  R_0^{-\al}\norm{\tilde\psi}_{L^\infty}  
	+ \int_{\Gam_f\setminus B_{R_0}(X_0)} \tilde\psi(y)\abs{X_0-Y}^{-n-1}dY \right)\\
	& \leq C\left(  R_0^{-\al}C\ep_2\norm{\psi}_{L^\infty}  
	+ C\ep_2\norm{\psi}_{L^\infty}\int_{\Gam_f\setminus B_{R_0}(X_0)} \tilde\rho(\abs{y})\abs{X_0-Y}^{-n}dY \right)\\
	&\leq \tilde C\ep_2 \norm{\psi}_{L^\infty}.
\end{align*}

The lower bound follows similarly, but we instead use the inequality
\begin{align*}
	f-\tilde\psi=f-C\ep_2\norm{\psi}_{L^\infty}\abs{y}\tilde\rho(\abs{y})\leq g_1\leq f,
\end{align*}
so that 
\begin{align*}
	0\leq \partial_\nu U_f- \partial_\nu U_{g_1}
	\leq \partial_\nu U_f- \partial_\nu U_{f-\tilde\psi}.
\end{align*}
Thus, we can invoke Lemma \ref{lemLIP:LittleILipEstPart1} with $f$ replaced by $f-\tilde\psi$, $\psi=\tilde\psi$, and $f+\tilde\psi$ replaced by $f$.  The rest of the calculation is the same. This completes Step 3 and the proof of the lemma.

\end{proof}

Because the operator, $I$, is translation invariant, it is useful to define an auxiliary operator, fixed at $x=0$.

\begin{definition}\label{defLIP:LittleI}
	The functional, $i$, is defined as
	\begin{align*}
		i: \K(\del,L,m,\rho)\to \real,\ \ i(f):= I(f,0),
	\end{align*}
	and analogously, using (\ref{eqIN:TwoPhaseBulk}) and (\ref{eqIN:DefIPLusAndMinus}), we have
	\begin{align*}
		i^+(f)=i(f)=I^+(f,0)=I(f,0),\ \ \text{and}\ \ i^-(f)=I^-(f,0).
	\end{align*}
\end{definition}

\begin{lemma}\label{lemLIP:IAddAConstant}
	There exists a constant, $C$ depending upon $\del$, $L$, $m$, $\rho$ so that if $0<\ep<\del/2$, is a constant and $f\in\K(\del,L,m,\rho)$, then
	\begin{align*}
		i^+(f)-C\ep\leq i^+(f+\ep)\leq i^+(f)
	\end{align*}
	and
	\begin{align*}
		i^-(f)+C\ep \geq i^-(f+\ep) \geq i^-(f).
	\end{align*}
\end{lemma}

\begin{proof}[Proof of Lemma \ref{lemLIP:IAddAConstant}]
	We note that the restriction on $\ep$ is simply to keep both $i^+(f)$ and $i^-(f)$ well-defined.  If we were working with $i^+$ only, no restriction on $\ep$ would be necessary.  Furthermore, we will only establish the inequalities as they pertain to $i^+$.  The corresponding pair of inequalities for $i^-$ are analogous. 
	
	We first translate the function, $U_{f+\ep}$, down so that it vanishes on $\Gam_f$.  To this end, we define
	\begin{align*}
		V(X):= U_{f+\ep}(x,x_{d+1}+\ep),
	\end{align*}
	so that $V$ is defined in $D^+_f$, and $V=0$ on $\Gam_f$.  As $U_{f+\ep}\leq 1$, we see that $V\leq 1$ on $\{x_{d+1} = 0\}$.  This and the comparison principle implies that $V\leq U_f$ in $D_f$, and hence,
	\begin{align*}
		\partial_\nu V(X_0)\leq \partial_\nu U_f(X_0).
	\end{align*}
	But $\partial_\nu V(0,f(0)) = \partial_\nu U_{f+\ep}(0, f(0)+\ep)=i^+(f+\ep)$.  This establishes the second inequality.
	
	For the first inequality, we note that $U_{f+\ep}$ enjoys a uniform Lipschitz estimate depending on $\del$, $L$, $m$, $\rho$.  Thus, there is a universal $C$ so that, in particular,
	\begin{align*}
		\text{on}\ \{x_{d+1}=0\},\ \ 1-C\ep\leq V\leq 1.
	\end{align*}
	Thus $0\leq (U_f-V)\leq C\ep$ everywhere in $D_f$.  Again, by the universal Lipschitz estimate, we see that 
	\begin{align*}
		0\leq \partial_\nu (U_f-V)(0,f(0))\leq C\ep.
	\end{align*}
	Hence, this shows that
	\begin{align*}
		i^+(f)-i^+(f+\ep)\leq C\ep,
	\end{align*}
	which gives the first inequality of the Lemma.

\end{proof}

Although not used until the next subsection, it will be worthwhile to record a result about $i$ which is an immediate consequence of Corollary \ref{corLIP:CorOfLipEstPart1}.

\begin{lemma}\label{lemLIP:LittleIPluMinusStrictIncreasing}
	If $f$ and $\psi$ are functions as in Lemma \ref{lemLIP:LittleILipEstPart1}, then for the same constants as in Corollary \ref{corLIP:CorOfLipEstPart1},
	\begin{align*}
		\frac{1}{C}\left( \int_{B_{R_0}}\psi(y)\abs{y}^{-n-1}dy  \right)
		\leq i^+(f+\psi)-i^+(f)
		\leq C\left(  R_0^{-\al}\norm{\psi}_{L^\infty(\real^n\setminus B_{R_0})} + \int_{B_{R_0}} \psi(y)\abs{y}^{-n-1}dy  \right),
	\end{align*}
	and
	\begin{align*}
		\frac{1}{C}\left( \int_{B_{R_0}}\psi(y)\abs{y}^{-n-1}dy  \right)
		\leq i^-(f)-i^-(f+\psi)
		\leq C\left(  R_0^{-\al}\norm{\psi}_{L^\infty(\real^n\setminus B_{R_0})} + \int_{B_{R_0}} \psi(y)\abs{y}^{-n-1}dy  \right).
	\end{align*}
	
\end{lemma}

We are now in a position to prove Proposition \ref{propLIP:BigILipOnKRho}.

\begin{proof}[Proof of Proposition \ref{propLIP:BigILipOnKRho}]
	We first note that we will choose parameters, $\ep_1$ and $\ep_2$, depending upon $\del$, $L$, $m$, and $\rho$ so that we establish the proposition whenever 
	\begin{align}\label{eqLIP:ProofBigILipConstraintFMinusG}
		\norm{f-g}_{L^\infty}\leq \ep_1,\ \ \text{and}\ \ \norm{\grad f-\grad g}_{L^\infty}\leq \ep_2.
	\end{align}
	Assuming we have already proved the proposition under this restriction on $f-g$, we see that we can choose the Lipschitz constant to also depend upon $\ep_1$ and $\ep_2$.  Indeed if either $\norm{f-g}>\ep_1$ or $\norm{\grad f-\grad g}>\ep_2$, since $I$ is bounded on $\K(\del,L,m,\rho)$, we see that
	\begin{align*}
		\norm{I(f)-I(g)}_{L^\infty}\leq \norm{I(f)}_{L^\infty}+\norm{I(g)}_{L^\infty}\leq 2C\leq 2C (\ep_1^{-1}\norm{f-g}_{L^\infty}+\ep_2^{-1}\norm{\grad f-\grad g}_{L^\infty})
	\end{align*}
	(as under the assumption on $f-g$, $1<(\ep_1^{-1}\norm{f-g}+\ep_2^{-1}\norm{\grad f-\grad g})$).

	Now, we explain how to choose $\ep_1$ and $\ep_2$ and establish the proposition under (\ref{eqLIP:ProofBigILipConstraintFMinusG}).  We note that with $i$ as in Definition \ref{defLIP:LittleI}, by translation invariance, 
	\begin{align*}
		I(f,x)=i(\tau_x f).
	\end{align*}
	Thus, we will establish that $i$ is Lipschitz.

	Let us assume, without loss of generality that $f(0)>g(0)$. First, we take 
	\begin{align*}
		\ep=f(0)-g(0),
	\end{align*}
	and we define the new function, 
	\begin{align*}
		\tilde g = g+\ep.
	\end{align*}
	Since $f,g\in\K(\del,L,m,\rho)$, we can choose the parameter, $\ep_1<\del/2$, so that
	\begin{align*}
		\tilde g\in \K(\del/2,L,m,\rho).
	\end{align*}
	Next, we take $\ep_2$ to be the parameter from Lemma \ref{lemLIP:LittleILipEstPart2} that corresponds to the set $\K(\del/2,L,m,\rho)$.  Under this assumption, we see that $\psi=\tilde g-f$ satisfies the assumptions of Lemma \ref{lemLIP:LittleILipEstPart2}.  Hence, because by definition $i(f)=\partial_\nu U_f(X_0)$, we see that
	\begin{align*}
		\abs{i(\tilde g)-i(f)}\leq C\abs{\grad(\tilde g-f)(0)} + \ep_2\norm{\tilde g-f}_{L^\infty}
		\leq C(\norm{f-g}_{L^\infty}+\norm{\grad f-\grad g}_{L^\infty}).
	\end{align*}
	Furthermore, Lemma \ref{lemLIP:IAddAConstant} shows that
	\begin{align*}
		\abs{i(\tilde g)-i(g)}\leq C\abs{f(0)-g(0)}\leq C\norm{f-g}_{L^\infty}.
	\end{align*}
	
	This shows that $i$ is Lipschitz, and hence also $I$.

\end{proof}

%%%%%%%%%%%%%%%%%%%%%%%%%%%%%%%%%%%%%%%%%%%%%%%%%
%%%%%%%%%%%%%%%%%%%%%%%%%%%%%%%%%%%%%%%%%%%%%%%%%
%%%%%%%%%%%%%%%%%%%%%%%%%%%%%%%%%%%%%%%%%%%%%%%%%
%%%%%%%%%%%%%%%%%%%%%%%%%%%%%%%%%%%%%%%%%%%%%%%%%

\subsection{Analysis For $H$}

Because of the assumptions on $G$, the following corollary is immediate from Proposition \ref{propLIP:BigILipOnKRho} and Corollary \ref{corLIP:IMinusIsLipschitz}, recalling that $I^+=I$.

\begin{corollary}\label{corLIP:HIsLip}
	For each $\del$, $L$, $m$, $\rho$ fixed, $H$ is a Lipschitz mapping,
	\begin{align*}
		H: \K(\del,L,m,\rho)\to C^0_b(\real^n),
	\end{align*}
	and the Lipschitz norm of $H$ depends upon all of $\del$, $L$, $m$, $\rho$.
\end{corollary}

The results of Lemma \ref{lemLIP:LittleILipEstPart1} and Corollary \ref{corLIP:CorOfLipEstPart1} are also used in building appropriate finite dimensional approximations to $I$ and $H$.  We note that $H$ also enjoys these properties.

\begin{lemma}\label{lemLIP:HEnjoysLittleIEstPart1}
	The results in Lemma \ref{lemLIP:LittleILipEstPart1}, Corollary \ref{corLIP:CorOfLipEstPart1}, and Lemma \ref{lemLIP:LittleIPluMinusStrictIncreasing} hold for the operator, 
	\begin{align}\label{eqLIP:DefLittleH}
		h(f)=H(f,0).
	\end{align}
\end{lemma}

\begin{proof}[Proof of Lemma \ref{lemLIP:HEnjoysLittleIEstPart1}] Since $\nabla\psi(0) = 0$, we have $\nabla(f + t \psi)(0) = \nabla f(0)$ for all $t \geq 0$. Consequently,
$$h(f + \psi) - h(f) = \left(G(i^+(f+\psi),i^-(f+\psi)) - G(i^+(f),i^-(f)) \right)\sqrt{1+\abs{\grad{f}(0)}^2}.$$
We proceed to estimate $G(i^+(f+\psi),i^-(f+\psi)) - G(i^+(f),i^-(f))$. First, observe that $i^+(f + \psi) \geq i^+(f)$ and $i^-(f + \psi) \leq i^-(f)$. By the assumptions on $G$, we thus have
\begin{align*}
& G(i^+(f+\psi),i^-(f+\psi)) - G(i^+(f),i^-(f)) \\
= \ & G(i^+(f+\psi),i^-(f+\psi)) - G(i^+(f),i^-(f+\psi)) + G(i^+(f),i^-(f+\psi)) - G(i^+(f),i^-(f)) \\
\geq \ & \lambda \left(i^+(f+\psi) - i^+(f) \right) + \lambda \left(i^-(f) - i^-(f+\psi) \right).
\end{align*}
Similarly,
$$G(i^+(f+\psi),i^-(f+\psi)) - G(i^+(f),i^-(f))  \leq \Lambda \left(i^+(f+\psi) - i^+(f) \right) + \Lambda \left(i^-(f) - i^-(f+\psi) \right).$$
The claim now follows from Lemma \ref{lemLIP:LittleILipEstPart1} and Corollary \ref{corLIP:CorOfLipEstPart1}, where we note the factor $\sqrt{1+\abs{\grad f(0)}^2}$ is controlled by $(1+m)$, and so can be absorbed into the constant in the resulting inequalities.

\end{proof}

The extension of the results in Lemma \ref{lemLIP:LittleILipEstPart1}, Corollary \ref{corLIP:CorOfLipEstPart1}, and Lemma \ref{lemLIP:LittleIPluMinusStrictIncreasing} to the operator $H$ also implies that the remaining key result above applies to $H$ as well.  We omit the proof, as it follows the same adaptations as in the previous Lemma. 

\begin{lemma}\label{lemLIP:HEnjoysLittleIEstPart2}
	The results in Lemma \ref{lemLIP:LittleILipEstPart2} hold for the operator, $h(f)=H(f,0)$.
\end{lemma}

%%%%%%%%%%%%%%%%%%%%%%%%%%%%%%%%%%%%%%%%%%%%%%%%%
%%%%%%%%%%%%%%%%%%%%%%%%%%%%%%%%%%%%%%%%%%%%%%%%%
%%%%%%%%%%%%%%%%%%%%%%%%%%%%%%%%%%%%%%%%%%%%%%%%%
%%%%%%%%%%%%%%%%%%%%%%%%%%%%%%%%%%%%%%%%%%%%%%%%%
%%%%%%%%%%%%%%%%%%%%%%%%%%%%%%%%%%%%%%%%%%%%%%%%%
%%%%%%%%%%%%%%%%%%%%%%%%%%%%%%%%%%%%%%%%%%%%%%%%%
%%%%%%%%%%%%%%%%%%%%%%%%%%%%%%%%%%%%%%%%%%%%%%%%%
%%%%%%%%%%%%%%%%%%%%%%%%%%%%%%%%%%%%%%%%%%%%%%%%%
%%%%%%%%%%%%%%%%%%%%%%%%%%%%%%%%%%%%%%%%%%%%%%%%%

\section{Proof of Theorem \ref{thm:StructureOfHMain}}\label{sec:ProofOfStructureThm}

Before we can prove Theorem \ref{thm:StructureOfHMain}, we must make a number of observations about how $I$ behaves with respect to some positive perturbations in $X_{\rho,0}$ and especially what this behavior implies for the linear operators in $\D_I$ (recall $\D_I$ is in Definition \ref{defFD:DJLimitingDifferential}, which is applicable since $I$ is Lipschitz).  Then we will show how these properties carry over to $H$, and finally we will collect the ideas to finish the proof of Theorem \ref{thm:StructureOfHMain}.

%%%%%%%%%%%%%%%%%%%%%%%%%%%%%%%%%%%%%%%%%%%%%%%%%
%%%%%%%%%%%%%%%%%%%%%%%%%%%%%%%%%%%%%%%%%%%%%%%%%
%%%%%%%%%%%%%%%%%%%%%%%%%%%%%%%%%%%%%%%%%%%%%%%%%
%%%%%%%%%%%%%%%%%%%%%%%%%%%%%%%%%%%%%%%%%%%%%%%%%

\subsection{Estimates on the L\'evy measures for $I$ and $H$}\label{sec:LevyMeasureEstimate}

We will show that once a Lipschitz operator, $J$, with the GCP enjoys bounds similar to those in Lemma \ref{lemLIP:LittleILipEstPart1} and Corollary \ref{corLIP:CorOfLipEstPart1}, then as a consequence, its resulting linear supporting operators are comparable to a modified 1/2-Laplacian, and subsequently the corresponding L\'evy measures have a density that is comparable to the 1/2-Laplacian.  The main result in this direction is Proposition \ref{propPfThm1:IntegralBoundsForEll}.
Basically, the analysis we use follows almost exactly as some arguments in \cite[Section 4.6]{GuSc-2019MinMaxNonlocalTOAPPEAR} regarding inequalities for extremal operators and linear functionals in the min-max representation.

For this subsection, we will assume that $J$ is an operator as in Section \ref{sec:FiniteDim}, and assume further that $J$ satisfies the assumptions of Theorem \ref{thmFD:StructureOfJFromGSEucSpace} and the conclusion of Lemma \ref{lemLIP:LittleILipEstPart1}.

We can utilize the translation invariance of $J$ to focus on linear functionals via evaluation at $x=0$.

\begin{definition}\label{defPfThm1:DefDHAt0}
	\begin{align}\label{eqPfThm1:DifferentialsAtZero}
		\D_J(0)=\{ \ell\in \left(X_\rho\right)^*\ :\ \exists\  L\in D_J,\ \textnormal{s.t.}\  \forall\ f\in\K,\ \ell(f)=L(f,0)  \}.
	\end{align}
\end{definition}

We will compare the linear support functionals of $J$ to a modified version of the 1/2-Laplacian, which we define here.

\begin{definition}

With the constant, $R_0$ as in Theorem \ref{thm:GreenBoundaryBehavior}, the linear operator, $L_\Delta$, is defined as 
\begin{align*}
	L_{\Delta}(f,x)=\int_{B_{R_0}}\del_hf(x)\abs{h}^{-n-1}dh,
\end{align*}
which is well defined for all $f\in X_{\rho}$.  Note, this is simply the 1/2-Laplacian, but computed with a truncated kernel.
\end{definition}

\begin{lemma}\label{lemPfThm1:FDApproxCompareToFracLaplace}
	Let $R_0$ be the constant in Theorem \ref{thm:GreenBoundaryBehavior}. There exists a constant, $C>0$, so that if $J$ is an operator that satisfies the assumptions of Theorem \ref{thmFD:StructureOfJFromGSEucSpace} as well as the outcome of Lemma \ref{lemLIP:LittleILipEstPart1}, $J^m$ and $\left(L_\Delta\right)^m$ are the finite dimensional approximations to $J$ and $L_\Delta$, defined in (\ref{eqFD:DefOfFdApprox}), then for all $x\in G^m$ and $\psi\in X_{\rho,x}\intersect C^{1,\al}(\real^n)$, with $f+\psi\in\K$ and $\supp(\psi)\subset B_{R_0}(x)$,
	\begin{align*}
		-C h_m^\beta\norm{\psi}_{C^{1,\al}}+
		\frac{1}{C}\left( L_\Delta \right)^m(\psi,x)   \leq J^m(f+\psi,x)-J^m(f,x)
		\leq C\left( L_\Delta \right)^m(\psi,x)
		+C h_m^\beta\norm{\psi}_{C^{1,\al}},
	\end{align*}
	where $h_m$ is the grid size parameter from Definition \ref{defFD:GridSetGm} and $\beta$ is the exponent from Lemma \ref{lemFD:RemainderForNonNegPi}.
\end{lemma}

\begin{proof}[Proof of Lemma \ref{lemPfThm1:FDApproxCompareToFracLaplace}]
	We note that by the translation invariance of both $J$ and $L_\Delta$, it suffices to prove this result for $x=0$ (see Lemma \ref{lemFD:TransInvariant}).  We need to utilize Lemma \ref{lemFD:RemainderForNonNegPi} because we will also use  Lemma \ref{lemLIP:LittleILipEstPart1} and Corollary \ref{corLIP:CorOfLipEstPart1}, which require $\psi$ to be non-negative.  Given such a $\psi$, it is not true in general that $\pi_m\psi\geq0$, but we can correct with a quantifiable remainder.  This is what follows.

	Let $R$ be the function, $R_{\al,m,w,0}$ which results from Lemma \ref{lemFD:RemainderForNonNegPi} when applied to $\psi$.  Let $\tilde\psi_m$ be the function,
	\begin{align*}
		\tilde\psi_m= \pi_m\psi + R,
	\end{align*}
	so that $\tilde\psi_m\in X_{\rho,0}$.

	This means we can apply Corollary \ref{corLIP:CorOfLipEstPart1} to $\pi_mf + \tilde\psi_m$, and this gives
	\begin{align*}
		\frac{1}{C}\int_{B_{R_0}} \tilde\psi_m(y)\abs{y}^{-n-1}dy 
		\leq J(\pi_mf+\tilde\psi_m,0)-J(\pi_mf,0)
		\leq C\int_{B_{R_0}} \tilde\psi_m(y)\abs{y}^{-n-1}dy,
	\end{align*}
	and hence since $\tilde\psi_m(0)=0$ and $\grad\tilde\psi(0)=0$,
	\begin{align*}
		\frac{1}{C}L_\Delta(\tilde\psi_m,0)\leq J(\pi_mf+\tilde\psi_m,0)-J(\pi_mf,0)
		\leq C L_\Delta(\tilde\psi_m,0).
	\end{align*}
	Using the continuity of $L_\Delta$ over $C^{1,\al/2}$ as well as the Lipschitz nature of $J$ over $X_\rho$ and $C^{1,\al}$ (recall that $\pi_m\psi\in X_\rho\intersect C^{1,\al}$ as well as Remark \ref{remFD:JAlsoLipOnC1Al}), we obtain
	\begin{align*}
		-\tilde C\norm{R}_{C^{1,\al/2}}+\frac{1}{C}L_\Delta(\pi_m\psi,0)\leq J(\pi_mf+\pi_m\psi,0)-J(\pi_mf,0)
		\leq C L_\Delta(\pi_m\psi,0) + \tilde C\norm{R}_{C^{1,\al/2}}.
	\end{align*}
	Invoking Lemma \ref{lemFD:RemainderForNonNegPi}, for the parameter, $\beta$ in Lemma \ref{lemFD:RemainderForNonNegPi},
	\begin{align*}
		-\tilde C h_m^\beta\norm{\psi}_{C^{1,\al}}+\frac{1}{C}L_\Delta(\pi_m\psi,0)\leq J(\pi_mf+\pi_m\psi,0)-J(\pi_mf,0)
		\leq C L_\Delta(\pi_m\psi,0) + \tilde C h_m^\beta\norm{\psi}_{C^{1,\al}}.
	\end{align*}
	Finally, using the fact that the operator $E_m^0\circ T_m$ is linear, preserves ordering, and agrees with its input function over $G^m$, we see that applying $E^0_m\circ T_m$ to each of the operators in the last inequality, not evaluated at $x=0$, and then relabeling constants and evaluating at $x=0\in G^m$,
	\begin{align*}
		-C h_m^\beta\norm{\psi}_{C^{1,\al}}+\frac{1}{C}\left(L_\Delta\right)^m(\psi,0)\leq J^m(f+\psi,0)-J^m(f,0)
		\leq C \left(L_\Delta\right)^m(\psi,0) + \tilde C h_m^\beta\norm{\psi}_{C^{1,\al}}.
	\end{align*}

\end{proof}

\begin{corollary}\label{corPfThm1:LComparableToHalfLaplace}
	If $J$ satisfies the assumptions of Theorem \ref{thmFD:StructureOfJFromGSEucSpace} as well as the outcome of Lemma \ref{lemLIP:LittleILipEstPart1}, then for all $L\in\D_J$, the constant $C$ and functions $\psi$ as in Lemma \ref{lemPfThm1:FDApproxCompareToFracLaplace},
	\begin{align*}
		\frac{1}{C}L_\Delta(\psi,x)\leq L(\psi,x) \leq CL_\Delta(\psi,x).
	\end{align*}
\end{corollary}

\begin{proof}[Proof of Corollary \ref{corPfThm1:LComparableToHalfLaplace}]
	We recall that $\D_J$ is a convex hull of limits of linear operators that are derivatives of $J^m$.  Thus, it suffices to prove the result for $f\in X_{\rho}$ so that $f=\lim f_m$ and that $J^m$ is differentiable at $f_m$, i.e.
	\begin{align*}
		\forall\ \psi,\ DJ^m(f_m)[\psi]= \lim_{s\to0}\frac{J^m(f_m+s\psi)-J^m(f_m)}{s},
	\end{align*}
	and
	\begin{align*}
		L=\lim_{m\to\infty} DJ^m(f_m).
	\end{align*}
	Thus, for all $\psi$ satisfying the requirements of Lemma \ref{lemPfThm1:FDApproxCompareToFracLaplace}, we see that
	\begin{align*}
		-Ch_m^\beta\norm{\psi}_{C^{1,\al}} + \frac{1}{C}\left( L_\Delta\right)^m(\psi)
		\leq DJ^m(f_m)[\psi]
		\leq C\left( L_\Delta \right)^m(\psi) + Ch_m^\beta\norm{\psi}_{C^{1,\al}}.
	\end{align*}

We can now take limits as $m\to0$, using that $h_m\to0$ and Proposition \ref{propFD:ConvergenceOFApprox} that shows $(L_\Delta)^m\to L_\Delta$, to conclude the result of the corollary for such $L$, $f_m$, and $f$.  Since these inequalities are stable under convex combinations, we are finished.

\end{proof}

Just as above, thanks to translation invariance, we have the luxury of focusing on all of the operators in $\D_J$ evaluated at $x=0$.  Thus, as an immediate consequence of Corollary \ref{corPfThm1:LComparableToHalfLaplace}, we obtain the following result.

\begin{proposition}\label{propPfThm1:IntegralBoundsForEll}
	For all $\ell\in\D_J(0)$, for $f\in\K$, for $\psi\in X_{\rho,0}$ with $\supp(\psi)\subset B_{R_0}$,
	\begin{align*}
		\frac{1}{C}\int_{B_{R_0}}\psi(y)\abs{y}^{-n-1}dy\leq \ell(\psi) \leq C\int_{B_{R_0}}\psi(y)\abs{y}^{-n-1}dy,
	\end{align*}
\end{proposition}

\begin{corollary}\label{corIUG:LevyMeasures}
	If  $\ell\in\D_J(0)$, and $\mu_\ell$ is the L\'evy measure corresponding to $\ell$ from Theorem \ref{thmFD:StructureOfJFromGSEucSpace}, then there exists a function, $K_\ell$ so that
	\begin{align*}
		\mu_\ell(E) = \int_E K_\ell(h)dh,
	\end{align*}
	and
	\begin{align*}
		\forall\ \ h\in B_{R_0}\setminus \{0\},\ \ 
		\frac{1}{C}\abs{h}^{-n-1}\leq K_\ell(h)\leq C\abs{h}^{-n-1}.
	\end{align*}
\end{corollary}

\begin{proof}[Proof of Corollary \ref{corIUG:LevyMeasures}]
	We recall the structure of $\ell$ from Theorem \ref{thmFD:StructureOfJFromGSEucSpace} and the fact that for any $\psi$ as in Lemma \ref{lemPfThm1:FDApproxCompareToFracLaplace} we have $\psi(0)=0$ and $\grad \psi(0)=0$, so that 
	\begin{align*}
		\ell(\psi) = \int_{\real^n}\psi(h)\mu_\ell(dh).
	\end{align*}
	Hence, for each fixed $r$, from Proposition \ref{propPfThm1:IntegralBoundsForEll} we can already deduce that $\mu_\ell$ has a density in $B_{R_0}\setminus B_{r}$, and that this density must inherit the bounds given in Proposition \ref{propPfThm1:IntegralBoundsForEll}.   Hence the Corollary holds for the measure $\mu_\ell$, restricted to $B_{R_0}(0)\setminus B_r(0)$.  Since $r>0$ was arbitrary, we see that there will be a density on the set $B_{R_0}(0)\setminus \{0\}$ and that the required bounds still follow from Proposition \ref{propPfThm1:IntegralBoundsForEll}.

\end{proof}

%%%%%%%%%%%%%%%%%%%%%%%%%%%%%%%%%%%%%%%%%%%%%%%%%
%%%%%%%%%%%%%%%%%%%%%%%%%%%%%%%%%%%%%%%%%%%%%%%%%
%%%%%%%%%%%%%%%%%%%%%%%%%%%%%%%%%%%%%%%%%%%%%%%%%
%%%%%%%%%%%%%%%%%%%%%%%%%%%%%%%%%%%%%%%%%%%%%%%%%

\subsection{Estimates on the drift}\label{sec:DriftEstimates}

Just as the estimates for the L\'evy measures corresponding to a mapping, $J$, depended upon a variant of the inequality of Lemma \ref{lemLIP:LittleILipEstPart1} being inherited by the finite dimensional approximations, so too will the proof here for the estimate on the drift.  This time, we need a finite dimensional version of Lemma \ref{lemLIP:LittleILipEstPart2}.

\begin{lemma}\label{lemPfThm1:FiniteDimVersionOfLittleIEst2}
	With $C$, $\ep_2$, $f$, $\psi$ as in Lemma \ref{lemLIP:LittleILipEstPart2}, we also have
	\begin{align*}
		\abs{J^m(f+\psi,0)-J^m(f,0)}\leq C\left( \abs{\grad\psi(0)} + \ep_2\norm{\psi}_{L^\infty} \right).
	\end{align*}
\end{lemma}

\begin{proof}[Proof of Lemma \ref{lemPfThm1:FiniteDimVersionOfLittleIEst2}] 
Applying Lemma \ref{lemLIP:LittleILipEstPart2} or Lemma \ref{lemLIP:HEnjoysLittleIEstPart2} to $\pi_mf$ and $\pi_m\psi$, we obtain
$$|J(\pi_m f+\pi_m\psi,0)-J(\pi_mf,0)| \leq C\left(|\nabla (\pi_m\psi)(0)| + \ep_2||\pi_m\psi||_{L^{\infty}} \right).$$
We next apply Theorem \ref{thmFD:ProjectionIsLinearAndBounded} to bound $||\pi_m\psi||_{L^{\infty}}$ with respect to $||\psi||_{L^{\infty}}$. Also, since $\pi_m$ agrees up to first order with its input function on $G_m$, and because $0 \in G_m$, we see that $\nabla (\pi_m\psi)(0) = \nabla \psi(0)$. Finally, using the fact that $E^0_m \circ T_m$ is order-preserving and agrees with its input function on $G_m$, we obtain the desired estimate.
\end{proof}

With this information in hand, we need to address how the drift and L\'evy measures given in Theorem \ref{thmFD:StructureOfJFromGSEucSpace} relate to each other, particularly in the context of the assumptions in Section \ref{sec:BackgroundParabolic}.
To this end, fix $e \in \RN$, $|e| = 1$, and a smooth cutoff function $\eta \in C^{\infty}_c(\RN)$ between $B_{1/2}$ and $B_1$. We define the functions, for $0 < \tau \leq r$,
\begin{align}\label{eqPfThm1:DefOfTruncatedLinearPhi}
	\phi(y) = (e\cdot y)\eta(y)\ \ \ \ \text{and}\ \ \ \ 
\phi_{\tau,r}(y) := \tau r \phi \left(\frac{y}{r} \right).
\end{align}
A crucial property of $\phi_{\tau,r}$ is given in the next lemma.

\begin{lemma}\label{lemPfThm1:BoundsOnEllPhi} There exists a constant, $C$, depending on $\del$, $L$, $m$, $\rho$, such that if 
	\begin{align*}
		\ell\in \D_J(0),\ \ 
		\text{with}\ \ell(f)=c_\ell f(0) + b_\ell\cdot\grad f(0) + \int_{\real^n}\del_h f(0)K_\ell(h)dh,
	\end{align*}
	then for $\phi_{\tau,r}$ defined in (\ref{eqPfThm1:DefOfTruncatedLinearPhi}),
	\begin{align*}
		|\ell(\phi_{\tau,r})| \leq C\tau \quad \text{ for all } \tau \leq r.
	\end{align*}
\end{lemma}

\begin{proof}[Proof of Lemma \ref{lemPfThm1:BoundsOnEllPhi}]
	
	First, we list a number of properties of $\phi_{\tau,r}$.

	\begin{enumerate}
	\item[(i)] $\phi_{\tau,r}(0) = 0$ and $\nabla \phi_{\tau,r}(0) = \tau e$.
	\item[(ii)] There exists a universal constant $C' > 0$ such that $|\nabla \phi_{\tau,r}(y)| \leq C' \tau$ for all $y \in \supp(\phi_{\tau,r})$.
	\item[(iii)] If $\eta$ is $C^{1,\text{Dini}}_{\rho}$, then $\phi_{\tau,r}$ is $C^{1,\text{Dini}}_{\rho}$. Indeed, by the concavity of $\rho$, and since $\tau \leq r$, we see that for any $x \in \RN$ and $y \in B_r(x)$, we have
	\begin{align*}
	& |\phi_{\tau,r}(x+y) - \phi_{\tau,r}(x) - \nabla \phi_{\tau,r}(x) \cdot y| \\
	& = \bigg|(\tau e\cdot (x+y)) \eta\left(\frac{x+y}{r}\right) - (\tau e\cdot x) \eta\left(\frac{x}{r}\right) - (\tau e\cdot y) \eta\left(\frac{x}{r}\right) - \left(\tau e\cdot \frac{x}{r}\right) \left(\nabla \eta\left(\frac{x}{r}\right) \cdot y \right) \bigg| \\
	& = \bigg|(\tau e\cdot (x+y)) \left(\eta\left(\frac{x+y}{r}\right) - \eta \left( \frac{x}{r} \right) \right) - \left(\tau e\cdot x \right) \left(\nabla \eta\left(\frac{x}{r}\right) \cdot \frac{y}{r} \right) \bigg| \\
	& = \bigg|(\tau e\cdot (x+y)) \left(\eta\left(\frac{x+y}{r}\right) - \eta \left( \frac{x}{r} \right) - \nabla \eta\left(\frac{x}{r}\right) \cdot \frac{y}{r}  \right) + \left(\tau e\cdot y \right) \left(\nabla \eta\left(\frac{x}{r}\right) \cdot \frac{y}{r} \right) \bigg| \\
	& \leq \tau |x + y| \bigg|\frac{y}{r} \bigg|\rho\left(\frac{y}{r}\right) + \frac{\tau}{r} ||\nabla \eta||_{L^{\infty}(\RN)} |y|^2 \\
	& \leq |y| \tau \rho\left(\frac{y}{r}\right) + ||\nabla \eta||_{L^{\infty}(\RN)} |y|^2 \\
	& \leq |y| \frac{\tau}{r} \rho(y) + ||\nabla \eta||_{L^{\infty}(\RN)} |y|^2 \\
	& \leq |y|(\rho(y) + C|y|).
	\end{align*}
	\end{enumerate}
	Without loss of generality, we can assume that for $\abs{y}\leq 1$, $\abs{y}\leq \rho(\abs{y})$.

	In order to conclude the bound on $\ell(\phi_{\tau,r})$, we look to Lemma \ref{lemPfThm1:FiniteDimVersionOfLittleIEst2}.  This shows that for all $\ell\in \D_J(0)$, for all $\psi\in X_\rho$
	\begin{align*}
		\abs{\ell(\psi)}\leq C(\abs{\grad \psi(0)} + \ep_2\norm{\psi}_{L^\infty}),
	\end{align*}
	where $C$ is a constant that depends only on $\del$, $L$, $m$, $\rho$.  Thus as $\phi_{\tau,r}\in X_\rho$, applying this to $\psi=\phi_{\tau,r}$ shows $\abs{\ell(\phi_{\tau,r})}\leq C\tau$.

\end{proof}

We are finally ready to prove the estimates on the drift.

\begin{lemma}\label{lemPfThm1:bijBounded}
	There exists a constant, $C$, depending on $\del$, $L$, $m$, $\rho$, such that if $r_0$ and $\del_h f$ are as in Theorem \ref{thmFD:StructureOfJFromGSEucSpace},
	\begin{align*}
		\ell\in \D_J(0),\ \ 
		\text{with}\ \ell(f)=c_\ell f(0) + b_\ell\cdot\grad f(0) + \int_{\real^n}\del_h f(0)K_\ell(h)dh,
	\end{align*}
	then for $0<r<r_0$,
	\begin{align*}
		\abs{b_\ell-\int_{B_{r_0}\setminus B_r} hK_\ell(h)dh}\leq C.
	\end{align*}
\end{lemma}

\begin{proof}[Proof of Lemma \ref{lemPfThm1:bijBounded}]
	Fix $e \in \RN$, $|e| = 1$ and $0 < \tau \leq r  < r_0$. Consider the function $\phi_{\tau,r}$ defined above. We have 
	\begin{align*}
	\ell(\phi_{\tau,r}) & = \tau (b_{\ell}\cdot e) + \int_{\RN} (\tau e \cdot h) \left[\eta\left(\frac{h}{r}\right) - \Indicator_{B_{r_0}}(h) \right] \ K_{\ell}(h) dh \\
	& = \tau\left(b_{\ell}\cdot e + \int_{B_{r_0} \backslash B_{r/2}} (e \cdot h) \left[\eta\left(\frac{h}{r}\right) - 1 \right] \ K_{\ell}(h) dh \right)\\
	& = \tau\left(b_{\ell}\cdot e + \int_{B_{r_0} \backslash B_r} (e \cdot h) \left[\eta\left(\frac{h}{r}\right) - 1 \right] \ K_{\ell}(h) dh  +  \int_{B_r \backslash B_{r/2}} (e \cdot h) \left[\eta\left(\frac{h}{r}\right) - 1 \right] \ K_{\ell}(h) dh \right)\\
	& = \tau\left(b_{\ell}\cdot e - \int_{B_{r_0} \backslash B_r} (e \cdot h) \ K_{\ell}(h) dh  +  \int_{B_r \backslash B_{r/2}} (e \cdot h) \left[\eta\left(\frac{h}{r}\right) - 1 \right] \ K_{\ell}(h) dh \right)\\
	\end{align*}
	Consequently,
	$$\tau\left(b_{\ell}\cdot e - \int_{B_{r_0} \backslash B_r} (e \cdot h) \ K_{\ell}(h) dh \right) = \ell(\phi_{\tau,r}) +  \tau \int_{B_r \backslash B_{r/2}} (e \cdot h) \left[1 - \eta\left(\frac{h}{r}\right) \right] \ K_{\ell}(h) dh.$$
	Using Lemma \ref{lemPfThm1:BoundsOnEllPhi}, we have
	$$\tau\bigg|b_{\ell}\cdot e - \int_{B_{r_0} \backslash B_r} (e \cdot h) \ K_{\ell}(h) dh \bigg| \leq C\tau + \tau \bigg|\int_{B_r \backslash B_{r/2}} (e \cdot h) \left[1 - \eta\left(\frac{h}{r}\right) \right] \ K_{\ell}(h) dh \bigg|.$$
	Dividing by $\tau$ yields the estimate
	$$\bigg|b_{\ell}\cdot e - \int_{B_{r_0} \backslash B_r} (e \cdot h) \ K_{\ell}(h) dh  \bigg| \leq C + \bigg|\int_{B_r \backslash B_{r/2}} (e \cdot h) \left[1 - \eta\left(\frac{h}{r}\right) \right] \ K_{\ell}(h) dh \bigg|.$$
	To estimate the integral on the right-hand side, we recall from Corollary \ref{corIUG:LevyMeasures} that $K_{\ell}(h) \approx |h|^{-(n+1)}$. This yields
	\begin{align*}
	\bigg|\int_{B_r \backslash B_{r/2}} (e \cdot h) \left[1 - \eta\left(\frac{h}{r}\right) \right] \ K_{\ell}(h) dh \bigg| & \leq \int_{B_r \backslash B_{r/2}} C |h|^{-n} \ dh \\
	& = C \int_{r/2}^r s^{-1} \ ds = C \log(2).
	\end{align*}
\end{proof}

%%%%%%%%%%%%%%%%%%%%%%%%%%%%%%%%%%%%%%%%%%%%%%%%%
%%%%%%%%%%%%%%%%%%%%%%%%%%%%%%%%%%%%%%%%%%%%%%%%%
%%%%%%%%%%%%%%%%%%%%%%%%%%%%%%%%%%%%%%%%%%%%%%%%%
%%%%%%%%%%%%%%%%%%%%%%%%%%%%%%%%%%%%%%%%%%%%%%%%%

\subsection{Collecting the arguments to finish Theorem \ref{thm:StructureOfHMain}}\label{sec:FinishProofOfStructureTheorem}

Here we just list all of the particular previous results that are used to culminate in the proof of Theorem \ref{thm:StructureOfHMain}.

\begin{proof}[Proof of Theorem \ref{thm:StructureOfHMain}]

First, we note that the function, $H$, enjoys the GCP over $\K$ (see Definition \ref{defFD:GCP}).  This was already established in \cite{ChangLaraGuillenSchwab-2019SomeFBAsNonlocalParaboic-NonlinAnal}, but we will briefly comment on it here.  Indeed, if $f,g\in\K$ and $f\leq g$ with $f(x_0)=g(x_0)$, then we also know that $D_f^+\subset D_g^+$.  Thus, since $U_g^+\geq 0$ on $\Gam_f$, we see that $U_g^+$ is a supersolution of the same equation that governs $U_f^+$.  Since $U_f^+(x_0,f(x_0))=0=U_g^+(x_0,g(x_0))$, we see that also $\partial_\nu^+ U_f(x_0,f(x_0))\leq \partial_\nu^+ U_g(x_0,g(x_0))$.   Hence, $I^+(f,x_0)\leq I^+(g,x_0)$.  A similar argument can be applied to $\partial_\nu^- U_f$ and $\partial_\nu^- U_g$,  but this time the ordering is reversed (per the definition in (\ref{eqIN:DefOfPosNegNormalDeriv})), as now we have $D_g^-\subset D_f^-$.  Combining these inequalities with the definition in (\ref{eqIN:DefOfH}), and remembering that $G$ is increasing in its first variable and decreasing in its second variable (and by assumption on $f$ and $g$, $\grad f(x_0)=\grad g(x_0)$), we conclude the GCP for $H$.

We know that since $H$ is Lipschitz on $\K$ and enjoys the GCP, we will want to invoke Theorem \ref{thmFD:StructureOfJFromGSEucSpace}.  However, we still need to establish that the extra decay requirement in (\ref{eqFD:ExtraModulusConditionOutsideBR}) is satisfied. Indeed it is, which we will show after this current proof in  Lemma \ref{lemPfThm1:ExtraDecayCondition}, below.  Now, \emph{assuming} we have established (\ref{eqFD:ExtraModulusConditionOutsideBR}), then Theorem \ref{thmFD:StructureOfJFromGSEucSpace} shows that all $\ell\in \D_H(0)$ enjoy the structure as claimed in part (i) of Theorem \ref{thm:StructureOfHMain} ($\D_H(0)$ is from Definition \ref{defPfThm1:DefDHAt0}, following Definition \ref{defFD:DJLimitingDifferential}).  After a relabeling of $a^{ij}=h(g)-\ell(g)$ and the triple $c_\ell$, $b_\ell$, and $K_\ell$, from Theorem \ref{thmFD:StructureOfJFromGSEucSpace}, we see that part (i) has been established.

To conclude part (ii) of the theorem, we can invoke Corollary \ref{corIUG:LevyMeasures} for the L\'evy measure estimates and Lemma \ref{lemPfThm1:bijBounded} for the bounds involving the drift terms.

\end{proof}

\begin{lemma}\label{lemPfThm1:ExtraDecayCondition}
	There exists constants, $C>0$, and $\al\in(0,1]$, depending on $\del$, $L$, $m$, $\rho$, and $n$, so that if  $f,g\in\K(\del,L,m,\rho)$ and $f\equiv g$ in $B_{2R}$, then 
	\begin{align*}
		\norm{H(f)-H(g)}_{L^\infty(B_R)}\leq \frac{C}{R^\al}. 
	\end{align*}
\end{lemma}

\begin{proof}[Proof of Lemma \ref{lemPfThm1:ExtraDecayCondition}]
	First, we will establish that
	\begin{align}\label{eqPfThm1:ExtraDecayForI}
		\norm{I^+(f)-I^+(g)}_{L^\infty(B_R)}\leq \frac{C}{R^\al}.
	\end{align}
	Then, as following the proof of Lemma \ref{lemLIP:HEnjoysLittleIEstPart1}, we will see that this estimate carries over to $H$ as well.

	The proof of (\ref{eqPfThm1:ExtraDecayForI}) goes very similarly to the proofs of Lemmas  \ref{lemLIP:LittleILipEstPart1} and \ref{lemLIP:LittleILipEstPart2} (specifically, Step 3), combined with Lemma \ref{lemGF:DecayAtInfinityForTechnicalReasons}.  As in the proof of Lemma \ref{lemLIP:HEnjoysLittleIEstPart2}, we define the functions,
	\begin{align*}
		g_1=\min\{f,g\}\ \ \ \text{and}\ \ \ g_2=\max\{f,g\},
	\end{align*}
	and by construction, the respective domains are the ordered as follows:
	\begin{align*}
		D_{g_1}\subset D_f \subset D_{g_2}\ \ \ \text{and}\ \ \
		D_{g_1}\subset D_g \subset D_{g_2}.
	\end{align*}
	Thus, we see that at least in $B_R$, since $f\equiv g$,
	\begin{align*}
		\partial_\nu U_{g_1}\leq\partial_\nu U_f \leq \partial_\nu U_{g_2}\ \ \ \text{and}\ \ \ 
		\partial_\nu U_{g_1}\leq\partial_\nu U_g\leq \partial_\nu U_{g_2},
	\end{align*}
	so we have
	\begin{align*}
		\partial_\nu U_{g_2} - \partial_\nu U_{g_1}\leq
		\partial_\nu U_f(X)-\partial_\nu U_g(X) \leq \partial_\nu U_{g_1}(X) - \partial_\nu U_{g_2}(X).
	\end{align*}
	Furthermore, for the function $W$, defined as 
	\begin{align*}
		\begin{cases}
			\Delta W = 0\ &\text{in}\ D_{g_1}\\
			W=0\ &\text{on}\ \{x_{n+1}=0\}\\
			W= U_{g_2}|_{\Gam_{g_1}}\ &\text{on}\ \Gam_{g_1},
		\end{cases}
	\end{align*}
	we see that in the smaller domain, $D_{g_1}$,
	\begin{align*}
		U_{g_2} = U_{g_1} + W.
	\end{align*}
	Thus, we have reduced the estimate to
	\begin{align*}
		\partial_\nu U_{g_1}(X) - \partial_\nu U_{g_2}(X) = \partial_\nu W(X),
	\end{align*}
	and so 
	\begin{align*}
		\abs{\partial_\nu U_f(X)-\partial_\nu U_g(X)}\leq \abs{\partial_\nu W(X)}.
	\end{align*}
	As $g_1$ and $g_2$ are $C^{1,\dini}(B_{2R})$ and globally Lipschitz with Lipschitz constant, $m$, we see that Lemma \ref{lemGF:DecayAtInfinityForTechnicalReasons} gives for $X\in\Gam_f\intersect B_R$, and $s>0$,
	\begin{align*}
		W(X+s\nu(X))\leq \frac{Cs}{R^\al},
	\end{align*}
	and hence 
	\begin{align*}
		\partial_\nu W(X)\leq \frac{C}{R^\al}.
	\end{align*}
	Thus, we have established
	\begin{align*}
		\forall\ x\in B_R,\ \ \abs{I^+(f,x) - I^+(g,x)}\leq \frac{C}{R^\al}.
	\end{align*}
	
\end{proof}

\begin{rem}
	The curious reader may see that our approach to establish Theorem \ref{thm:StructureOfHMain} deviates slightly from the one given in \cite{ChangLaraGuillenSchwab-2019SomeFBAsNonlocalParaboic-NonlinAnal}, and we believe the reasons for this deviation are noteworthy.  We will discuss this in more detail in Section \ref{sec:Commentary}.  
\end{rem}

%%%%%%%%%%%%%%%%%%%%%%%%%%%%%%%%%%%%%%%%%%%%%%%%%
%%%%%%%%%%%%%%%%%%%%%%%%%%%%%%%%%%%%%%%%%%%%%%%%%
%%%%%%%%%%%%%%%%%%%%%%%%%%%%%%%%%%%%%%%%%%%%%%%%%
%%%%%%%%%%%%%%%%%%%%%%%%%%%%%%%%%%%%%%%%%%%%%%%%%
%%%%%%%%%%%%%%%%%%%%%%%%%%%%%%%%%%%%%%%%%%%%%%%%%
%%%%%%%%%%%%%%%%%%%%%%%%%%%%%%%%%%%%%%%%%%%%%%%%%
%%%%%%%%%%%%%%%%%%%%%%%%%%%%%%%%%%%%%%%%%%%%%%%%%
%%%%%%%%%%%%%%%%%%%%%%%%%%%%%%%%%%%%%%%%%%%%%%%%%
%%%%%%%%%%%%%%%%%%%%%%%%%%%%%%%%%%%%%%%%%%%%%%%%%

\section{Proof of Theorem \ref{thm:FBRegularity}}\label{sec:KrylovSafonovForHS}

Here we will prove Theorem \ref{thm:FBRegularity}.  As a first step, we wish to exhibit which ellipticity class will apply to the equation solved by the finite differences of $f$.  Determining this class gives a result that depends on the structure provided in Theorem \ref{thm:StructureOfHMain}, and the class, as well as resulting regularity results will depend on the parameters $\del$, $L$, $m$, $\rho$, which is the source for the dependence in the outcome of Theorem \ref{thm:FBRegularity}.  The key is to note what are some valid choices for extremal operators that govern our mapping, $H$ (extremal operators are those defined in (\ref{eqPara:ExtremalOperators}) that satisfy (\ref{eqPara:ExtremalInequalities})). We see from the min-max representation of $h$ in Theorem \ref{thm:StructureOfHMain} (recall $h(f)=H(f,0)$, as in Lemma \ref{lemLIP:HEnjoysLittleIEstPart1}) that if $f_1, f_2 \in \K$, then
\begin{align*}
h(f_1) & = \min_{g\in\K} \left( \max_{\ell\in \D_H(0)} h(g) - \ell(g) + \ell(f)   \right) \\
& \leq  \max_{\ell\in \D_H(0)} h(f_2) - \ell(f_2) + \ell(f_1) \\
& \leq h(f_2) +  \max_{\ell\in \D_H(0)}  \ell(f_1 - f_2).
\end{align*}
Next, let $g_1 \in \K$ be such that $h(f_1) = \max_{\ell\in \D_H(0)} h(g_1) - \ell(g_1) + \ell(f_1)$, and let $\ell_2 \in \D_H(0)$ be such that $\ell_2(f_2 - g_1) = \max_{\ell\in \D_H(0)} \ell(f_2 - g_1)$. We then find that
\begin{align*}
h(f_1) - h(f_2) & = \left( \max_{\ell\in \D_H(0)} h(g_1) - \ell(g_1) + \ell(f_1)   \right) - \min_{g\in\K} \left(  \max_{\ell\in \D_H(0)} h(g) - \ell(g) + \ell(f_2)  \right) \\
& \geq h(g_1) +  \left( \max_{\ell\in \D_H(0)} \ell(f_1 - g_1) \right) - \left(  \max_{\ell\in \D_H(0)} h(g_1) - \ell(g_1) + \ell(f_2)  \right) \\
& = h(g_1) +  \left( \max_{\ell\in \D_H(0)} \ell(f_1 - g_1) \right) - h(g_1) -  \left( \max_{\ell\in \D_H(0)} \ell(f_2 - g_1) \right) \\
& = \max_{\ell\in \D_H(0)} \ell(f_1 - g_1)  - \ell_2(f_2 - g_1) \\
& \geq \ell_2(f_1 - g_1) - \ell_2(f_2 - g_1)\\
& = \ell_2(f_1 - f_2) \geq  \min_{\ell\in \D_H(0)} \ell(f_1 - f_2)
\end{align*}
In summary, we have
\begin{equation}\label{eqPfThm2:boundwithwrongmaximaloperators}
\forall\ f_1, f_2 \in \K,\ \ \ \min_{\ell\in \D_H(0)}\ell(f_1 - f_2) \leq h(f_1) - h(f_2) \leq  \max_{\ell\in \D_H(0)} \ell(f_1 - f_2).
\end{equation}
Now let $C$ be the constant in Theorem \ref{thm:StructureOfHMain} (ii), and let $\L_{\Lam}$ be the class of operators from Definition \ref{defPara:scaleinvariantclass} with $\Lam = C$. We claim there exist constants $C_1,C_2$ such that
\begin{equation}\label{eqPfThm2:boundinglinearoperatorsbymaximaloperators}
\max_{\ell \in \D_H(0)} \ell(f) \leq \mathcal{M}^+_{\L_{\Lam}}(f) + C_1||f||_{L^{\infty}(\RN)}, \quad \min_{\ell \in \D_H(0)} \ell(f) \geq \mathcal{M}^-_{\L_{\Lam}}(f) - C_2||f||_{L^{\infty}(\RN)}.
\end{equation}
First notice that the lower bound on the kernels in Theorem \ref{thm:StructureOfHMain} (ii) is only valid in a small ball. To be able to apply the regularity results in Section \ref{sec:BackgroundParabolic}, namely Proposition \ref{propPara:holderestimate}, the kernels must satisfy the lower bound stated in Definition \ref{defPara:scaleinvariantclass}. To do this, we employ a strategy similar to that in \cite[Section 14]{CaSi-09RegularityIntegroDiff} for truncated kernels. Indeed, if $\ell \in \D_H(0)$, then we may write
$$\ell(f)(x) = c^{ij}f(x) + b^{ij}\cdot\grad f(x) + \int_{\real^n}\del_y f(x)\tilde{K}^{ij}(y)dy - \Lam^{-1} \int_{\real^n \backslash B_{r_0}} (f(x+y)-f(x))|y|^{-n-1} \ dy,$$
where $\tilde{K}^{ij}(y) = K^{ij}(y) + \Lam^{-1}\Indicator_{\real^n \backslash B_{r_0}} |y|^{-n-1}$. Since $b^{ij}, \tilde{K}^{ij} \in \L_{\Lam}$ and $\Lam^{-1}\Indicator_{\real^n \backslash B_{r_0}} |y|^{-n-1} \in L^1(\RN)$, and taking into account the bound on $c^{ij}$ given in Theorem \ref{thm:StructureOfHMain}, the inequalities \eqref{eqPfThm2:boundinglinearoperatorsbymaximaloperators} hold.

As an immediate consequence of \eqref{eqPfThm2:boundinglinearoperatorsbymaximaloperators} and \eqref{eqPfThm2:boundwithwrongmaximaloperators}, we find that
\begin{equation}\label{eqPfThm2:ellipticityofH}
-C_1 |f_1 - f_2| + \mathcal{M}^-_{\L_{\Lam}}(f_1 - f_2) \leq h(f_1) - h(f_2) \leq \mathcal{M}^+_{\L_{\Lam}}(f_1 - f_2) + C_2|f_1 - f_2| \text{ for all } f_1, f_2 \in \K.
\end{equation}
With \eqref{eqPfThm2:ellipticityofH} at hand, Theorem \ref{thm:FBRegularity} follows by combining the conclusion of (\ref{eqPfThm2:ellipticityofH}) with the following $C^{1,\gamma}$ estimate for translation invariant operators, Proposition \ref{propPfThm2:TranslationInvariantC1gammaEstimate}, whose statement and proof are essentially that of \cite[Theorem 6.2]{Silvestre-2011RegularityHJE}.  As $\L_\Lam$ depends upon $\del$, $L$, $m$, $\rho$, then so does the constant obtained in Theorem \ref{thm:FBRegularity}.  We note that by \cite[Theorem 1.1 (iii)]{ChangLaraGuillenSchwab-2019SomeFBAsNonlocalParaboic-NonlinAnal}, the Lipschitz bound on $f(\cdot,0)$ is preserved for all time.  Thus, in the following Proposition \ref{propPfThm2:TranslationInvariantC1gammaEstimate}, when applied to $f$ in Theorem \ref{thm:FBRegularity}, we can replace $\norm{f}_{C^{0,1}(\real^n\times[0,T])}$ by $\norm{f}_{C^{0,1}(\real^n\times\{0\})}$.

We provide the standard argument for the proof of Proposition \ref{propPfThm2:TranslationInvariantC1gammaEstimate} using difference quotients for the sake of completeness.

\begin{proposition}\label{propPfThm2:TranslationInvariantC1gammaEstimate} Suppose $u \in C^{0,1}(\RN \times [0,t_0])$ is a viscosity solution of the translation invariant non-local equation $\partial_t u - J(u) =0$ in $\RN \times (0,t_0)$, where $J$ satisfies the ellipticity condition
\begin{equation}\label{eqPfThm2:ellipticityofI}
	-C_1|u - v| + \M^-_{\L_{\Lam}}(u-v)
	\leq J(u)-J(v)
	\leq \M^+_{\L_{\Lam}}(u-v) +C_2 |u - v|, \quad \text{for all } u, v \in C^{0,1}(\RN).
\end{equation}
Then we have the estimate
$$
||u||_{C^{1, \gamma}(Q_{\frac{t_0}{2}}(t_0,x_0))} \leq \frac{C(1+t_0)}{t_0^{\gamma}} ||u||_{C^{0,1}(\mathbb{R}^n \times [0,t_0])},
$$
where $C$ and $\gamma$ are the constants from Proposition \ref{propPara:holderestimate}.
\end{proposition}

\begin{rem}
	The constants $C$ and $\gam$ arising from Propostion \ref{propPara:holderestimate} depend upon the ellipticity class, $\L_\Lam$.  Since, as above, our particular choice of class, $\L_\Lam$, depends on the estimates of Theorem \ref{thm:StructureOfHMain}, which depend upon $\del$, $L$, $m$, $\rho$, we see that an invocation of Proposition \ref{propPfThm2:TranslationInvariantC1gammaEstimate} for our situation retains such dependence on the constants in the $C^{1,\gam}$ estimate.
\end{rem}

\begin{proof}[Proof of Proposition \ref{propPfThm2:TranslationInvariantC1gammaEstimate}]

For $(x,t) \in Q_1$, consider the difference quotient in space
$$v_h(x,t) := \frac{u(x+h,t) - u(x,t)}{|h|}.$$
Using the ellipticity condition \eqref{eqPfThm2:ellipticityofI} and the translation invariance of $J$, we find that \emph{in the viscosity sense}, $v_h$ solves
$$
C_2|v_h| + \M^+_{\L_{\Lam}}(v_h)  \geq \frac{I(u(\cdot + h,t),x,t) - I(u,x,t)}{|h|} = \partial_t v_h(x,t).
$$
Since $||v_h||_{L^{\infty}(\RN \times [0,t_0])} \leq ||u||_{C^{0,1}(\RN \times [0,t_0])}$ independently of $h$, it follows that $v_h$ satisfies the inequality, \emph{in the viscosity sense},
$$\partial_t v_h(x,t) - \M^+_{\L_{\Lam}}(v_h)(x,t) \leq ||u||_{C^{0,1}(\RN \times [0,t_0])} \qquad \text{for all } (x,t) \in \RN \times (0,t_0).$$
A similar argument shows $v_h$ also satisfies \emph{in the viscosity sense},
$$\partial_t v_h(x,t) - \M^-_{\L_{\Lam}}(v_h)(x,t) \geq -||u||_{C^{0,1}(\RN \times [0,t_0])} \qquad \text{for all } (x,t) \in \RN \times (0,t_0).$$
Applying Proposition \ref{propPara:holderestimate} to $v_h$, we conclude that
$$||v_h||_{C^{\gamma}(Q_{\frac{t_0}{2}}(t_0,x_0))} \leq \frac{C(1+t_0)}{t_0^{\gamma}} ||u||_{C^{0,1}(\mathbb{R}^n \times [0,t_0])}.$$
Since the right-hand side is independent of $h$, we may let $|h| \rightarrow 0$ to obtain the desired $C^{1,\gamma}$ estimate in space.  By considering $(x,t) \in \RN \times (0,t_0)$ the difference quotient in time 
$$w_h(x,t) := \frac{u(x,t+h) - u(x,t)}{|h|},$$
with $h$ sufficiently small and carrying out an argument as above, we also obtain a $C^{1,\gamma}$ estimate in time.  The one extra step is that once we obtain the regularity in space, we see that $u_t$ is bounded \emph{in the viscosity sense}, and hence $u$ is Lipschitz in time.  Thus, $w_h$ is a bounded viscosity solution of the extremal inequalities.  Another invocation of Proposition \ref{propPara:holderestimate} concludes the regularity in time.

\end{proof}

%%%%%%%%%%%%%%%%%%%%%%%%%%%%%%%%%%%%%%%%%%%%%%%%%
%%%%%%%%%%%%%%%%%%%%%%%%%%%%%%%%%%%%%%%%%%%%%%%%%
%%%%%%%%%%%%%%%%%%%%%%%%%%%%%%%%%%%%%%%%%%%%%%%%%
%%%%%%%%%%%%%%%%%%%%%%%%%%%%%%%%%%%%%%%%%%%%%%%%%
%%%%%%%%%%%%%%%%%%%%%%%%%%%%%%%%%%%%%%%%%%%%%%%%%
%%%%%%%%%%%%%%%%%%%%%%%%%%%%%%%%%%%%%%%%%%%%%%%%%
%%%%%%%%%%%%%%%%%%%%%%%%%%%%%%%%%%%%%%%%%%%%%%%%%
%%%%%%%%%%%%%%%%%%%%%%%%%%%%%%%%%%%%%%%%%%%%%%%%%
%%%%%%%%%%%%%%%%%%%%%%%%%%%%%%%%%%%%%%%%%%%%%%%%%

\section{Commentary on Many Issues}\label{sec:Commentary}

\subsection{Where is the min-max structure utilized?}

The first place the min-max in Theorem \ref{thm:StructureOfHMain} is used is to identify the correct class of integro-differential operators for invoking the Krylov-Safonov theory.  In most of the existing literature on regularity theory (as well as existence and uniqueness theory), a min-max structure is assumed for the given equations.  However, the min-max structure is quickly replaced by simply requiring the existence of a class of linear nonlocal operators so that the relevant nonlinear operator, say $J$, satisfies inequalities such as (\ref{eqPara:ExtremalInequalities}).  Then, as one sees by, e.g. Proposition \ref{propPara:holderestimate}, it is these extremal inequalities that govern the regularity theory.  Thus, as outlined in Section \ref{sec:KrylovSafonovForHS}, as soon as a min-max, plus some properties of the ingredients are obtained as in Theorem \ref{thm:StructureOfHMain} one can deduce which ellipticity class and results will apply to solutions of $\partial_t f= H(f)$.  It was rather striking to find in the case of (\ref{eqIN:HeleShawIntDiffParabolic}), under the extra $C^{1,\dini}$ regularity assumption for $f$, that the resulting ellipticity class had already been studied in the literature as in \cite{ChangLaraDavila-2016HolderNonlocalParabolicDriftJDE}.  Furthermore, thanks to the translation invariance of $H$, combined with the inequalities (\ref{eqPara:ExtremalInequalities}), it is not hard to show that  the finite differences, $w=\frac{1}{\abs{h}}(f(\cdot +h)-f(\cdot))$ satisfy, in the viscosity sense, the pair of inequalities (\ref{eqPara:ExtremalOperators}).  This is key to obtaining the $C^{1,\gam}$ regularity for $f$.

In some sense, the min-max provided by Theorem \ref{thm:StructureOfHMain} gives a way of ``linearizing'' the equation, but in a possibly slightly different manner than sometimes carried out.  One way to linearize (\ref{eqIN:HeleShawIntDiffParabolic}) would be to fix a very smooth solution, $f_0$, and then find an equation, say $\partial_t \psi = L_{f_0}\psi$, where $L_{f_0}$ is an operator with coefficients depending upon $f_0$, and the equation governs functions of the form $f=f_0+\ep\psi$ for $\ep<<1$.  The min-max gives a different linear equation in the sense that for \emph{any} solution, say $f$, of (\ref{eqIN:HeleShawIntDiffParabolic}), one can think $f$ \emph{itself} solves a linear equation with bounded measurable coefficients of the form,
\begin{align*}
	\partial_t f = c^*_f(x)f(x) + b^*_f(x)\cdot \grad f(x)
	+ \int_{\real^n} \del_h f(x) K^*_f(x,h)dh,
\end{align*}
where $c^*_f$, $b^*_f$, $K^*_f$ are all $x$-depended coefficients that can be any of those that attain the min-max for $f$ in Theorem \ref{thm:StructureOfHMain} at a given $x$.  Of course, one cannot expect these coefficients to be better than bounded and measurable in $x$, and this is one reason why it is typically presented in the elliptic and parabolic literature that linear equations with bounded measurable coefficients are as easy or hard to treat (it depends upon your point of view) as fully nonlinear equations that are translation invariant.  Of course, we ``linearized'' equation (\ref{eqIN:HeleShawIntDiffParabolic}) in neither of the two approaches mentioned above, but as earlier, we found that linearizing for $w=\frac{1}{\abs{h}}(f(\cdot +h)-f(\cdot))$ gives the inequalities pertinent to Proposition \ref{propPara:holderestimate}.  If one used the mean value theorem, it would formally give a linear equation with bounded measurable coefficients, \emph{assuming} that $H$ was a Fr\'echet differentiable map (but one can obtain the inequalities \eqref{eqPara:ExtremalInequalities} without any assumption of differentiability of $H$, thanks to the min-max).

In Section \ref{sec:KrylovSafonovForHS},  the min-max representation of $h$ suggests that the natural maximal and minimal operators corresponding to $h$ should be the ones given by \eqref{eqPfThm2:boundwithwrongmaximaloperators}. However, one does not know if there is regularity theory available for these maximal and minimal operators.  One annoyance in this direction is that  the class of linear operators used to define them is, in general, not invariant under translations and dilations.  Certainly the needed regularity is true, but the arguments to produce such results are better implemented for a larger class of equations, such as those described in Definition \ref{defPara:scaleinvariantclass}.  The bounds obtained in Theorem \ref{thm:StructureOfHMain} (ii) instead allow us to estimate $h$ by a different set of maximal and minimal operators as shown in \eqref{eqPfThm2:ellipticityofH}, where the operators  $\M^+_{\L_{\Lam}}$ and $ \M^-_{\L_{\Lam}}$ are defined using a class of linear operators which satisfies the translation and dilation invariance properties necessary to invoke existing regularity theory while also containing the linear functionals that support $h$.

We also note an interesting departure from an easier min-max approach as utilized in \cite{ChangLaraGuillenSchwab-2019SomeFBAsNonlocalParaboic-NonlinAnal} and \cite[Theorem 1.10]{GuSc-2019MinMaxEuclideanNATMA}.  The curious reader may see that since $H$ is translation invariant, there is a quicker and more straightforward way to obtaining the first half of Theorem \ref{thm:StructureOfHMain} stated in part (i).  The translation invariance means that it suffices to look only at $H(f,0)$, and as a Lipschitz \emph{functional} from the Banach space, $X_\rho$, to $\real$, $H(f,0)$ enjoys a larger collection of tools from the nonlinear analysis setting built in Clarke's book \cite{Clarke-1990OptimizationNonsmoothAnalysisSIAMreprint}.  The mean value theorem of Lebourg \cite{Clarke-1990OptimizationNonsmoothAnalysisSIAMreprint} that we give a variant on in Lemma \ref{lemFD:MVmaxProperty} has a more straightforward presentation using a more natural subdifferential set than the one defined in Definition \ref{defFD:DJLimitingDifferential}.  This is the approach that is pursued in proving the corresponding result in \cite[Theorem 1.4]{ChangLaraGuillenSchwab-2019SomeFBAsNonlocalParaboic-NonlinAnal}  and \cite[Theorem 1.10]{GuSc-2019MinMaxEuclideanNATMA}.  The problem with using the more natural subdifferential set that circumvents the cumbersome details of the finite dimensional approximations is that it is very hard to capture in the linear operators for the min-max the non-degeneracy property that is proved in Lemma \ref{lemLIP:LittleILipEstPart1}.  For a lack of a better analogy, it is like saying that for the function $A:\real^N\to\real$ given by $A(x)=\abs{x}$, one can think of the contrast in reconstructing $A$, by considering the set of all possible supporting hyperplanes, versus considering the actual derivative $DA$ at any point where $DA$ may exist.  In the former situation, one cannot avoid that degenerate linear functionals, such as the zero functional, appear in the collection that makes up a min-max (just a max, actually) representation of $A$, whereas in the latter, one can see that the only differentials that would be used will be those with norm $1$, and hence are ``non-degenerate'' in a sense.  This is the reason for the finite dimensional approximations used in Section \ref{sec:ProofOfStructureThm} because a non-degeneracy property like that in Lemma \ref{lemLIP:LittleILipEstPart1} can be preserved in the functionals used for the min-max in Corollary \ref{corFD:GenericMinMaxForJ}.

\subsection{A counter example}

There are interesting pathologies in Hele-Shaw free boundary problems related to the contrast between $U$ being regular in space-time and $\partial\{U>0\}$ being regular in space-time.  Aside from the fact that there are geometries in which the free boundary may stagnate and then immediately jump in space-time (see \cite{KingLaceyVazquez-1995PersistenceOfCornersHeleShaw}), there are solutions of (\ref{eqIN:HSMain}) with space-planar free boundaries such as (see \cite{Kim-2006RegularityFBOnePhaseHeleShaw-JDE})
\begin{align*}
	U(X,t) = a(t)\left( X_{n+1} + \int_0^t a(s)ds  \right),
\end{align*}
with e.g. $a$ is a bounded function of $t$.
The zero set is, of course, given by 
\begin{align*}
	\partial\{U>0\}=\left\{X_{n+1}=-\int_0^t a(s)ds\right\},\ \ \text{hence}\ \ f(x,t)=-\int_0^ta(s)ds.
\end{align*}
We note that this special solution does not necessarily satisfy the spatial boundary conditions prescribed by (\ref{eqIN:HSMain}), and indeed, in the absence of further restrictions on $a$, it is not true that $f\in C^{1,\gam}(\real^n\times[\tau,T])$.  However, if one insists that this solution does satisfy (\ref{eqIN:HSMain}) exactly, we then see
the boundary condition that $U(0,t)=1$ means that
\begin{align*}
	a(t)\left(\int_0^t a(s)ds\right)=1,
\end{align*}
whereby $a(t)=\pm(2t+c)^{-1/2}$, for some $c\geq0$, and hence $\int_0^t a(s)ds=\pm (2t+c)^{1/2}$.  In order that $U>0$, we see that in fact $a(t)=-(2t+c)^{-1/2}$, and  hence 
\begin{align*}
	U(X,t)=-(2t+c)^{-1/2}\left(X_{n+1}-(2t+c)^{1/2}\right),\ \ \ \text{in}\ \ \ \{0<X_{n+1}<(2t+c)^{1/2}\},
\end{align*}
and so
\begin{align*}
	f(x,t)=(2t+c)^{1/2}.
\end{align*}
In particular, requiring that the free boundary resides in the region $\real^n\times[\del,L-\del]$, we see
\begin{align*}
	c>\del^2.
\end{align*}
Thus, indeed, $f\in C^{\infty}(\real^n\times[0,T])$, with a norm that depends on $\del$, which is compatible with the result in Theorem \ref{thm:FBRegularity}.

\subsection{Some questions}

Here, we list some questions related to (\ref{eqIN:HSMain}) and Theorem \ref{thm:FBRegularity}.

\begin{itemize}
	\item Is the gain in regularity given in Theorem \ref{thm:FBRegularity} enough to prove higher regularity, such as a $C^\infty$ free boundary?  This would be related to higher regularity via Schauder or bootstrap methods for integro-differential equations, such as that pursued in e.g. \cite{Bass2009-RegularityStableLikeJFA}, \cite{DongJinZhang-2018DiniAndSchauderForNonlocalParabolicAPDE}, \cite{JinXiong-2015SchauderEstLinearParabolicIntDiffDCDS-A}, \cite{MikuleviciusPragarauskas-2014CauchyProblemIntDiffHolderClassesPotAnalysis}; or like the analysis for free boundary problems that attains smooth solutions, such as in \cite{BarriosFigalliValdinoci-2014BootstrapRegularityIntDiffNonlocalMinimalAnnScuolNormPisa}, \cite{ChoiJerisonKim-2009LocalRegularizationOnePhaseINDIANA}, \cite{KinderlehrerNirenberg-1977RegularityFBAnnScuolNormPisa}, \cite{KinderlehrerNirenberg-1978AnalaticityAtTheBoundaryCPAM}.

	\item Is it possible to include variable coefficients in equation (\ref{eqIN:HSMain}) and obtain the same regularity of the solution?  This could be for either a divergence form operator or a non-divergence form operator.  It is conceivable that similar regularity should hold, and one may expect to use either directly, or modifications of works such as \cite{DongJinZhang-2018DiniAndSchauderForNonlocalParabolicAPDE}, \cite{Kriventsov-2013RegRoughKernelsCPDE}, \cite{Serra-2015RegularityNonlocalParabolicRoughKernels-CalcVar}, when the order of the kernels is $1$.

	\item How does incorporating an inhomogeneous boundary law, $V=G(X,\partial^+_\nu U^+, \partial^-_\nu U^-)$, in (\ref{eqIN:HSMain}) change the outcome of the results?  At least when $G(X,\partial^+_\nu U^+, \partial^-_\nu U^-)=g(X)\tilde G(\partial^+_\nu U^+, \partial^-_\nu U^-)$ it appears as though the steps would be very similar, but if the $X$ dependence is more general, the analysis in Section \ref{sec:KrylovSafonovForHS}, may be complicated by the fact that the equation is not translation invariant, and the $x$ dependence is not as easily isolated.
	
	\item The most important question to address could be to adapt the method to apply to situations in which $\partial\{U>0\}$ is only \emph{locally} a space-time graph of a function.  In many free boundary problems related to (\ref{eqIN:HSMain}), it is not natural to assume that the free boundary is globally the graph of some function.  Rather, without assuming the free boundary is a graph, some low regularity assumption like a Lipschitz condition or a flatness condition then forces the free boundary to in fact be locally a graph that is quite regular (at least for small time that avoids different regions of the free boundary colliding and causing topological changes).  This could be attained by including as a parameter in the definition of $I$, some extra space-time boundary condition that allows $I$ to act on functions that are merely defined in, say, $B_1$, instead of $\real^n$, with this extra boundary condition providing the information of the free boundary outside of $B_1$.
	
	\item Another interesting question is to address the possibility to modify the method to apply to Stefan type problems wherein (\ref{eqIN:HSMain}) now requires $U$ to solve a parabolic problem in the sets $\{U>0\}$ and $\{U<0\}$.  Of course, the two-phase Stefan problem itself is already rather well understood, but there are many variations that could be considered.  This would require adapting the results in Section \ref{sec:FiniteDim} to accommodate operators acting on $f:\real^n\times[0,T]$ that satisfy the GCP in space-time, rather than simply looking at those operators that satisfy the GCP in space.  

\end{itemize}

\appendix

%%%%%%%%%%%%%%%%%%%%%%%%%%%%%%%%%%%%%%%%%%%%%%%%%
%%%%%%%%%%%%%%%%%%%%%%%%%%%%%%%%%%%%%%%%%%%%%%%%%
%%%%%%%%%%%%%%%%%%%%%%%%%%%%%%%%%%%%%%%%%%%%%%%%%
%%%%%%%%%%%%%%%%%%%%%%%%%%%%%%%%%%%%%%%%%%%%%%%%%
%%%%%%%%%%%%%%%%%%%%%%%%%%%%%%%%%%%%%%%%%%%%%%%%%
%%%%%%%%%%%%%%%%%%%%%%%%%%%%%%%%%%%%%%%%%%%%%%%%%
%%%%%%%%%%%%%%%%%%%%%%%%%%%%%%%%%%%%%%%%%%%%%%%%%
%%%%%%%%%%%%%%%%%%%%%%%%%%%%%%%%%%%%%%%%%%%%%%%%%
%%%%%%%%%%%%%%%%%%%%%%%%%%%%%%%%%%%%%%%%%%%%%%%%%

\section{Proofs related to Green's function estimates}

Before proving Lemma \ref{lemGF:lineargrowth}, we recall the following fact from \cite{GruterWidman-1982GreenFunUnifEllipticManMath}.

\begin{lemma}\label{lemGF:3.2InGW} (cf. Lemma 3.2 in \cite{GruterWidman-1982GreenFunUnifEllipticManMath}) Suppose $A$ is $\lam,\Lam$ uniformly elliptic and Dini continuous with modulus, $\om$, and $v$ solves the Dirichlet problem
\begin{align}\label{eqAppendix:Lem3.2InGW}
\begin{cases}
L_A v = 0 \quad \text{in } \mathcal{A}_{2r}(x_0), \ x_0 \in \O, \ r \leq 1,\\
v = 1 \quad \text{on } \partial B_r(x_0), \\
v = 0 \quad \text{on } \partial B_{2r}(x_0).
\end{cases}
\end{align}
There exists a constant $K  = K(n,\lambda, \Lambda, \omega) > 0$ such that
$$|\nabla v(x)| \leq \frac{K}{r} \qquad \text{for all } x \in \A_{2r}(x_0).$$
\end{lemma}

\begin{proof}[Proof of Lemma \ref{lemGF:lineargrowth}]
We first perform a reduction to a model problem. By Harnack's inequality applied to the non-negative solution $u$ in the ball $B_r(x_0)$, we know there exists a constant $\tilde{C} = \tilde{C}(n,\lambda,\Lambda)$ such that
$$\inf_{\partial B_r(x_0)} u \geq \tilde{C} u(x_0).$$
Rescaling $u$, we may thus assume $\inf_{\partial B_r(x_0)} u = 1$. Let $v$ be the solution to the problem
\begin{equation}\label{eqnforv}
\begin{cases}
L_Av = 0 \quad \text{in } \A_{2r}(x_0),\\
v = 1 \quad \text{on } \partial B_r(x_0),\\
v = 0 \quad \text{on } \partial B_{2r}(x_0).
\end{cases}
\end{equation}
We recall that by assumption, $\A_{2r}(x_0)\subset\Om$, and hence as $u\geq 0$, by the maximum principle, $u \geq v$ on $\A_{2r}(x_0)$ and so it suffices to prove the estimate \eqref{lineargrowthatboundary} for the function $v$. 

Consider the constant coefficient operator $L_0 := -\text{div}(A(z_0) \nabla \cdot)$, and let $\hat{v}$ solve the problem
\begin{equation}\label{eqnforhatv}
\begin{cases}
L_0 \hat{v} = 0 \quad \text{in } \A_{2r}(x_0),\\
\hat{v} = 1 \quad \text{on } \partial B_r(x_0),\\
\hat{v} = 0 \quad \text{on } \partial B_{2r}(x_0).
\end{cases}
\end{equation}
The function $w: = \hat{v} - v$ vanishes on the boundary of $\A_{2r}(x_0)$. If $G_0$ is the Green's function for the operator $L_0$, then by the representation formula for $L_0$, we have for all $x \in \mathcal{A}_{2r}(x_0)$
$$w(x) = \int_{\A_{2r}(x_0)} G_0(x,y) L_0 w(y) \ dy = \int_{\A_{2r}(x_0)} \la \nabla_y G_0(x,y) , A(z_0) \nabla w(y) \ra \ dy.$$
Now since $\hat{v}$ solves \eqref{eqnforhatv}, we know that 
$$\int_{\A_{2r}(x_0)} \la \nabla_y G_0(x,y), A(z_0) \nabla \hat{v}(y) \ra \ dy = 0.$$
Consequently,
$$w(x) = -\int_{\A_{2r}(x_0)} \la \nabla_y G_0(x,y) , A(z_0) \nabla v(y) \ra \ dy.$$
Next, since $v$ solves \eqref{eqnforv}, we know that 
$$\int_{\A_{2r}(x_0)} \la \nabla_y G_0(x,y) , A(y) \nabla v(y) \ra \ dy = 0.$$
It follows that
$$w(x) = -\int_{\A_{2r}(x_0)} \la \nabla_y G_0(x,y) , (A(z_0) - A(y)) \nabla v(y) \ra \ dy.$$
Differentiating in $x$ yields
$$Dw(x) = - \int_{\A_{2r}(x_0)} \la D^2_{x,y}G_0(x,y) , (A(z_0) - A(y)) \nabla v(y) \ra \ dy.$$
Evaluating at $x = z_0$, we thus conclude
$$Dw(z_0) = - \int_{\A_{2r}(x_0)} \la D^2_{x,y}G_0(z_0,y) , (A(z_0) - A(y)) \nabla v(y) \ra \ dy.$$
Now by estimates for the Green's function for constant coefficient operators, we know there exists a constant $C_1 = C_1(n,\lambda, \Lambda) > 0$ such that
$$|D^2_{x,y}G_0(z_0,y)| \leq C_1|z_0-y|^{-n}.$$
It follows that
$$|Dw(z_0)| \leq C_1 \int_{\A_{2r}(x_0)} \frac{|A(z_0) - A(y)|}{|z_0 - y|^n} |\nabla v(y)| \ dy.$$
By Lemma \ref{lemGF:3.2InGW}, there exists a constant $K  = K(n,\lambda, \Lambda, \omega) > 0$ such that
$$|\nabla v(y)| \leq \frac{K}{r} \qquad \text{for all } y \in \A_{2r}(x_0).$$
Therefore,
$$|Dw(z_0)| \leq \frac{C_1 K}{r} \int_{\A_{2r}(x_0)} \frac{|A(z_0) - A(y)|}{|z_0 - y|^n} \ dy.$$
We now write the integral above as
\begin{align*}
\int\limits_{\A_{2r}(x_0)} \frac{|A(z_0) - A(y)|}{|z_0 - y|^n} \ dy & = \int\limits_{\A_{2r}(x_0) \cap B_r(z_0)} \frac{|A(z_0) - A(y)|}{|z_0 - y|^n} \ dy + \int\limits_{\A_{2r}(x_0) \backslash B_r(z_0)} \frac{|A(z_0) - A(y)|}{|z_0 - y|^n} \ dy \\
& = \text{I} + \text{II}.
\end{align*}
Converting to polar coordinates centered at $z_0$, and using the Dini continuity of the coefficients $A(\cdot)$ yields
$$\text{I} \leq C_2 \int_0^r \frac{\omega(t)}{t} \ dt,$$
for a dimensional constant $C_2 > 0$. To control $\text{II}$, we notice that $|z_0 - y| \geq r$ if $y \in \A_{2r}(x_0) \backslash B_r(z_0)$, and so 
$$\text{II} \leq r^{-n} \int\limits_{\A_{2r}(x_0)} |A(z_0) - A(y)| \ dy \leq r^{-n}|\A_{2r}(x_0)| \sup_{y \in \A_{2r}(x_0)} \omega(|z_0 - y|)| \leq C_3 \sup_{y \in \A_{2r}(x_0)} \omega(|z_0 - y|),$$
where $C_3 > 0$ is a dimensional constant. It follows that given $\e > 0$, there exists $r_0 = r_0(n,\omega, \lambda, \Lambda, \e)$ such that if $r \leq r_0$, then $|Dw(z_0)| \leq \frac{\e}{r}$. 

By Taylor expansion around $z_0$, we have
$$v(x) = v(z_0) + Dv(z_0)\cdot(x-z_0) + o(|x-z_0|) \qquad \text{for all } x \in [x_0,z_0] \cap \A_{2r}(x_0).$$
Let $D_{\nu} \varphi(z_0) := \la D\varphi(z_0), \nu(z_0) \ra$ denote the derivative of a function $\varphi$ in the direction of the inward pointing unit normal vector $\nu(z_0)$ to $\partial B_{2r}(x_0)$ at $z_0$. Since $v(z_0) = 0$ and $d(x)\nu(z_0) = x - z_0$, we see that
$$v(x) = D_{\nu}v(z_0) d(x) + o(d(x)) \qquad \text{for all } x \in [x_0,z_0] \cap \A_{2r}(x_0).$$
Writing $v = \hat{v} - w$, we thus obtain
$$v(x) = \left(D_{\nu}\hat{v}(z_0) - D_{\nu}w(z_0) \right)d(x) + o(d(x)) \qquad \text{for all } x \in [x_0,z_0] \cap \A_{2r}(x_0).$$
Now, by explicit calculation of $\hat{v}$, it is possible to show that there exists a constant $C_4 = C_4(n,\lambda,\Lambda) > 0$ such that
$$D_{\nu} \hat{v}(z_0) \geq \frac{C_4}{r}.$$
If we now choose $\e := \frac{C_4}{2}$ above, we obtain
$$D_{\nu}\hat{v}(z_0) - D_{\nu}w(z_0) \geq \frac{C_4}{r} - \frac{\e}{r} = \frac{C_4}{2r}.$$
Therefore, there exist constants $C = C(n,\lambda,\Lambda) > 0$ and $r_0 = r_0(n,\omega,\lambda, \Lambda) > 0$  such that if $r \leq r_0$, then
$$v(x) \geq \frac{C}{r} \ d(x) + o(d(x)) \qquad \text{for all } x \in [x_0,z_0] \cap \A_{2r}(x_0).$$

\end{proof}

From here on, we assume we are working with $\O \subset \Rn$. Before we prove Theorem \ref{thm:GreenBoundaryBehavior}, let us first recall a number of useful facts from \cite{CaffarelliFabesMortolaSalsa-1981Indiana, GruterWidman-1982GreenFunUnifEllipticManMath}. For any $y_0 \in \partial \O$ and $r > 0$, let $\Delta_r(y_0) := B_r(y_0)\cap \partial \O$. We denote by $W_{r,y_0}$ the solution to the Dirichlet problem
\begin{equation}\label{harmonicmeasure}
\begin{cases}
L_A W_{r,y_0} = 0 \quad \text{in } \O,\\
W_{r,y_0} = \mathbbm{1}_{\Delta_r(y_0)} \quad \text{on } \partial\O,
\end{cases}
\end{equation}
i.e. $W_{r,y_0}(x)$ is the harmonic measure of $\Delta_r(y_0)$, based at $x$.

\begin{lemma}\label{CFMSLemma2.1} (cf. Lemma 2.1 in \cite{CaffarelliFabesMortolaSalsa-1981Indiana}) There exist positive numbers $r_0 =r_0(m)$ and $C = C(\lambda, \Lambda, m)$ such that for $r \leq r_0$, we have 
$$W_{r,y_0}(y_0 + r\nu(y_0)) \geq C.$$
\end{lemma}

\begin{lemma}\label{CFMSLemma2.2} (cf. Lemma 2.2 in \cite{CaffarelliFabesMortolaSalsa-1981Indiana}) There exist positive numbers $r_0 =r_0(m)$ and $c = c(\lambda,\Lambda,m)$ such that for $r \leq r_0$ and for all $x \notin B_{3r}(y_0) \cap \Omega$, we have
$$c^{-1}r^{n-1} G(y_0 + r\nu(y_0), x) \leq W_{r,y_0}(x) \leq cr^{n-1} G(y_0 + r\nu(y_0), x),$$
where $G$ is the Green's function corresponding to $L_A$ in $\O  \subset \Rn$.
\end{lemma}

\begin{lemma}\label{GWTheorem1.1} (cf. Theorem 1.1 in \cite{GruterWidman-1982GreenFunUnifEllipticManMath}) There exists a positive constant $K = K(n, \lambda, \Lambda)$ such that if $p,q \in \O \subset \RN$ satisfy $|p-q| \leq \frac{1}{2}d(q)$, then
$$G(p,q) \geq K|p-q|^{1-n},$$
where $G$ is the Green's function corresponding to $L_A$ in $\O  \subset \Rn$.
\end{lemma}

\begin{proof}[Proof of Theorem \ref{thm:GreenBoundaryBehavior}] By flattening $D_f$, we may work on the domain $\O = \left\{0 < x_{n+1} < L \right\}$. We will only focus on proving the estimate \eqref{GlobalGreenEstimate} on the portion of the boundary, $\Gamma_0 := \left\{x_{n+1} = 0\right\}$. Let $R_0$ be the minimum of $L$ and the smallest value of $r_0$ for which the conclusions of Lemma \ref{lemGF:lineargrowth}, Lemma \ref{CFMSLemma2.1}, and Lemma \ref{CFMSLemma2.2} hold. Evidently, $R_0$ depends only on  the Dini modulus of $A(\cdot)$, the $C^{1,\text{Dini}}$ modulus of $f$, and other universal parameters. Since the upper bound in \eqref{GlobalGreenEstimate} is a consequence of \cite[Theorem 3.3]{GruterWidman-1982GreenFunUnifEllipticManMath}, we only show the proof of the lower bound.

Fix $x, y \in \left\{0 < x_{n+1} < L \right\}$ and let $r := |x-y| \leq R_0$. Let $x_0$ (resp. $y_0$) denote the point on $\Gamma_0$ closest to $x$ (resp. $y$), and define $x^* := x_0 + re_{n+1}$ (resp. $y^* := y_0 + r e_{n+1}$). Notice that $d(x) = \text{dist}(x,\Gamma_0) = x_{n+1}$ (resp. $d(y) = \text{dist}(y,\Gamma_0) = y_{n+1}$). Consider the following scenarios:

\noindent {\bf Case 1:} $0 < d(x),d(y) \leq \frac{r}{2}$. \\
Since $x \notin B_r(y^*)$, $G(\cdot, x)$ satisfies the hypotheses of Lemma \ref{lemGF:lineargrowth} in $B_r(y^*)$ and vanishes at $y_0$. Hence, there exists $C_1 = C_1(\lambda,\Lambda,n) > 0$ such that
$$G(y,x) \geq \frac{C_1}{r}G(y^*,x)d(y).$$
Let $\hat{y}:= y_0 + \frac{r}{2\sqrt{3}} e_{n+1}$. By the Boundary Harnack Principle, there exists a constant $C_2 = C_2(\lambda,\Lambda,n) > 0$ such that
$$G(y^*,x) \geq C_2 G\left(\hat{y},x\right).$$
Notice that $x \notin B_{\frac{\sqrt{3}r}{2}}(y_0)$ since
$$|x_0 - y_0|^2 = |x - y|^2 - |x_{n+1} - y_{n+1}|^2 \geq r^2 - \frac{r^2}{4} = \frac{3r^2}{4}.$$
Therefore, by Lemma \ref{CFMSLemma2.2}, there exists a constant $C_3 = C_3(\lambda, \Lambda, m) > 0$ such that
$$G\left(\hat{y},x\right) \geq C_3 r^{1-n} W_{\frac{r}{2\sqrt{3}},y_0}(x).$$
Applying Lemma \ref{lemGF:lineargrowth} to $W_{\frac{r}{2\sqrt{3}},y_0}$ in $B_r(x^*)$, we find there exists a constant $C_4 = C_4(\lambda, \Lambda, n) > 0$ such that
$$W_{\frac{r}{2\sqrt{3}},y_0}(x) \geq \frac{C_4}{r} W_{\frac{r}{2\sqrt{3}},y_0}(x^*) d(x).$$
A crude estimate shows
$$|\hat{y} - x^*| \leq |\hat{y} - y_0| + |y_0 - x_0| + |x_0 - x^*| \leq \frac{r}{2\sqrt{3}} + r + r < \frac{5r}{2}.$$
It follows from a covering argument and Harnack's inequality that there exists a constant $C_5 = C_5(\lambda, \Lambda, n) > 0$ such that
$$W_{\frac{r}{2\sqrt{3}},y_0}(x^*) \geq C_5 W_{\frac{r}{2\sqrt{3}},y_0}(\hat{y}).$$
Finally, by Lemma \ref{CFMSLemma2.1}, there exists a constant $C_6 = C_6(\lambda,\Lambda,m) > 0$ such that
$$W_{\frac{r}{2\sqrt{3}},y_0}(\hat{y}) \geq C_6.$$
Combining all the bounds above, and recalling that $|x - y| = r$ we conclude that
$$G(x,y) \geq C r^{-(n+1)}d(x)d(y) = C\frac{d(x)d(y)}{|x-y|^{n+1}}.$$

\noindent {\bf Case 2:} $d(y) \leq \frac{r}{2} < d(x)$. \\
Since $|x -y| = r$, it follows that $d(x) \leq |x-y| + d(y)  \leq \frac{3r}{2}$. Let $\hat{x} \in \partial B_r(y) \cap \left\{x_{n+1} = \frac{r}{2} \right\}$ be the point closest to $x$. Then $d(\hat{x}) = \frac{r}{2} \geq \frac{d(x)}{3}$ and $|\hat{x} - y| = r = |x-y|$. Consequently, by Case 1,
$$G(\hat{x},y) \geq C\frac{d(\hat{x})d(y)}{|\hat{x}-y|^{n+1}} \geq \frac{C}{3}\frac{d(x)d(y)}{|x-y|^{n+1}}.$$
On the other hand, by a covering argument and Harnack's inequality, there exists a constant $C_1 = C_1(\lambda, \Lambda, n) > 0$ such that
$$G(x,y) \geq C_1 G(\hat{x},y).$$

\noindent {\bf Case 3:} $\frac{r}{2} < d(y), d(x)$. \\
In this case, 
$$\min\left\{ \frac{d(x)d(y)}{|x-y|^{n+1}}, \frac{1}{4|x-y|^{n-1}} \right\} = \frac{1}{4|x-y|^{n-1}}.$$
Let $p = y + \frac{1}{4}(x-y)$ and $q = y$. Note that $d(p) \geq \frac{r}{2}$ by convexity of the half-space $\left\{x_{n+1} \geq \frac{r}{2} \right\}$. Also, $|p - q| = \frac{r}{4} < \frac{1}{2} d(q)$. Consequently, by Lemma \ref{GWTheorem1.1}, we have
$$G(p,y) = G(p,q) \geq K|p - q|^{1-n} = K4^{n-1} |x-y|^{1-n}.$$
On the other hand, by connecting the points $p$ and $x$ using a Harnack chain using balls of radius $\frac{r}{8}$, and applying Harnack's inequality to the positive solution $G(\cdot,y)$, we conclude that there exists a  positive constant $C_3 = C_3(n,\lambda,\Lambda)$ such that
$$G(x,y) \geq C_3 G(p,y).$$
The estimate \eqref{GlobalGreenEstimate} thus follows.

\end{proof}

In order to address the behavior of $P_f$ in $\real^n\setminus B_R$, for large $R$, we need a variation on the barrier function given in Lemma \ref{lemGF:3.2InGW}.  The difference between the two results is that Lemma \ref{lemGF:3.2InGW} applies to the situation for $r\in(0,1]$, whereas in the Lemma \ref{lemAppendix:BarrierLargeR}, $r>1$.  This is a modification of a well known result about the uniform H\"older continuity of solutions to equations with bounded measurable coefficients in domains with an exterior cone condition, e.g. \cite[Lemma 7.1]{GruterWidman-1982GreenFunUnifEllipticManMath}.

\begin{lemma}\label{lemAppendix:BarrierLargeR}
	There exists constants, $C>0$, $\al\in(0,1]$, and $\ep>0$, depending on the Dini modulus and ellipticity of $A$ and $n$, so that for all $r>1$, and for $v$ as in Lemma \ref{lemGF:3.2InGW}, for all $\abs{X}\leq r+\ep$,
	\begin{align*}
		v(X)\leq C\frac{d(X)}{r^\al}.
	\end{align*}
\end{lemma}

\begin{proof}[Proof of Lemma \ref{lemAppendix:BarrierLargeR}]
	First, we note that in Lemma \ref{lemGF:3.2InGW}, the constant, $C$, to depended upon only the Dini modulus of $A$, ellipticity, and $n$.  The scaling argument used for $r<1$ in Lemma \ref{lemGF:3.2InGW} will not work here because in order to have $A$ given in a ball of radius $r>1$, the result at scale $1$ must be applied to the coefficients $A(rx)$, whose Dini modulus blows up as $r$ is large. 
	
	Thus, instead, we can appeal to results at scale $r=1$ that only depend on ellipticity, and then rescale the equation in $\A_{2r}$ to $\A_2$, which preserves ellipticity, but not the Dini modulus.  This is the reason for the appearance of the factor $r^\al$ for possibly $\al<1$.  To this end, we simply note that for $v_1$ that solves equation \ref{eqAppendix:Lem3.2InGW} with $r=1$, $v_1$ is H\"older continuous for some universal $\al\in(0,1]$ in $\overline \A_2$.  Thus, under rescaling, we see that as $v_1\equiv 0$ on $\partial B_1$, (e.g. \cite[Lemma 1.7]{GruterWidman-1982GreenFunUnifEllipticManMath})
	\begin{align*}
		0\leq v_1(X)\leq Cd(X)^\al.
	\end{align*}
	Under rescaling, back to the case of $v_r$ that solves (\ref{eqAppendix:Lem3.2InGW}) in $\A_{2r}$, we have
	\begin{align*}
		0\leq v_r(X)\leq C\frac{d(X)^\al}{r^\al}.
	\end{align*}
	
	Now, as the domain $\A_{2r}$ enjoys a uniform exterior ball condition of radius $r>1$, we can invoke the Dini property of $A$ to use a barrier for $v$ near the boundary $\partial B_r$.  In particular, we can use a barrier in an outer annulus with inner radius $1$, outer radius $2$ (given in Lemma \ref{lemGF:3.2InGW}), to conclude that 
	\begin{align*}
		v(X)\leq C\frac{d(X)}{r^\al}.
	\end{align*} 
	This, of course follows from the fact that the first estimate established in this proof that for all $X$ with $r<\abs{X}\leq r+1$, $v(X)\leq C\frac{1}{r^\al}$.

\end{proof}

\begin{proof}[Proof of Proposition \ref{propGF:PoissonKernel}]
	First of all, we address the bounds for the case $\abs{X-Y}<R_0$.  As
	\begin{align*} 
		P_f(X,Y)=(\partial_{\nu}G(X,\cdot))(Y), 
	\end{align*}	
	we see that the bounds on $P_f$ are immediate from Theorem \ref{thm:GreenBoundaryBehavior}.

	Now, we focus on the second estimate.
	We may assume, without loss of generality, that $X = X_0 = (0,f(0))$. Notice that 
	$$\int_{\Gam_f\setminus B_{R}(X_0)} P_f(X_0+s\nu(X_0),Y)d\sigma(Y)  = W(X_0 + s\nu(X_0)),$$
	where $W$ solves the equation
	$$
	\begin{cases}
		\Delta W=0\ \text{in}\ D_f,\\
		W=\Indicator_{B_R^c(X_0)}\ \text{on}\ \Gam_f,\\
		W=0\ \text{on}\ \{x_{n+1}=0\}.
	\end{cases}
	$$
	We next flatten the domain $D_f$ by using the transformation $T_f$ defined in \eqref{flatteningtransformation}. The function $\tilde{W} = W\circ T_f^{-1}$ then solves 
	$$
	\begin{cases}
	\div(A(y) \nabla \tilde{W}(y)) = 0\ \text{in}\ \real^n \times [0,L],\\
	\tilde{W}=\Indicator_{B_R^c(0,L)}\ \text{on}\ \left\{y_{n+1} = L \right\},\\ 
	\tilde{W}=0\ \text{on}\ \{y_{n+1}=0\},
	\end{cases}
	$$
	with $A(y) \in \mathbb{R}^{(n+1)\times(n+1)}$ uniformly elliptic and Dini continuous (depending on $\del$, $L$, $m$, $\omega$). Note that $0\leq \tilde W\leq 1$ on $\real^n \times [0,L]$ by the comparison principle.
	
	We now extend the coefficients $A$ to all of $\Rn$ in a Dini continuous fashion with the same modulus of continuity $\omega$, and denote them $\hat{A}$. The corresponding divergence form operator on $\Rn$ will be denoted $\hat{L} := \div(\hat{A}(y) \nabla \cdot)$. Note that $\hat{A}$ can also be taken to satisfy the same ellipticity conditions as $A$. Now suppose $R > \sqrt{3}L$, and let $Y_0 = (0, L + \frac{R}{\sqrt{3}})$. On the annular domain $\mathcal{A}_{\frac{2R}{\sqrt{3}}}(Y_0)$, consider the function $\varphi$ which solves the problem
	$$
	\begin{cases}
	\hat{L}\varphi = 0 \text{ in } \mathcal{A}_{\frac{2R}{\sqrt{3}}}(Y_0), \\
	\varphi = 0 \text{ on } \partial B_{\frac{R}{\sqrt{3}}}(Y_0), \\
	\varphi = 1 \text{ on } \partial B_{\frac{2R}{\sqrt{3}}}(Y_0). \\
	\end{cases}
	$$
	By Lemma \ref{lemAppendix:BarrierLargeR} (we can assume, without loss of generality that $R>1$), there exists constant $K = K(n,\lambda,\Lambda, \omega)$ such that when $R>1$, $\abs{\varphi(X)}\leq C\frac{d(X)}{R^\al}$ for all $X \in \mathcal{A}_{\frac{2R}{\sqrt{3}}}(Y_0)$ with $R<\abs{X}<R+\ep$. Consequently, since $\varphi(0,L) = 0$, we conclude that $\varphi(0, L-s) \leq \frac{Ks}{R^\al}$ for all $s > 0$ sufficiently small.
	
It remains to show that $\tilde{W} \leq \varphi$ on $\Omega_R := \mathcal{A}_{\frac{2R}{\sqrt{3}}}(Y_0) \cap \RN \times [0,L]$. To show this, notice that $\partial \Omega_R$ consists of three pieces; the first two are the flat portions consisting of the intersection of $\mathcal{A}_{\frac{2R}{\sqrt{3}}}(Y_0)$ with $\left\{y_{n+1} = 0 \right\}$ and $\left\{y_{n+1} = L \right\}$ respectively, while the third piece is the intersection of $\partial B_{\frac{2R}{\sqrt{3}}}(Y_0)$ with $\RN \times [0,L]$. On the flat portions, we know $\tilde{W} = 0$ and since $\varphi \geq 0$ by the maximum principle, we see that $\varphi \geq \tilde{W}$ on this portion of $\partial \Omega_R$. On the remaining portion of $\partial \Omega_R$, we know that $\varphi = 1$ and since $\tilde{W} \leq 1$ on $\RN \times [0,L]$, we conclude that $\varphi \geq \tilde{W}$ on this piece of $\partial \Omega_R$ as well. Consequently, by the maximum principle, $\varphi \geq \tilde{W}$ on $\Omega_R$. In particular, $\tilde{W}(0,L-s) \leq \varphi(0, L-s) \leq \frac{Ks}{R^\al}$ for all $s > 0$ sufficiently small. Rewriting this in terms of $W$, we obtain the desired estimate \eqref{eqGF:DecayOnMassOfPoissonKernel}.

\end{proof}

With only a few modifications, we can adapt the proof of Proposition \ref{propGF:PoissonKernel} to also give the proof of Lemma \ref{lemGF:DecayAtInfinityForTechnicalReasons}.

\begin{proof}[Proof of Lemma \ref{lemGF:DecayAtInfinityForTechnicalReasons}]
We note that in this setting, as $D_f$ is a Lipschitz domain, then $P_f$ exists and is an $A^\infty$ weight as in \cite{Dahlberg-1977EstimatesHarmonicMeasARMA}, and by the above results, $P_f$ will be more regular when restricted to $B_{R}$, as in that region, $\Gam(f)$ is $C^{1,\dini}$.  
	
We see that this time, we have
	\begin{align*}
		\int_{\Gam_f\setminus B_{2R}(X)} P_f(X+s\nu(X),Y)d\sigma_f(Y) = W(X+s\nu(X)),
	\end{align*} 
	where $W$ is the unique solution of 
	\begin{align*}
		\begin{cases}
			\Delta W=0\ &\text{in}\ D_f\\
			W=\Indicator_{B_{2R}^c(X)}\ &\text{on}\ \Gam_f\\
			W=0\ &\text{on}\ \{x_{n+1}=0\}.
		\end{cases}
	\end{align*}
	Owing to the fact that $f$ is globally Lipschitz and $C^{1,\dini}_\rho(B_{2R})$, we see that
	after the straightening procedure, $\tilde W$ solves an equation on $\real^n\times[0,L]$, with coefficients, $\hat A$, that have been extended to all of $\real^{n+1}$ and that are Dini continuous in $B_{2R}\times\real$, while they are globally bounded and uniformly elliptic.  We note that we are now concerned with the behavior of $\tilde W$ at $\tilde X - se_{n+1}$, where for $X=(x,f(x))$, $\tilde X=(x,L)$. Thus, for the barrier, $\varphi$, we can now center the annular region at $Y_0=(x, L+\frac{R}{\sqrt 3})$.  As $\tilde X\in B_R\times[0,L]$, it also holds that $\A_R(Y_0)$ is contained in $B_{2R}\times\real$, in which $\hat A$ is Dini continuous.  Thus, Lemma \ref{lemAppendix:BarrierLargeR} is applicable.  The rest of the proof is the same.

\end{proof}

%%%%%%%%%%%%%%%%%%%%%%%%%%%%%%%%%%%%%%%%%%%%%%
%%%%%%%%%%%%%%%%%%%%%%%%%%%%%%%%%%%%%%%%%%%%%%
%%%%%%%%%%%%%%%%%%%%%%%%%%%%%%%%%%%%%%%%%%%%%%
%%%%%%%%%%%%%%%%%%%%%%%%%%%%%%%%%%%%%%%%%%%%%%
%%%%%%%%%%%%%%%%%%%%%%%%%%%%%%%%%%%%%%%%%%%%%%
%%%%%%%%%%%%%%%%%%%%%%%%%%%%%%%%%%%%%%%%%%%%%%
%%%%%%%%%%%%%%%%%%%%%%%%%%%%%%%%%%%%%%%%%%%%%%
%%%%%%%%%%%%%%%%%%%%%%%%%%%%%%%%%%%%%%%%%%%%%%
%%%%%%%%%%%%%%%%%%%%%%%%%%%%%%%%%%%%%%%%%%%%%%
%%%%%%%%%%%%%%%%%%%%%%%%%%%%%%%%%%%%%%%%%%%%%%

\bibliography{refs}
\bibliographystyle{plain}
%%%%%%%%%%%%%%%%%%%%%%%%%%%%%%%%%%%%%%%%%%%%%%
\end{document}